\documentclass[a4paper,12pt
]{article}
\usepackage{amsmath,amssymb,amsthm}
\usepackage{comment}
\usepackage {varwidth}
\usepackage{mathtools}
\numberwithin{equation}{section}
\usepackage[margin=20mm]{geometry}
\newtheorem{thm}{Theorem}[section]
\newtheorem{prop}[thm]{Proposition}

\newtheorem{exam}[thm]{Example}
\newtheorem{rem}[thm]{Remark}
\newtheorem{lem}[thm]{Lemma}

\def\Q{{\mathbb Q}}
\def\R{{\mathbb R}}

\def\d{{\rm d}}

\makeatletter

\title{Spatial asymptotic behaviors of fractional stochastic heat equations driven by additive L\'evy white noise}
\author{Yuichi Shiozawa\thanks{Department of Mathematical Sciences, Faculty of Science and Engineering, Doshisha University,
Kyotanabe, Kyoto, 610-0394, Japan; \texttt{yshiozaw@mail.doshisha.ac.jp}}\qquad
Jian Wang\thanks{School of Mathematics and Statistics  \&
Key Laboratory of Analytical Mathematics and Applications (Ministry of Education) \& Fujian Provincial Key Laboratory
of Statistics and Artificial Intelligence,
Fujian Normal University, Fuzhou, 350007, P.R. China; \texttt{jianwang@fjnu.edu.cn}}}

\begin{document}
\allowdisplaybreaks
\maketitle
\begin{abstract}
We establish explicit integral tests for spatial asymptotic behaviors of fractional stochastic heat equations
driven by additive L\'evy white noise.
Our results indicate that fractional stochastic heat equations enjoy
the so-called additive physical intermittent property in all dimensions
when the driven L\'evy white noise is sufficiently light-tailed.
The proofs are based on heat kernel estimates for the fractional Laplacian
and exact tail behaviors for Poissonian functionals associated with the driven L\'evy white noise.
\end{abstract}



\section{Introduction}
Stochastic partial differential equations (SPDEs) driven by space-time white noise
(i.e., Gaussian noise which has the covariance structure of Brownian motion in space-time),
initiated in \cite{Wa} via the random field approach,
have
been attracted a lot of increasing attentions in the past few decades
(see, e.g., \cite{DK,K0}).
Recently,
as a non-Gaussian space-time white noise counterpart,
there  
have been great  
developments in the study of SPDEs with L\'evy white noise by using the random field approach;
see \cite{Chong, Chong1, Chong1-D} and the references therein.
In this paper, we consider the following (linear) fractional stochastic heat equation
\begin{equation}\label{eq:fractional-she1}
\left\{
\begin{aligned}
\frac{\partial X}{\partial t}(t,x)
&=-(-\Delta)^{\alpha/2}X(t,x)+\dot{\Lambda}(t,x), \quad (t,x)\in (0,\infty)\times \R^d, \\
X(0,x)&=0, \quad x\in \R^d.
\end{aligned}
\right.
\end{equation}
Here, $-(-\Delta)^{\alpha/2}$ with $\alpha\in (0, 2)$ is the fractional Laplacian
which is the infinitesimal generator of a (rotationally) symmetric $\alpha$-stable process on $\R^d$,
and the measure $\Lambda$ is a L\'evy space-time white noise
on ${\cal B}((0,\infty))\otimes {\cal B}(\R^d)$ defined by
\[
\Lambda(\d t\, \d x)
=m\,\d t\,\d x+\int_{(0,1]}z\,(\mu-\nu)(\d t\, \d x\, \d z)
+\int_{(1,\infty)}z\,\mu(\d t\, \d x\, \d z),
\]
where $m\in \R$, and $\mu$ is a Poisson random measure on $(0,\infty)\times \R^d\times (0,\infty)$
corresponding to the L\'evy measure $\nu(\d t\, \d x\, \d z)=\d t \otimes \d x \otimes \lambda({\rm d}z)$
with $\lambda(\d z)$ being a nontrivial
measure on ${\cal B}((0,\infty))$ so that
$\int_{(0,\infty)}(1\wedge |z|^2)\,\lambda(\d z)<\infty$.
A mild solution to \eqref{eq:fractional-she1} is a predictable process $X(t,x)$
satisfying
\begin{equation}\label{mild---}
X(t,x)=\int_{(0,t]\times \R^d}p_{t-s}(x-y)\,\Lambda(\d s \, \d y), \quad (t,x)\in
(0,\infty) 
\times \R^d,
\end{equation}
where $p_t(x,y) 
=p_t(x-y)$ is the heat kernel of the fractional Laplacian
(also called the density function of the transition probability density for
the symmetric $\alpha$-stable process).
Nowadays there are lots of works on the stochastic heat equation
(that is, $\alpha=2$ in \eqref{eq:fractional-she1}) driven by non-Gaussian L\'evy white noises;
see \cite{Chong, Chong1, Chong1-D, M98, S98} and the references therein.
Though SPDEs with the fractional Laplacian play important roles
both in theories and applications (see e.g.\ \cite{Bi,FK,FL,FN,Kim}),
to the best of our
knowledge,
there is no literature focusing
on the fractional stochastic heat equation \eqref{eq:fractional-she1} with L\'evy space-time white noise.

The main aim of this paper is to establish the spatial asymptotic behaviors of $X(t,x)$,
i.e., the almost-sure behaviors of $X(t,x)$ for fixed time $t>0$ as $|x|\to \infty$.
In particular, we will extend the results in \cite{CK22} from the stochastic heat equation to the fractional
one.
It should be noted that such kind of
asymptotic behaviors are closely related to the phenomenon of intermittency in the analysis of random fields,
which refers to the chaotic behavior of a random
field that develops unusually high peaks over small areas.

To highlight the contribution of our paper,
in this section we suppose that the driven L\'evy white noise $ \Lambda$ associates with the L\'evy measure
\begin{equation}\label{e:levym}
\lambda((0,1])=0,\quad \lambda((z,\infty))= z^{-\beta},\,\, z>1
\end{equation}
for some $\beta>d/(d+\alpha)$.
We shall emphasize that, different from
the stochastic heat equation
driven by non-Gaussian L\'evy white noise considered in \cite{CK22,CK23},
the requirement that $\beta>d/(d+\alpha)$ is
optimal to ensure necessary and sufficient conditions for the
almost surely finiteness of the mild solution $X(t,x)$
to the fractional stochastic heat equation \eqref{eq:fractional-she1}
with the L\'evy white noise $ \Lambda$
satisfying \eqref{e:levym};
see Theorem \ref{thm:sol-exist} below for more details.

\begin{thm}\label{thm1}
Let $ \Lambda$ be the L\'evy white noise
such
that its associated L\'evy measure satisfies \eqref{e:levym}. Let $f:(0,\infty)\to (0,\infty)$ be a nondecreasing function.
\begin{itemize}
\item[{\rm (1)}] If $\alpha/d\ge d/(d+\alpha)$, then the following statements hold.
 \begin{itemize}

\item[{\rm (i)}] Suppose that $\beta>\alpha/d$. Then, almost surely,
$$\limsup_{r\to\infty}\frac{\sup_{|x|\le r}X(t,x)}{f(r)}=\infty \quad \text {or} \quad\limsup_{r\to\infty}\frac{\sup_{|x|\le r}X(t,x)}{f(r)}=0$$
according to whether the integral $$\int_1^\infty r^{d-1}f(r)^{-\alpha/d}\,dr$$ diverges or converges.
\item[{\rm (ii)}] Suppose that $\beta\in (\alpha/d,\infty)\setminus \{1+\alpha/d\}$. Then
$$\limsup_{r\to\infty}\frac{\sup_{x\in \mathbb Z^d, |x|\le r}X(t,x)}{f(r)}=\infty
\quad \text {or} \quad\limsup_{r\to\infty}\frac{\sup_{x\in \mathbb Z^d, |x|\le r}X(t,x)}{f(r)}=0$$
according to whether the integral
$$\int_1^\infty r^{d-1}f(r)^{-((1+\alpha/d)\wedge\beta)}\,dr$$
diverges or converges.
\end{itemize}

\item [{\rm (2)}] If $\alpha/d> d/(d+\alpha)$, then, for
any  
$\beta\in (d/(d+\alpha), \alpha/d)$,
almost surely both
$$\limsup_{r\to\infty}\frac{\sup_{|x|\le r}X(t,x)}{f(r)}
=\infty \quad \text {or} \quad\limsup_{r\to\infty}\frac{\sup_{|x|\le r}X(t,x)}{f(r)}=0$$
and
$$\limsup_{r\to\infty}\frac{\sup_{x\in \mathbb Z^d, |x|\le r}X(t,x)}{f(r)}=\infty
\quad \text {or} \quad\limsup_{r\to\infty}\frac{\sup_{x\in \mathbb Z^d, |x|\le r}X(t,x)}{f(r)}=0$$ according to whether the integral
$$\int_1^\infty r^{d-1}f(r)^{-\beta}\,dr$$
diverges or converges.
\end{itemize}
\end{thm}

Theorem \ref{thm1}(1) (resp.\ Theorem \ref{thm1}(2)) corresponds
to \cite[Theorem A and Theorem C(i)]{CK23} (resp.\ \cite[Theorem B(i) and Theorem C(iv)]{CK23})
for the stochastic heat equation (that is, $\alpha=2$ in \eqref{eq:fractional-she1})
driven by non-Gaussian L\'evy white noise.
As already pointed out in \cite{CK22,CK23},
the behaviors of $X(t,x)$ with non-Gaussian L\'evy white noise are completely different
from these with Gaussian white noise.
In particular, the latter processes obey the almost-sure growth rate with finite limit
as i.i.d. Gaussian random variables and their maxima (\cite[see (6.3)]{KKX17}),
while in the non-Gaussian
setting,
the spatial asymptotics of the solution are governed
by an integral test with zero-infinity limit.
The integral tests for the almost-sure growth rate with zero-infinity limit are well known
for path properties of heavy-tailed L\'evy processes, see e.g. \cite{KKW}.

Theorem \ref{thm1} indicates that
the fractional stochastic heat equation \eqref{eq:fractional-she1} with additive L\'evy white noise enjoys
the so-called additive physical intermittent property in all dimensions,
in particular when the L\'evy white noise $\Lambda$ is sufficiently light-tailed.
The readers are referred to \cite{CK20} and the references therein
for the notation and the background of the physical intermittency
for the stochastic heat equations with L\'evy white noise.
To clearly state the different spatial asymptotic behaviors of fractional stochastic heat equations
driven by additive L\'evy white noise on the whole space $\R^d$ and the lattice $ \mathbb Z^d$,
we consider the following example.

\begin{exam}
Let $ \Lambda$ be the L\'evy white noise such
that its associated L\'evy measure satisfies \eqref{e:levym} with $\beta>1+\alpha/d$.
Then, almost surely the following hold:
\begin{itemize}
\item[{\rm (i)}]
if $p>d/\alpha$, then
$$\limsup_{r\to\infty}\frac{\sup_{|x|\le r}X(t,x)}{r^{d^2/\alpha}(\log r)^p}=0;$$
if $0\le p\le d/\alpha$, then
$$
\limsup_{r\to\infty}\frac{\sup_{|x|\le r}X(t,x)}{r^{d^2/\alpha}(\log r)^p}=\infty.
$$
\item[{\rm (ii)}] if $p>d/(d+\alpha)$, then
$$\limsup_{r\to\infty}\frac{\sup_{x\in \mathbb Z^d, |x|\le r}X(t,x)}{r^{d^2/(d+\alpha)}(\log r)^p}=0;$$
if $0\le p\le d/(d+\alpha)$, then
$$\limsup_{r\to\infty}\frac{\sup_{x\in \mathbb Z^d, |x|\le r}X(t,x)}{r^{d^2/(d+\alpha)}(\log r)^p}=\infty.$$
\end{itemize}
\end{exam}

\begin{rem}\rm
The proof of Theorem \ref{thm1}, as well as the proofs of general results in Section \ref{section5},
mainly follow from the arguments in \cite{CK22} for the stochastic heat equation driven by additive L\'evy white noise.
The main difference is due to the expression of the mild solution given by \eqref{mild---},
where the heat kernel $p_t(x,y)$ of the fractional Laplacian now becomes polynomial decaying
instead of the exponential decaying like the Gaussian estimate.
Just because of this difference,
necessary and sufficient
conditions for the almost surely finiteness of the solution $X(t,x)$
to \eqref{eq:fractional-she1} are in contrast to these in \cite{CK22}; that is,
\begin{equation}\label{e:condition1}
\int_{(1,\infty)}z^{d/(d+\alpha)}\,\lambda(\d z)<\infty
\end{equation}
instead of
$$\int_{(1,\infty)}(\log z)^{d/2}\,\lambda(\d z)<\infty$$
in \cite[(1.7)]{CK22}.
This also explains the reason why we require $\beta>d/(d+\alpha)$ in Theorem \ref{thm1}.
Thus, the main difficulties and the novelties of the paper are follows.
\begin{itemize}
\item[{\rm (i)}]
Besides the different decaying property with the Gaussian estimates as mentioned above,
the full expression of the heat kernel $p_t(x,y)$ for the fractional Laplacian
$-(-\Delta)^{\alpha/2}$ 
with all $\alpha\in (0,2)$ is not available.
This will bring a few
difficulties  
in the arguments for considering the asymptotic behaviors of fractional stochastic heat equations, as compared with the paper \cite{CK22}.
In particular, our arguments here make full use of the asymptotic properties of the heat kernel $p_t(x,y)$.
\item[{\rm(ii)}]
For the stochastic heat equation and the fractional one,
the tail behaviors of the mild solutions follow those of the associated L\'evy measures,
but they are different from each other
(see Theorem \ref{lem:tau}, Lemma \ref{lem:eta-express-2} and
\cite[Theorem 2.4 and Lemma 2.2]{CK22}).
This is also the case for the tail behaviors of the local supremum of the mild solutions
(see Lemma \ref{lem:tau}, Proposition \ref{prop:tau-tail}  and \cite[Lemma 3.2 and Theorem 3.3]{CK22}).
\item[{\rm(iii)}]
To relate the tail asymptotics of the mild solution or its local supremum
with that
of the L\'evy measure $\lambda$,
one needs much more effort to overcome the restriction arising from \eqref{e:condition1}.
See the proofs of Lemma \ref{lem:eta-tail} and \cite[Lemma 3.4]{CK22}.
\item[{\rm(iv)}]
In order to get tight integral tests for spatial asymptotic behaviors of fractional stochastic heat equations
driven by additive L\'evy white noise,
we will make full use of the moments of the martingale part
in the decomposition of the solution (see $X_2(t,x)$ in \eqref{eq:x-decomposition}).
This idea enables
us to improve parts of the results and the arguments in \cite[Section 4]{CK22}.
In particular, we
can also
obtain conclusions for fractional stochastic heat equations
driven by additive L\'evy white noise on the whole space $\R^d$ when $\beta=\alpha/d$,
and on the lattice $ \mathbb Z^d$ when $\beta=1+\alpha/d$ in the framework of Theorem \ref{thm1}. See Subsection \ref{S:example} for the details.
\end{itemize}
  \end{rem}

Furthermore, it is easily seen from the arguments of our paper that
the results above still hold for the following nonlinear fractional stochastic heat equation on $\R^d$
driven by L\'evy space-time white noise $\Lambda(\d t,\d x)$ with zero initial condition:
\begin{equation}\label{e:levy1}
\left\{
\begin{aligned}
\frac{\partial X}{\partial t}(t,x)
&=-(-\Delta)^{\alpha/2}X(t,x)+\sigma(X(t,x))\dot{\Lambda}(t,x), \quad (t,x)\in (0,\infty)\times \R^d, \\
X(0,x)&=0, \quad x\in \R^d,
\end{aligned}
\right.
\end{equation}
where $\sigma:\R\to (0,\infty)$ is a Lipschitz continuous function,
and it is bounded away from $0$ and infinity.
We defer to Subsection \ref{subsect:nonlinear} for the validity of the assertion above.
See \cite{CK20} for related works about the almost-sure
long time asymptotics for a fixed spatial point
of the solution to the stochastic heat equation
driven by a L\'evy space-time white noise.
It should be interesting to investigate the spatial asymptotics of the solution for
the fractional stochastic heat equation driven by
multiplicative L\'evy noise (i.e., $\sigma(x)=x$ in \eqref{e:levy1} and $X(0,x)=1$);
see \cite{CK23} for the stochastic heat equation driven by L\'evy space-time white noise and
\cite{BC} for the intermittency of stochastic heat equations with multiplicative L\'evy noise.
In connection with the fractional stochastic heat equation driven by
multiplicative L\'evy noise,
Berger-Lacoin \cite{BL22} refer to possible conditions (\cite[(2.30) and (2.31)]{BL22}),
which are similar to \eqref{e:condition1} and \eqref{eq:small-conv-0},
for the nondegeneracy of the partition function of the long range directed polymer
in L\'evy noise.
See \cite[Subsection 2.5.3 (A)]{BL22} for details.
\ \

The remainder of this paper is
organized
as follows.
In the next section, we establish
necessary and sufficient
conditions for the existence of the almost surely finiteness of the mild solution $X(t,x)$
to the fractional stochastic heat equation \eqref{eq:fractional-she1}.
In particular, we claim that the solution $X(t,x)$ is an infinitely divisible random variable,
and give the explicit expression of the characteristic function for it.
In Section \ref{section3},
we present the tail asymptotics of the solution $X(t,x)$ in terms of its associated L\'evy measure $\eta$.
In particular, we claim that the tail $\overline{\eta}$ of the L\'evy measure $\eta$ is
of extended regular variation at infinity.
Here we also directly relate the tail $\overline{\eta}$ with
the tail $\overline{\lambda}$ of the L\'evy measure $\lambda$.
Section \ref{section4} is devoted to the tail asymptotics of  the spatial supremum $\sup_{x\in A}X(t,x)$
to the solution $X(t,x)$ on some $A\in {\cal B}(\R^d)$.
For this,
we will verify that the solution $X(t,x)$ has a locally bounded and continuous modification.
With the aid of this, we can make a bridge between the spatial supremum $\sup_{x\in A}X(t,x)$
to the solution and the tail behaviors of modified L\'evy measures
$\eta_A$.
With all the conclusions at hand, we then give general results for spatial asymptotic behaviors of fractional stochastic heat equations
driven by additive L\'evy white noise on the whole space $\R^d$ or on the lattice $ \mathbb Z^d$ in Section \ref{section5}.

The notation $f\asymp g$ means that there is a constant $c\ge1$ such that
$c^{-1}f\le g\le c f$, and  $f\sim g$ means that $\lim_{r\to \infty}f(r)/g(r)=1$.
Furthermore, $f\preceq g$ (resp.\ $f\succeq g$) means that
there is a constant $c>0$ such that $f\le c g$ (resp.\ $f\ge cg$).

\section{Existence of the mild solution}
We first formulate the L\'evy space-time white noise.
Let $\lambda(\d z)$ be a
nontrivial
measure on ${\cal B}((0,\infty))$ such that
$\int_{(0,\infty)}(1\wedge |z|^2)\,\lambda(\d z)<\infty$.
Let $\nu$ be a measure on ${\cal B}((0,\infty))\otimes {\cal B}(\R^d) \otimes {\cal B}((0,\infty))$ defined by
$\nu(\d t\, \d x\, \d z)=\d t \otimes \d x \otimes \lambda({\rm d}z)$.
Let $\mu$ denote the Poisson random measure
on ${\cal B}((0,\infty))\otimes {\cal B}(\R^d) \otimes {\cal B}((0,\infty))$
with intensity measure $\nu$; that is,
\[
P(\mu(A)=n)=e^{-\nu(A)}\frac{\nu(A)^n}{n!},
\quad n=0,1,2,\dots, \ A\in {\cal B}((0,\infty))\otimes {\cal B}(\R^d)\otimes {\cal B}((0,\infty)).
\]
Define the L\'evy space-time white noise as a measure $\Lambda$
on ${\cal B}((0,\infty))\otimes {\cal B}(\R^d)$:
\[
\Lambda(\d t\, \d x)
=m\,\d t\,\d x+\int_{(0,1]}z\,(\mu-\nu)(\d t\, \d x\, \d z)
+\int_{(1,\infty)}z\,\mu(\d t\, \d x\, \d z),
\]
where $m\in \R^d$.

For $\alpha\in(0,2)$,
let $p_t(x,y)$ be the density function of the transition probability density
for the
(rotationally)
symmetric $\alpha$-stable process on $\R^d$ generated by $-(-\Delta)^{\alpha/2}$.
Then there exists a strictly decreasing smooth function $g:[0,\infty)\to (0,\infty)$ such that
\[
p_t(x,y)=\frac{1}{t^{d/\alpha}}g\left(\frac{|x-y|}{t^{1/\alpha}}\right), \quad t>0, \ x,y\in \R^d
\]
and for some $c_{d,\alpha}>0$,
\begin{equation}\label{eq:g-asymp}
g(r)\sim \frac{c_{d,\alpha}}{r^{d+\alpha}}, \quad r\rightarrow\infty;
\end{equation}
see \cite[Theorem 2.1]{BG60} and \cite[Proof of Lemma 5]{BJ07}.

We can formally define
the fractional stochastic heat equation with zero initial value condition
as follows:
\begin{equation}\label{eq:fractional-she}
\left\{
\begin{aligned}
\frac{\partial X}{\partial t}(t,x)
&=-(-\Delta)^{\alpha/2}X(t,x)+\dot{\Lambda}(t,x), \quad (t,x)\in (0,\infty)\times \R^d, \\
X(0,x)&=0, \quad x\in \R^d.
\end{aligned}
\right.
\end{equation}
The mild solution of \eqref{eq:fractional-she} is defined by
\begin{equation}\label{eq:mild}
X(t,x)=\int_{(0,t]\times \R^d}p_{t-s}(x-y)\,\Lambda(\d s \, \d y).
\end{equation}
When
\begin{equation}\label{eq:first-moment}
\int_{(0,1]}z\,\lambda({\rm d}z)<\infty,
\end{equation}
we
will consider the following
non-compensated version of $X(t,x)$:
\begin{equation}\label{eq:decomp-1}
X(t,x)=m_0t+\int_{(0,t]
\times \R^d\times (0,\infty)
}p_{t-s}(x-y)z\,\mu(\d s \, \d y\,\d z)
\end{equation}
with
$m_0=m-\int_{(0,1]}z\,\lambda(\d z)$.

We next present  necessary and sufficient conditions
for the existence of $X(t,x)$.
Let $\overline{\lambda}(r)=\lambda((r,\infty))$ for $r>0$.
\begin{thm}\label{thm:sol-exist}
For each $(t,x)\in (0,\infty)\times \R^d$,
the mild solution $X(t,x)$ given by \eqref{eq:mild} exists as a finite value $P$-a.s.\ if and only if
\begin{equation}\label{eq:small-conv-0}
\int_{(0,1]}z^{(1+\alpha/d)\wedge2}|\log z|^{{\bf 1}_{\{d=\alpha\}}}\,\lambda({\rm d}z)<\infty
\end{equation}
and
\begin{equation}\label{eq:big-conv}
\int_{(1,\infty)}z^{d/(d+\alpha)}\,\lambda(\d z)<\infty.
\end{equation}
In this case, for any $\theta\in \R$,
\begin{align*}
&E\left[\exp\left(i\theta X(t,x)\right)\right]
=\exp\left(i\theta a
+\int_{(0,\infty)}(e^{i\theta u}-1-i\theta(u\wedge 1))\,\eta(\d u)\right),
\end{align*}
where
\begin{align*}
a=&
(m-\overline{\lambda}(1))t\\
&-\int_{(0,t]\times \R^d\times (0,1]}
p_s(y)z{\bf 1}_{\{p_s(y)z>1\}}\,\d s\,\d y\, \lambda(\d z)
+\int_{(0,t]\times \R^d\times (1,\infty)}
\left\{1\wedge (p_s(y)z)\right\}
\,\d s\,\d y\, \lambda(\d z)
\end{align*}
and
the measure $\eta$ is defined by
\begin{equation}\label{eq:eta}
\eta(B)=\nu(\left\{(s,y,z)\in (0,t]\times \R^d\times (0,\infty)
: p_s(y)z\in B\right\}), \quad B\in {\cal B}((0,\infty)).
\end{equation}

In particular, if \eqref{eq:first-moment} is
satisfied, then
$X(t,x)$ given by \eqref{eq:decomp-1} exists as a finite value $P$-a.s.\ if and only if
\eqref{eq:big-conv} holds. In this case, for any $\theta\in \R$,
\begin{align*}
&E\left[\exp\left(i\theta X(t,x)\right)\right]
=\exp\left(i\theta m_0t
+\int_{(0,\infty)}(e^{i\theta u}-1)\,\eta(\d u)\right).
\end{align*}
\end{thm}

\begin{proof}
We prove Theorem \ref{thm:sol-exist} by following the proof of \cite[Theorem 1.1]{CK22}.

(1) We first discuss the existence of the non-compensated version $X(t,x)$ given by \eqref{eq:decomp-1}.
Let \eqref{eq:first-moment} hold.
Then by \cite[p.\ 43, Theorem 2.7 (i)]{K14},
the integral
\[
\int_{(0,t]\times \R^d\times (0,\infty)}p_{t-s}(x-y)z\,\mu(\d s \, \d y\, \d z)
\]
exists as a finite value  $P$-a.s.\ if and only if
$$
\int_{(0,t]\times \R^d\times (0,\infty)}
\left\{1\wedge (p_{t-s}(x-y)z)\right\}
\, \d s\, \d y\, \lambda(\d z)<\infty,
$$
which is equivalent to saying that
\begin{equation}\label{eq:comp-conv-1}
\int_{(0,t]\times \R^d\times (0,\infty)}
\left\{1\wedge (p_s(y)z)\right\}
\, \d s\, \d y\, \lambda(\d z)<\infty.
\end{equation}

To verify
\eqref{eq:comp-conv-1},
we reveal the condition for $p_s(y)z\le 1$.
Let $M=g(0)$, and let $g^{-1}$ be the inverse function of $g$.
For $s>0$ and $z>0$, define
\[
H_1(z)=\left(Mz\right)^{\alpha/d}, \quad H_2(s,z)=s^{1/\alpha}g^{-1}\left(\frac{s^{d/\alpha}}{z}\right).
\]
Then $p_s(y)z\le 1$ if and only if
$s>H_1(z)$, or
$s\le H_1(z)$ and $|y|\ge H_2(s,z)$.
In particular, if we let $D=t^{d/\alpha}/M$,
then $H_1(z)\le t$ if and only if $z\le D$.
Hence, $s\le t\wedge H_1(z)$ if and only if
$z\le D$ and $s\le H_1(z)$, or $z>D$ and $s\le t$.
To summarize, we see that $p_s(y)z\le 1$ if and only if
either
of the next conditions is satisfied:
\begin{itemize}
\item $H_1(z)<s\le t$;
\item $z\le D$, $s\le H_1(z)$  and $|y|\ge H_2(s,z)$;
\item $z>D$, $s\le t$, and $|y|\ge H_2(s,z)$.
\end{itemize}

We define
\begin{align*}
A_1&=\left\{(s,y,z)\in (0,t]\times \R^d\times (0,\infty) :
z\le D, \, s\le H_1(z),\,  |y|<H_2(s,z) \right\},\\
A_2&=\left\{(s,y,z)\in (0,t]\times \R^d\times (0,\infty) :
z>D, \, |y|<H_2(s,z) \right\}.
\end{align*}
We also define
\begin{align*}
B_0&=\left\{(s,y,z)\in (0,t]\times \R^d\times (0,\infty) : H_1(z)<s\right\},\\
B_1&=\left\{(s,y,z)\in (0,t]\times \R^d\times (0,\infty) :
z\le D, \, s\le H_1(z), \, |y|\ge H_2(s,z) \right\},\\
B_2&=\left\{(s,y,z)\in (0,t]\times \R^d\times (0,\infty) :
z> D, \, |y|\ge H_2(s,z) \right\}.
\end{align*}
Then the sets $A_1$, $A_2$, $B_0$, $B_1$ and $B_2$ form a partition of
$(0,t]\times \R^d\times (0,\infty)$ and
\[
p_s(y)>\frac{1}{z}\iff (s,y,z)\in A_1\cup A_2.
\]

Let $\omega_d$ be the surface area of the unit ball in $\R^d$.
Then
\[
\int_{A_1}
\left\{1\wedge (p_s(y)z)\right\}
\,\d s\,\d y\,\lambda(\d z)
=\int_{A_1}\,\d s\,\d y\,\lambda(\d z)
=\omega_d\int_{(0,D]}\left(\int_0^{H_1(z)}H_2(s,z)^d\, \d s\right)\,\lambda(\d z).
\]
By the change of variables formula with $s=H_1(z)u^{\alpha/d} =(Mzu)^{\alpha/d}$,
we have
\[
\int_0^{H_1(z)}H_2(s,z)^d\, \d s
=\frac{\alpha M^{1+\alpha/d}}{d}z^{1+\alpha/d}\int_0^1g^{-1}(Mu)^du^{\alpha/d}\,\d u,
\]
and thus
\[
\int_{A_1}
\left\{1\wedge (p_s(y)z)\right\}
\,\d s\,\d y\,\lambda(\d z)
=\frac{\alpha \omega_dM^{1+\alpha/d}}{d}
\left(\int_0^1g^{-1}(Mu)^d u^{\alpha/d}\,\d u\right)\int_{(0,D]}z^{1+\alpha/d}\,\lambda(\d z).
\]
The first integral of the last expression above is convergent
because \eqref{eq:g-asymp} yields
\begin{equation}\label{eq:g^{-1}-asymp}
g^{-1}(r)\sim \frac{(c_{d,\alpha})^{1/(d+\alpha)}}{r^{1/(d+\alpha)}},
\quad r\rightarrow  
0.
\end{equation}
Therefore,
\begin{equation}\label{eq:a_1-equiv}
\int_{A_1}
\left\{1\wedge (p_s(y)z)\right\}
\,\d s\,\d y\,\lambda(\d z)<\infty
\iff \int_{(0,1]}z^{1+\alpha/d}\,\lambda(\d z)<\infty.
\end{equation}

In the same way as above,
\[
\int_{A_2}
\left\{1\wedge (p_s(y)z)\right\}
\,\d s\,\d y\,\lambda(\d z)
=\int_{A_2}\,\d s\,\d y\,\lambda(\d z)
=\omega_d\int_{(D,\infty)}\left(\int_0^t H_2(s,z)^d\, \d s\right)\,\lambda(\d z).
\]
Then,
 again
 by the change of variables formula with
$s=H_1(z)u^{\alpha/d}  = (Mu)^{\alpha/d} $,
we have
\[
\int_0^t H_2(s,z)^d\, \d s
=\frac{\alpha M^{1+\alpha/d}}{d}z^{1+\alpha/d}
\int_0^{t^{d/\alpha}/(Mz)}g^{-1}(Mu)^d u^{\alpha/d}\,\d u
,
\]
and so
\begin{align*}
&\int_{(D,\infty)}\left(\int_0^t H_2(s,z)^d\, \d s\right)\,\lambda(\d z)\\
&=\frac{\alpha \omega_d}{d}M^{1+\alpha/d}
\int_{(D,\infty)}
\left(\int_0^{t^{d/\alpha}/(Mz)}g^{-1}(Mu)^du^{\alpha/d}\,\d u\right)z^{1+\alpha/d}\,\lambda(\d z).
\end{align*}
As $D=t^{d/\alpha}/M$,
we see that $z\in (D,2D]$ if and only if $1/2\le t^{d/\alpha}/(Mz)<1$,
which yields
\begin{align*}
&\int_{(D,2D]}\left(\int_0^t H_2(s,z)^d\, \d s\right)\,\lambda(\d z)\\
&=
\frac{\alpha \omega_d}{d}M^{1+\alpha/d}
\int_{(D,2D])}
\left(\int_0^{t^{d/\alpha}/(Mz)}g^{-1}(Mu)^du^{\alpha/d}\,\d u\right)z^{1+\alpha/d}\,\lambda(\d z)\\
&\asymp \lambda((D,2D]).
\end{align*}
We also see that $z>2D$ if and only if $t^{d/\alpha}/(Mz)<1/2$.
Combining this with \eqref{eq:g^{-1}-asymp},
we get
\begin{align*}
&\int_{(2D,\infty)}
\left(\int_0^{t^{d/\alpha}/(Mz)}g^{-1}(Mu)^du^{\alpha/d}\,\d u\right)
z^{1+\alpha/d}\,\lambda(\d z)\\
&\asymp \int_{(2D,\infty)}
\left(\int_0^{t^{d/\alpha}/(Mz)}u^{(\alpha/d)-d/(d+\alpha)}\,\d u\right)
z^{1+\alpha/d}\,\lambda(\d z)
\asymp \int_{(2D,\infty)}z^{d/(d+\alpha)}\, \lambda(\d z).
\end{align*}
As a result, we obtain
\begin{equation}\label{eq:a_2-equiv}
\int_{A_2}
\left\{1\wedge (p_s(y)z)\right\}
\,\d s\,\d y\,\lambda(\d z)<\infty
\iff \int_{(1,\infty)}z^{d/(d+\alpha)}\, \lambda(\d z)<\infty.
\end{equation}

Furthermore, since $\displaystyle\int_{\R^d}p_s(y)\,\d y=1$, we have
\begin{align*}
\int_{B_0}
\left\{1\wedge (p_s(y)z)\right\}
\,\d s\, \,d y \, \lambda(\d z)
=\int_{B_0}p_s(y)z\,\d s\, dy \, \lambda(\d z)
=\int_{(0,D]}\left(t-H_1(z)\right)z\,\lambda(\d z).
\end{align*}
Then
\[
\int_{(0,D]}(t-H_1(z))z\,\lambda(\d z)\le t\int_{(0,D]}z\,\lambda(\d z)
\]
and
\[
\int_{(0,D]}(t-H_1(z))z\,\lambda(\d z)
\ge \int_{(0,D/2^{d/\alpha}]}(t-H_1(z))z\,\lambda(\d z)
\ge \frac{t}{2}\int_{(0,D/2^{d/\alpha}]}z\,\lambda(\d z).
\]
At the last inequality above, we used the fact that
$0<z\le D/2^{d/\alpha}$
 if and only if
$0<H_1(z)\le t/2$.
Thus
\begin{equation}\label{eq:b_0-equiv}
\int_{B_0}
\left\{1\wedge (p_s(y)z)\right\}
\,\d s\,\d y\,\lambda(\d z)<\infty
\iff \int_{(0,1]}z\,\lambda(\d z)<\infty.
\end{equation}

By definition,
$$\int_{B_1}
\left\{1\wedge (p_s(y)z)\right\}
\,\d s\, \d y\,\lambda(\d z)
=\int_{(0,D]}\left(\int_0^{(Mz)^{\alpha/d}}
\int_{|y|\ge H_2(s,z)}p_s(y)\,\d y\,\d s\right)z\,\lambda(\d z).
$$
If $0<z\le D$ and $0\le s\le (Mz)^{\alpha/d}$,
then by the polar coordinate transform and
change of variables formula with $r=s^{1/\alpha}u$,
\begin{align*}
\int_{|y|\ge H_2(s,z)}p_s(y)\,\d y
&=\int_{|y|\ge H_2(s,z)}\frac{1}{s^{d/\alpha}}g\left(\frac{|y|}{s^{1/\alpha}}\right)\,\d y
=\frac{\omega_d}{s^{d/\alpha}}
\int_{H_2(s,z)}^{\infty}g\left(\frac{r}{s^{1/\alpha}}\right)r^{d-1}\,\d r\\
&=\omega_d\int_{H_2(s,z)/s^{1/\alpha}}^{\infty}g(u)u^{d-1}\,\d u
=\omega_d
\int_{g^{-1}(s^{d/\alpha}/z)}^{\infty}g(u)u^{d-1}\,\d u.
\end{align*}
Hence by the Fubini theorem,
\begin{equation}\label{eq:b_1-fubini}
\begin{split}
&\int_0^{(Mz)^{\alpha/d}}
\left(\int_{|y|\ge H_2(s,z)}p_s(y)\,\d y\right)\,\d s
=\omega_d\int_0^{(Mz)^{\alpha/d}}
\left(\int_{g^{-1}(s^{d/\alpha}/z)}^{\infty}g(u)u^{d-1}\,\d u\right)\,{\rm d}s\\
&=\omega_d\int_0^{\infty}
\left(\int_{(zg(u))^{\alpha/d}}^{(Mz)^{\alpha/d}}\,\d s\right)g(u)u^{d-1}
\,\d u
=\omega_d z^{\alpha/d}\int_0^{\infty}(M^{\alpha/d}-g(u)^{\alpha/d})g(u)u^{d-1}\,\d u.
\end{split}
\end{equation}
The last integral above is convergent by \eqref{eq:g-asymp}.
Combining both conclusions above yields that
\[
\int_{B_1}
\left\{1\wedge (p_s(y)z)\right\}
\,\d s\, \d y\,\lambda(\d z)
=\omega_d\int_{(0,D]}z^{1+\alpha/d}\,\lambda(\d z)
\left(\int_0^{\infty}(M^{\alpha/d}-g(u)^{\alpha/d})g(u)u^{d-1}\,\d u\right),
\]
and so
\begin{equation}\label{eq:b_1-equiv}
\int_{B_1}
\left\{1\wedge (p_s(y)z)\right\}
\,\d s\, \d y\,\lambda(\d z)<\infty
\iff \int_{(0,1]}z^{1+\alpha/d}\,\lambda(\d z)<\infty.
\end{equation}

By definition,
\[
\int_{B_2}
\left\{1\wedge (p_s(y)z)\right\}
\,\d s\,\d y\, \lambda(\d z)
=\int_{(D,\infty)}\left(\int_0^t
\int_{|y|\ge H_2(s,z)}p_s(y)\,\d y\, \d s\right)z\,\lambda(\d z).
\]
Then as in \eqref{eq:b_1-fubini}, we have
\[
\int_0^t
\int_{|y|\ge H_2(s,z)}p_s(y)\,\d y\, \d s
=\omega_d\int_{g^{-1}(t^{d/\alpha}/z)}^{\infty}\left(t-(zg(u))^{\alpha/d}\right)g(u)u^{d-1}\,\d u
,
\]
so that
\[
\int_{B_2}
\left\{1\wedge (p_s(y)z)\right\}
\,\d s\,\d y\, \lambda(\d z)
=\omega_d \int_{(D,\infty)}
\left(\int_{g^{-1}(t^{d/\alpha}/z)}^{\infty}\left(t-(zg(u))^{\alpha/d}\right)g(u)u^{d-1}\,\d u\right)
z\,\lambda(\d z).
\]
Note that
\[
\int_{g^{-1}(t^{d/\alpha}/z)}^{\infty}\left(t-(zg(u))^{\alpha/d}\right)g(u)u^{d-1}\,\d u
\le t\int_{g^{-1}(t^{d/\alpha}/z)}^{\infty}g(u)u^{d-1}\,\d u
\]
and
\begin{align*}
\int_{g^{-1}(t^{d/\alpha}/z)}^{\infty}\left(t-(zg(u))^{\alpha/d}\right)g(u)u^{d-1}\,\d u
&\ge \int_{g^{-1}(t^{d/\alpha}/(2^{d/\alpha}z))}^{\infty}\left(t-(zg(u))^{\alpha/d}\right)g(u)u^{d-1}\,\d u\\
&\ge \frac{t}{2}\int_{g^{-1}(t^{d/\alpha}/(2^{d/\alpha}z))}^{\infty}g(u)u^{d-1}\,\d u.
\end{align*}
By \eqref{eq:g-asymp},
we also have for any $c\ge 1$ and $z>D$,
\[
\int_{g^{-1}(t^{d/\alpha}/(cz))}^{\infty}g(u)u^{d-1}\,\d u
\asymp \int_{g^{-1}(t^{d/\alpha}/(cz))}^{\infty}\frac{u^{d-1}}{u^{d+\alpha}}\,\d u
=\frac{1}{\alpha}\left(g^{-1}\left(\frac{t^{d/\alpha}}{cz}\right)\right)^{-\alpha}
\asymp
\frac{1}{z^{\alpha/(d+\alpha)}}.
\]
Therefore, if $z>D$, then
\[
\int_{g^{-1}(t^{d/\alpha}/z)}^{\infty}g(u)u^{d-1}\left(t-(zg(u))^{\alpha/d}\right)\,\d u
\asymp
\frac{1}{z^{\alpha/(d+\alpha)}}.
\]
This implies that
\begin{align*}
&\int_{(
D
,\infty)}
\left(\int_{g^{-1}(t^{d/\alpha}/z)}^{\infty}g(u)u^{d-1}\left(t-(zg(u))^{\alpha/d}\right)\,\d u\right)
z\,\lambda(\d z)\\
&\asymp
\int_{(
D
,\infty)}\frac{1}{z^{\alpha/(d+\alpha)}}
z\,\lambda(\d z)
=\int_{(
D
,\infty)}z^{d/(d+\alpha)}\,\lambda(\d z)
\end{align*}
and thus
\begin{equation}\label{eq:b_2-equiv}
\int_{B_2}
\left\{1\wedge (p_s(y)z)\right\}
\,\d s\, \d y\, \lambda(\d z)<\infty
\iff \int_{(1,\infty)}z^{d/(d+\alpha)}\,\lambda(\d z)<\infty.
\end{equation}

By \eqref{eq:a_1-equiv}, \eqref{eq:a_2-equiv}, \eqref{eq:b_0-equiv},
\eqref{eq:b_1-equiv} and \eqref{eq:b_2-equiv},
we see that, under \eqref{eq:first-moment},
\eqref{eq:comp-conv-1} holds if and only if
\eqref{eq:big-conv} holds.
Moreover, under this condition, it follows by
\cite[p.\ 43, Theorem 2.7 (i)]{K14}
that for any $\theta\in \R$,
\begin{align*}
&E[\exp\left(i\theta
X(t,x)
)\right)]=
E\left[\exp\left(i\theta\int_{(0,t]\times \R^d\times (0,\infty)}p_{t-s}(x-y)z\,\mu(\d s \, \d y \, \d z)\right)\right]\\
&=\exp\left(
\int_{(0,t]\times \R^d\times (0,\infty)}
\left(\exp\left(i\theta p_{t-s}(x-y)z\right)-1\right)\, \d s\, \d y \, \lambda(\d z)\right)\\
&=\exp\left(
\int_{(0,t]\times \R^d\times (0,\infty)}
\left(\exp\left(i\theta p_s(y)z\right)-1\right)\, \d s\, \d y \, \lambda(\d z)\right)
=\exp\left(
\int_{(0,\infty)}
\left(e^{i\theta u}-1\right)\,\eta(\d u)\right),
\end{align*}
where the measure $\eta$ is defined by \eqref{eq:eta}.

(2) We next show
the existence of the compensated version $X(t,x)$ given by \eqref{eq:mild}.
By definition,
\begin{align*}
&X(t,x)\\
&=mt+\int_{(0,t]\times \R^d\times (0,1]}p_{t-s}(x-y)z\,(\mu-\nu)(\d s \, \d y\, \d z)
+\int_{(0,t]\times \R^d\times (1,\infty)}p_{t-s}(x-y)z\,\mu(\d s \, \d y\, \d z)\\
&=mt+X_1(t,x)+X_2(t,x).
\end{align*}
Then by the argument as
in (1),
$X_2(t,x)$ is convergent if and only if \eqref{eq:big-conv} holds.
According to \cite[Theorem 2.7]{RR89},
$X_1(t,x)$ is convergent if and only if
$$\int_{(0,t]\times \R^d\times (0,1]}p_{s}(y)z{\bf 1}_{\{p_s(y)z>1\}}\, \d s\, \d y\, \lambda(\d z)<\infty$$
and
$$\int_{(0,t]\times \R^d\times (0,1]}
(p_{s}(y)z)^2
\, \d s\, \d y\, \lambda(\d z)<\infty.$$

First, we have
\begin{align*}
&\int_{(0,t]\times \R^d\times (0,1]}p_{s}(y)z{\bf 1}_{\{p_s(y)z>1\}}\, \d s\, \d y\, \lambda(\d z)\\
&=\int_{(0,1]}\left\{\int_0^{t\wedge H_1(z)}
\left(\int_{|y|<H_2(s,z)}p_s(y)\,\d y\right)\,\d s\right\}z\,\lambda(\d z)\\
&=\omega_d\int_{(0,1]}\left\{\int_0^{t\wedge (Mz)^{\alpha/d}}\frac{1}{s^{d/\alpha}}
\left(\int_0^{g^{-1}(s^{d/\alpha}/z)s^{1/\alpha}}
g\left(\frac{r}{s^{1/\alpha}}\right)r^{d-1}\,\d r\right)
\,\d s\right\}z\,\lambda(\d z)\\
&=\omega_d\int_{(0,1]}\left\{\int_0^{t\wedge (Mz)^{\alpha/d}}
\left(\int_0^{g^{-1}(s^{d/\alpha}/z)}
g(u)u^{d-1}\,\d u\right)\,\d s\right\} z\,\lambda(\d z).
\end{align*}
At the last equality above,
we used the change of variables formula with $r=s^{1/\alpha}u$.
Then the Fubini theorem yields  for $z\in (0,1]$,
\begin{align*}
\int_0^{t\wedge (Mz)^{\alpha/d}}
\left(\int_0^{g^{-1}(s^{d/\alpha}/z)}g(u)u^{d-1}\,\d u\right)\,\d s
&=\int_0^{\infty}\left(\int_0^{t\wedge (zg(u))^{\alpha/d})}\,\d s\right)g(u)u^{d-1}\,\d u\\
&=\int_0^{\infty}\left(t\wedge (zg(u))^{\alpha/d}\right)g(u)u^{d-1}\,\d u.
\end{align*}
Hence
\begin{align*}
&\int_{(0,1]}\left\{\int_0^{t\wedge (Mz)^{\alpha/d}}
\left(\int_0^{g^{-1}(s^{d/\alpha}/z)}
g(u)u^{d-1}\,\d u\right) \,\d s\right\}z\,\lambda(\d z)\\
&=\int_{(0,1]}\left(
\int_0^{\infty}\left(t\wedge (zg(u))^{\alpha/d}\right)g(u)u^{d-1}\,\d u\right)z\,\lambda(\d z)\\
&=\int_{(0,1\wedge D]}
\left(
\int_0^{\infty}(zg(u))^{\alpha/d}g(u)u^{d-1}\,\d u
\right)
z\,\lambda(\d z)
+t \int_{(1\wedge D,1]}\left(
\int_0^{g^{-1}(t^{d/\alpha}/z)}g(u)u^{d-1}\,\d u\right)z\,\lambda(\d z)\\
&\quad+\int_{(1\wedge D,1]}\left(
\int_{g^{-1}(t^{d/\alpha}/z)}^{\infty}g(u)u^{d-1}(zg(u))^{\alpha/d}\,\d u\right)z\,\lambda(\d z)\\
&=\int_0^{\infty}g(u)^{1+\alpha/d}u^{d-1}\,\d u
\int_{(0,1\wedge D]}z^{1+\alpha/d}\,\lambda(\d z)
+t \int_{(1\wedge D,1]}
\left(
\int_0^{g^{-1}(t^{d/\alpha}/z)}g(u)u^{d-1}\,\d u\right)
z\,\lambda(\d z)\\
&\quad+\int_{(1\wedge D,1]}
\left(
\int_{g^{-1}(t^{d/\alpha}/z)}^{\infty}g(u)^{1+\alpha/d}u^{d-1}\,\d u\right)
z^{1+\alpha/d}\,\lambda(\d z)\\
&={\rm (I)}_1+{\rm (I)}_2+{\rm (I)}_3.
\end{align*}
Since $\int_0^{\infty}g(u)^{1+\alpha/d}u^{d-1}\,\d u$ is convergent by \eqref{eq:g-asymp},
${\rm (I)}_2$ and ${\rm (I)}_3$ are convergent.
Combining all the conclusions above, we get
\begin{equation}\label{eq:pz>1-2}
\begin{split}
\int_{(0,t]\times \R^d\times (0,1]}p_{s}(y)z{\bf 1}_{\{p_s(y)z>1\}}\, \d s\, \d y\, \lambda(\d z)<\infty
&\iff {\rm (I)}_1<\infty\\
&\iff \int_{(0,1]}z^{1+\alpha/d}\,\lambda(\d z)<\infty.
\end{split}
\end{equation}

Let
\begin{align*}
&\int_{(0,t]\times \R^d\times (0,1]}(p_{s}(y)z)^2{\bf 1}_{\{p_s(y)z\le 1\}}\, \d s\, \d y\, \lambda(\d z)\\
&=\int_{B_0\cap((0,t]\times \R^d\times (0,1])}
(p_{s}(y)z)^2{\bf 1}_{\{p_s(y)z\le 1\}}\, \d s\, \d y\, \lambda(\d z)\\
&\quad+\int_{B_1\cap((0,t]\times \R^d\times (0,1])}
(p_s(y)z)^2{\bf 1}_{\{p_s(y)z\le 1\}}\, \d s\, \d y\, \lambda(\d z)\\
&\quad +\int_{B_2\cap((0,t]\times \R^d\times (0,1])}
(p_s(y)z)^2{\bf 1}_{\{p_s(y)z\le 1\}}\, \d s\, \d y\, \lambda(\d z)\\
&={\rm (II)}_1+{\rm (II)}_2+{\rm (II)}_3.
\end{align*}
Then, by the change of variables formula with $r=s^{1/\alpha}u$,
\begin{align*}
{\rm (II)}_1
&=\int_{B_0\cap((0,t]\times \R^d\times (0,1])}
\frac{z^2}{s^{2d/\alpha}}g\left(\frac{|y|}{s^{1/\alpha}}\right)^2 \,\d s\,\d y \, \lambda(\d z)\\
&=\omega_d\int_{0<z\le D\wedge1, \, H_1(z)<s\le t}
\frac{z^2}{s^{2d/\alpha}}\left(\int_0^{\infty}g\left(\frac{r}{s^{1/\alpha}}\right)^2r^{d-1}\,\d r\right)
\,\d s\lambda(\d z)\\
&=\omega_d\int_0^{\infty}g(u)^2u^{d-1}\,\d u
\int_{0<z\le D\wedge1, \, H_1(z)<s\le t}\frac{z^2}{s^{d/\alpha}}\,\d s \lambda(\d z)\\
&=\omega_d \int_0^{\infty}g(u)^2u^{d-1}\,\d u
\int_{(0,D\wedge1]}
\left(\int_{(Mz)^{\alpha/d}}^t\frac{1}{s^{d/\alpha}}\,\d s\right)
z^2\, \lambda(\d z)\\
&=
\begin{dcases}
\frac{\omega_d}{1-d/\alpha}\int_0^{\infty}g(u)^2u^{d-1}\,\d u
\int_{(0,D\wedge1]}(t^{1-d/\alpha}-(Mz)^{{\alpha/d}-1})z^2\,\lambda(\d z), & d\ne\alpha,\\
\omega_d \int_0^{\infty}g(u)^2u^{d-1}\,\d u
\int_{(0,D\wedge1]}z^2 \log\left(\frac{t}{(Mz)^{\alpha/d}}\right)\,\lambda(\d z), & d=\alpha.
\end{dcases}
\end{align*}
Therefore,
\[
{\rm (II)}_1<\infty \iff
\begin{dcases}
\int_{(0,1]}z^2|\log z|^{{\bf 1}_{\{d=\alpha\}}}\,\lambda(\d z)<\infty, & d\le \alpha,\\
\int_{(0,1]}z^{1+\alpha/d}\,\lambda(\d z)<\infty, & d>\alpha.
\end{dcases}
\]

If $0<z\le D\wedge1$,
then by the change of variables formula with $r=us^{1/\alpha}$ and the Fubini theorem,
\begin{align*}
&\int_0^{(Mz)^{\alpha/d}}
\frac{1}{s^{2d/\alpha}}\left(\int_{g^{-1}(s^{d/\alpha}/z)s^{1/\alpha}}^{\infty}g\left(\frac{r}{s^{1/\alpha}}\right)^2
r^{d-1}\,\d u\right)\,\d s\\
&=\int_0^{(Mz)^{\alpha/d}}
\frac{1}{s^{d/\alpha}}\left(\int_{g^{-1}(s^{d/\alpha}/z)}^{\infty}g(u)^2u^{d-1}\,\d u\right)\,\d s
=\int_0^{\infty}g(u)^2
\left(\int_{(g(u)z)^{\alpha/d}}^{(Mz)^{\alpha/d}}\frac{1}{s^{d/\alpha}}\,\d s\right)\, \d u.
\end{align*}
Hence
\begin{align*}
{\rm (II)}_2
&=\int_{B_1\cap ((0,t]\times \R^d\times (0,1])}\frac{z^2}{s^{2d/\alpha}}
g\left(\frac{|y|}{s^{1/\alpha}}\right)^2 \,\d s\,\d y \, \lambda(\d z)\\
&=\int_{(0,D\wedge1]}
\left\{\int_0^{H_1(z)}\frac{1}{s^{2d/\alpha}}
\left(\int_{|y|\ge H_2(s,z)}
g\left(\frac{|y|}{s^{1/\alpha}}\right)^2\,\d y\right) \,\d s\right\}
z^2\,\lambda(\d z)\\
&=\omega_d
\int_{(0,D\wedge1]}  \left\{\int_0^{(Mz)^{\alpha/d}}
\frac{1}{s^{2d/\alpha}}\left(\int_{g^{-1}(s^{d/\alpha}/z)s^{1/\alpha}}^{\infty}
g\left(\frac{r}{s^{1/\alpha}}\right)^2r^{d-1}\,\d r\right)\,\d s\right\}z^2\, \lambda(\d z)\\
&=\omega_d \int_{(0,D\wedge1]}
\left\{\int_0^{\infty}
\left(\int_{(g(u)z)^{\alpha/d}}^{(Mz)^{\alpha/d}}\frac{1}{s^{d/\alpha}}\,\d s\right)g(u)^2\, \d u\right\}
z^2\,\lambda(\d z)\\
&=
\begin{dcases}
\frac{\omega_d}{1-d/\alpha}
\int_0^{\infty} (M^{{\alpha/d}-1}-g(u)^{\alpha/d-1})g(u)^2\,\d u
\int_{(0,D\wedge1]}z^{1+\alpha/d}\,\lambda(\d z),
& d\ne\alpha,\\
\omega_d
\int_0^{\infty}g(u)^2 \log\left(\frac{M}{g(u)}\right)\, \d u
\int_{(0,D\wedge1]}z^2\,\lambda(\d z),
& d=\alpha.
\end{dcases}
\end{align*}
As a result, we obtain
\[
{\rm (II)}_2<\infty
\iff \int_{(0,1]}z^{1+\alpha/d}\,\lambda(\d z)<\infty.
\]

Furthermore,
\begin{align*}
{\rm (II)}_3
&=\int_{B_3\cap((0,t]\times \R^d\times (0,1])}
\frac{z^2}{s^{2d/\alpha}}g\left(\frac{|y|}{s^{1/\alpha}}\right)^2 \,\d s\,\d y \, \lambda(\d z)\\
&=\int_{(D\wedge1,1]}
\left\{\int_0^{t}\frac{1}{s^{2d/\alpha}}
\left(\int_{|y|\ge H_2(s,z)}
g\left(\frac{|y|}{s^{1/\alpha}}\right)^2\,\d y\right) \,\d s\right\}
z^2\,\lambda(\d z)\\
&=\omega_d
\int_{(D\wedge1,1]}  \left\{\int_0^{t}
\frac{1}{s^{2d/\alpha}}\left(\int_{g^{-1}(s^{d/\alpha}/z)s^{1/\alpha}}^{\infty}
g\left(\frac{r}{s^{1/\alpha}}\right)^2r^{d-1}\,\d r\right)\,\d s\right\}z^2\, \lambda(\d z)\\
&=\omega_d \int_{(D\wedge1,1]}
\left\{\int_0^{\infty}
\left(\int_{(g(u)z)^{\alpha/d}}^{t}\frac{1}{s^{d/\alpha}}\,\d s\right)g(u)^2\, \d u\right\}
z^2\,\lambda(\d z).
\end{align*}
We can calculate the last integral above as follows:
if $d<\alpha$, then
\[
{\rm (II)}_3
=\frac{d\omega_d}{\alpha-d}
\int_{(D\wedge1,1]}\left(\int_0^{\infty}(t^{1-d/\alpha}-(g(u)z)^{\alpha/d-1})g(u)^2\,\d u\right)z^2\lambda(\d z)
\preceq \int_{(0,1]}z^2\,\lambda(\d z).
\]
If $d>\alpha$, then
\[
{\rm (II)}_3
=\frac{d\omega_d}{d-\alpha}
\int_{(D\wedge1,1]}\left(\int_0^{\infty}((g(u)z)^{\alpha/d-1}-t^{1-d/\alpha})g(u)^2\,\d u\right)z^2\lambda(\d z)
\preceq \int_{(D\wedge1, 1]}z^{1+\alpha/d}\,\lambda(\d z).
\]
If $d=\alpha$, then
\begin{align*}
{\rm (II)}_3
&=\omega_d
\left(\int_0^{\infty}g(u)^2  \log\left(\frac{t}{g(u)}\right)\, \d u
\int_{(D\wedge1,1]}z^2\,\lambda(\d z)
+\int_0^{\infty}g(u)^2\, \d u
\int_{(D\wedge1,1]}z^2|\log z|\,\lambda(\d z)
\right)\\
&\preceq \int_{(D\wedge1,1]}z^2|\log z|\,\lambda(\d z).
\end{align*}
Hence, we see that
\[
\int_{(0,t]\times\R^d\times (0,1]}(p_s(y)z)^2{\bf 1}_{\{p_s(y)z\le 1\}}
\,\d s\, \d y\,\lambda(\d z)<\infty
\iff
\begin{dcases}
\int_{(0,1]}z^{
(1+\alpha/d)\wedge2}\,\lambda(\d z)<\infty, & d\ne \alpha,\\
\int_{(0,1]}z^2|\log z|\,\lambda(\d z)<\infty, & d=\alpha.
\end{dcases}
\]
Combining this with \eqref{eq:pz>1-2}, we arrive at the desired assertion.
\end{proof}

\section{Tail asymptotics}\label{section3}
In this section, we study
the tail behavior of the mild solution
$X(t,x)$ to
\eqref{eq:fractional-she}
in terms of its L\'evy measure.
Recall that  $\eta$ defined by
\eqref{eq:eta} is the L\'evy measure corresponding to the solution $X(t,x)$.
For $r>0$, let $\overline{\eta}(r)=\eta((r,\infty))$.
We first show basic properties of $\overline{\eta}$.

\begin{lem}\label{lem:eta-express}
The following statements hold.
\begin{enumerate}
\item $\overline{\eta}(r)<\infty$ for any $r>0$ if and only if $\lambda$ satisfies
\begin{equation}\label{eq:small-eta-conv}
\int_{(0,1]}z^{1+\alpha/d}\,\lambda(\d z)<\infty
\end{equation}
and \eqref{eq:big-conv}.

\item Under \eqref{eq:small-eta-conv} and \eqref{eq:big-conv},
$r\mapsto\overline{\eta}(r)$ is continuous and decreasing on $(0,\infty)$, and
\begin{equation}\label{eq:liminf-eta}
\liminf_{r\rightarrow\infty}r^{1+\alpha/d}\overline{\eta}(r)>0.
\end{equation}
Moreover, $\overline{\eta}$ is of extended regular variation at infinity; that is,
there exist positive constants
$\delta_1$ and $\delta_2$
such that for any $\lambda\ge 1$,
\[
\lambda^{\delta_1}
\le \liminf_{t\rightarrow\infty}\frac{\overline{\eta}(\lambda t)}{\overline{\eta}(t)}
\le \limsup_{t\rightarrow\infty}\frac{\overline{\eta}(\lambda t)}{\overline{\eta}(t)}
\le \lambda^{\delta_2}.
\]
\end{enumerate}
\end{lem}

It was proved in Theorem \ref{thm:sol-exist} that the mild solution
$X(t,x)$ to the equation
\eqref{eq:fractional-she} exists as a finite value almost surely
if and only if \eqref{eq:small-conv-0} and \eqref{eq:big-conv} are satisfied.
Since these two conditions are stronger than \eqref{eq:small-eta-conv} and \eqref{eq:big-conv},
the almost surely finiteness of $X(t,x)$ implies $\overline{\eta}(r)<\infty$ for any $r>0$.

\begin{proof}[Proof of Lemma $\ref{lem:eta-express}$]
(1)
 We first claim that
for any $r>0$,
\begin{equation}\label{eq:eta-expression}
\begin{split}
\overline{\eta}(r)
&=\omega_d\int_0^t
\left(\int_{(s^{d/\alpha}r/M,\infty)}g^{-1}(s^{d/\alpha}r/z)^d \,\lambda(\d z)\right)
s^{d/\alpha}\,\d s\\
&=\frac{\alpha \omega_d}{d}\frac{1}{r^{1+\alpha/d}}
\int_0^{t^{d/\alpha}r}
\left(\int_{(u/M,\infty)}g^{-1}\left(\frac{u}{z}\right)^d \,\lambda(\d z)\right)
u^{\alpha/d}\,\d u,
\end{split}
\end{equation} where $M=g(0)$. Indeed,
by definition,
\begin{equation}\label{eq:eta-express}
\begin{split}
\overline{\eta}(r)
&=\int_{(0,t]\times \R^d\times (0,\infty)}
{\bf 1}_{\{p_s(y)z>r\}}\,\d s\, \d y\,\lambda(\d z)\\
&=\int_{(0,\infty)}\left\{\int_0^{t\wedge H_1(z/r)}
\left(\int_{|y|<H_2(s,z/r)}\,\d y\right)\,\d s\right\}\,\lambda(\d z)\\
&=\omega_d \int_{(0,\infty)}
\left(\int_0^{t\wedge (Mz/r)^{\alpha/d}}H_2(s,z/r)^d \,\d s\right)\,\lambda(\d z)\\
&=\omega_d \int_{(0,\infty)}
\left(\int_0^{t\wedge (Mz/r)^{\alpha/d}}g^{-1}(s^{d/\alpha}r/z)^d s^{d/\alpha}\,\d s\right)\,\lambda(\d z).
\end{split}
\end{equation}
By the Fubini theorem
and the change of variables formula with $u=s^{d/\alpha}r$,
the last expression above is equal to
\begin{align*}
&\omega_d\int_0^t
\left(\int_{(s^{d/\alpha}r/M,\infty)}g^{-1}(s^{d/\alpha}r/z)^d \,\lambda(\d z)\right)
s^{d/\alpha}\,\d s\\
&=
\frac{\alpha \omega_d}{d}\frac{1}{r^{1+\alpha/d}}
\int_0^{t^{d/\alpha}r}
\left(\int_{(u/M,\infty)}g^{-1}\left(\frac{u}{z}\right)^d \,\lambda(\d z)\right)
u^{\alpha/d}\,\d u.
\end{align*}
Therefore, the proof of \eqref{eq:eta-expression} is complete.

(2) Next, we  verify the assertion (i). By \eqref{eq:g^{-1}-asymp},
we have $g^{-1}(r)\preceq r^{-1/(d+\alpha)}$ for any $r>0$
so that by \eqref{eq:eta-express},
\begin{equation}\label{eq:eta-upper-0}
\begin{split}
\overline{\eta}(r)
&\preceq
\int_0^t
\left(\int_{(s^{d/\alpha}r/M,\infty)}\left(\frac{z}{s^{d/\alpha}r}\right)^{d/(d+\alpha)}\,\lambda(\d z)\right)
\,s^{d/\alpha}\,\d s\\
&=\frac{1}{r^{d/(d+\alpha)}}\int_0^t
\left(\int_{(s^{d/\alpha}r/M,\infty)}z^{d/(d+\alpha)}\,\lambda(\d z)\right)\,s^{d/(d+\alpha)}\,\d s.
\end{split}
\end{equation}
Then by  the Fubini theorem,
\begin{equation}\label{eq:eta-fubini}
\begin{split}
&\int_0^t
\left(\int_{(s^{d/\alpha}r/M,\infty)}z^{d/(d+\alpha)}\,\lambda(\d z)\right)\,s^{d/(d+\alpha)}\,\d s\\
&=\int_{(0,\infty)}\left(\int_0^{t\wedge (Mz/r)^{\alpha/d}}s^{d/(d+\alpha)}\,\d s\right)z^{d/(d+\alpha)}\,\lambda(\d z)\\
&=\frac{d+\alpha}{2d+\alpha}
\int_{(0,\infty)}\left(t\wedge \left(\frac{Mz}{r}\right)^{\alpha/d}\right)^{1+d/(d+\alpha)}z^{d/(d+\alpha)}\,\lambda(\d z)\\
&=\frac{d+\alpha}{2d+\alpha}
\left(\frac{M^{\alpha/d+\alpha/(d+\alpha)}}{r^{\alpha/d+\alpha/(d+\alpha)}}
\int_{(0,rt^{d/\alpha}/M]}z^{1+\alpha/d}\,\lambda(\d z)
+t\int_{(rt^{d/\alpha}/M,\infty)}z^{d/(d+\alpha)}\,\lambda(\d z)\right).
\end{split}
\end{equation}
Namely, we have
\[
\overline{\eta}(r)
\preceq
\frac{1}{r^{1+\alpha/d}}
\int_{(0,rt^{d/\alpha}/M]}z^{1+\alpha/d}\,\lambda(\d z)
+\frac{t}{r^{d/(d+\alpha)}}
\int_{(rt^{d/\alpha}/M,\infty)}z^{d/(d+\alpha)}\,\lambda(\d z).
\]

On the other hand, if $z\ge 2s^{d/\alpha}r/M$, then by \eqref{eq:g^{-1}-asymp},
\[
g^{-1}\left(\frac{s^{d/\alpha}r}{z}\right)^d\asymp \left(\frac{z}{s^{d/\alpha}r}\right)^{d/(d+\alpha)}.
\]
As in \eqref{eq:eta-fubini} we obtain
\begin{equation}\label{eq:eta-lower-0}
\begin{split}
\overline{\eta}(r)
&\ge  \frac{\omega_d}{r^{d/(d+\alpha)}}\int_{(0,\infty)}
\left(\int_0^{t\wedge (Mz/(2r))^{\alpha/d}}g^{-1}\left(\frac{s^{d/\alpha}r}{z}\right)^d
s^{d/\alpha}\,\d s\right)
\,\lambda(\d z)\\
&\asymp \frac{1}{r^{1+\alpha/d}}
\int_{(0,2rt^{d/\alpha}/M]}z^{1+\alpha/d}\,\lambda(\d z)
+\frac{t}{r^{d/(d+\alpha)}}
\int_{(2rt^{d/\alpha}/M,\infty)}z^{d/(d+\alpha)}\,\lambda(\d z).
\end{split}
\end{equation}
Therefore, $\overline{\eta}(r)<\infty$ for any $r>0$
if and only if $\lambda$ satisfies
\eqref{eq:small-eta-conv} and \eqref{eq:big-conv}.

(3) By the Fubini theorem, we have
\begin{align*}
\overline{\eta}(r)
&=\int_{(0,t]\times \R^d\times (0,\infty)}
{\bf 1}_{\{p_s(y)z>r\}}\,\d s\, \d y\,\lambda(\d z)\\
&=\int_{(0,t]\times \R^d\times (0,\infty)}
{\bf 1}_{\{p_s(y)z\ge r\}}\,\d s\, \d y\,\lambda(\d z)
=\eta([r,\infty)),
\end{align*}
so that $\overline{\eta}(r)$ is continuous on $(0,\infty)$.
The decreasing property of $\overline{\eta}(r)$ is obvious by the definition, and
\eqref{eq:liminf-eta}
is a direct consequence of \eqref{eq:eta-expression}.

(4)  By the fundamental theorem of calculus,
\[
\log \overline{\eta}(r)
=\log\overline{\eta}(1)+\int_1^r \frac{\overline{\eta}'(s)}{\overline{\eta}(s)}\,\d s.
\]
Then by the change of variables formula with $rs^{d/\alpha}=t^{d/\alpha}v$ in \eqref{eq:eta-expression}, we have
$$\overline{\eta}(r)
=\frac{\alpha \omega_d}{d}\frac{t^{1+d/\alpha}}{r^{1+\alpha/d}}
\int_0^{r}
\left(\int_{(vt^{d/\alpha}/M,\infty)}g^{-1}\left(\frac{vt^{d/\alpha}}{z}\right)^d \,\lambda(\d z)\right)
v^{\alpha/d}\,\d v. $$
Since
\begin{align*}
\overline{\eta}'(r)
&=\frac{\alpha\omega_d}{d}t^{1+d/\alpha}
\Biggl\{-\frac{1+\alpha/d}{r^{2+\alpha/d}}
\int_0^r
\left(\int_{(t^{d/\alpha}v/M,\infty)}
g^{-1}\left(\frac{t^{d/\alpha}v}{z}\right)^d\,\lambda(\d z)\right)
v^{\alpha/d}\,\d v\\
&\qquad\qquad\qquad+\frac{1}{r}\int_{(t^{d/\alpha}r/M,\infty)}g^{-1}\left(\frac{t^{d/\alpha}r}{z}\right)^d\,\lambda(\d z)
\Biggr\}\\
&=\frac{\alpha \omega_d t^{1+d/\alpha}}{dr^{2+\alpha/d}}
\Biggl\{
r^{1+\alpha/d}\int_{(t^{d/\alpha}r/M,\infty)}g^{-1}\left(\frac{t^{d/\alpha}r}{z}\right)^d\,\lambda(\d z)\\
&\qquad\qquad\qquad-\left(1+\frac{\alpha}{d}\right)
\int_0^r
\left(\int_{(t^{d/\alpha}v/M,\infty)}g^{-1}\left(\frac{t^{d/\alpha}v}{z}\right)^d\,\lambda(\d z)\right)
v^{\alpha/d}\,\d v
\Biggr\},
\end{align*}
we have
\begin{equation}\label{eq:eta-deri}
\begin{split}
\frac{\overline{\eta}'(r)}{\overline{\eta}(r)}
&=\frac{1}{r}
\left\{
\frac{
r^{1+\alpha/d}
\int_{(t^{d/\alpha}r/M,\infty)} g^{-1}(t^{d/\alpha}r/z)^d\,\lambda(\d z)}
{\int_0^r \left(\int_{(t^{d/\alpha}v/M,\infty)}g^{-1}(t^{d/\alpha}v/z)^d\,\lambda(\d z)\right)
v^{\alpha/d}
\,\d v}
-\left(1+\frac{\alpha}{d}\right)\right\}.
\end{split}
\end{equation}
Since the function
\[
f(r)=\int_{(t^{d/\alpha}r/M,\infty)} g^{-1}\left(\frac{t^{d/\alpha}r}{z}\right)^d\,\lambda(\d z)
\]
is decreasing in $r$, we obtain
\begin{align*}
&\int_0^r
\left(\int_{(t^{d/\alpha}v/M,\infty)}g^{-1}\left(\frac{t^{d/\alpha}v}{z}\right)^d\,\lambda(\d z)\right)
v^{\alpha/d}\,\d v\\
&\ge \int_0^r \left(\int_{(t^{d/\alpha}r/M,\infty)}
g^{-1}\left(\frac{t^{d/\alpha}r}{z}\right)^d\,\lambda(\d z)\right)
v^{\alpha/d}\,\d v
=\frac{r^{1+\alpha/d}}{1+\alpha/d}\int_{(t^{d/\alpha}r/M,\infty)}
g^{-1}\left(\frac{t^{d/\alpha}r}{z}\right)^d\,\lambda(\d z)
\end{align*}
and thus
$$\frac{\overline{\eta}'(r)}{\overline{\eta}(r)}\le 0,
\quad r\ge 1. $$
We also see by \eqref{eq:eta-deri} that
$$\frac{\overline{\eta}'(r)}{\overline{\eta}(r)}\ge -\frac{1}{r}\left(1+\frac{\alpha}{d}\right),
\quad r\ge 1. $$
Then,
the function $h(s):=
s\overline{\eta}'(s)/\overline{\eta}(s)$
is bounded on $[1,\infty)$.
Moreover,  since
\[
\int_1^r \frac{\overline{\eta}'(s)}{\overline{\eta}(s)}\,\d s
=\int_1^r \frac{h(s)}{s}\,\d s,
\]
it follows by \cite[p.\ 74, Theorem 2.2.6]{BGT89} that
$\overline{\eta}$ is of extended regularly variation at infinity (see \cite[p.\ 65, Definition]{BGT89}
or \cite[p.\ 66, Theorem 2.0.7]{BGT89}).
\end{proof}

Let $(\tau_i,\eta_i,\zeta_i)\in (0,\infty)\times \R^d\times (0,\infty) \ (i\ge 1)$
be a realization of the points associated with the Poisson random measure
$\mu$.
For $t\ge 0$, we define
the following largest contribution to the process $X(t,x)$ by a single atom:
\[
\overline{X}(t)=\sup_{i\ge 1, \, \tau_i\le t}p_{t-\tau_i}(\eta_i)\zeta_i.
\]

We now prove that
the tail behavior of the mild solution $X(t,x)$ to \eqref{eq:fractional-she}
is determined by $\overline{\eta}$, and also dominated by $\overline{X}(t)$.
\begin{thm}\label{thm:eta-tail}
Suppose that \eqref{eq:small-conv-0} and \eqref{eq:big-conv} are satisfied.
Then, for
any
$t>0$ and $x\in \R^d$,
\[
P(X(t,x)>r)\sim P(\overline{X}(t)>r)\sim \overline{\eta}(r),
\quad r\rightarrow\infty.
\]
\end{thm}

\begin{proof}
As shown in Lemma \ref{lem:eta-express}(ii), $\overline{\eta}$ is of extended regular variation at infinity.
Then, it follows by \cite[p.\ 66, Theorem 2.0.7]{BGT89}  that,
for each $\Lambda>1$, the next asymptotic relation holds
uniformly in $\lambda\in [1,\Lambda]$:
\[
(1+o(1))\lambda^{\delta_1}
\le \frac{\overline{\eta}(\lambda t)}{\overline{\eta}(t)}
\le (1+o(1))\lambda^{\delta_2},
\quad t\rightarrow\infty.
\]
Then for each fixed $s>0$,
\begin{equation}\label{eq:extended-regular}
(1+o(1))\left(1+\frac{s}{r}\right)^{\delta_1}
\le \frac{\overline{\eta}(r+s)}{\overline{\eta}(r)}
=\frac{\overline{\eta}(r(1+s/r))}{\overline{\eta}(r)}
\le (1+o(1))\left(1+\frac{s}{r}\right)^{\delta_2}
\end{equation}
and thus
\[
\lim_{r\rightarrow\infty}
\frac{\overline{\eta}(r+s)}{\overline{\eta}(r)}=1.
\]
In particular, if we normalize $\overline{\eta}(t)$
as $\overline{\eta}_0(t):=\overline{\eta}(t)/\overline{\eta}(1)$ for $t\ge 1$,
then
\[
\lim_{r\rightarrow\infty}
\frac{\overline{\eta}_0(r+s)}{\overline{\eta}_0(r)}=1.
\]
We can also see by \eqref{eq:extended-regular} that
\[
\limsup_{t\rightarrow\infty}
\frac{\overline{\eta}_0(t)}{\overline{\eta}_0(2t)}\le
\frac{1}{2^{\delta_1}}.
\]
Hence,
$\overline{\eta}_0$
is subexponential by \cite[p.\ 429]{BGT89}:
if $\overline{\eta}_0*\overline{\eta}_0$ denotes the convolution of $\overline{\eta}_0$ and itself, then
\[
\lim_{x\rightarrow\infty}\frac{\overline{\eta}_0*\overline{\eta}_0(x)}{\overline{\eta}_0(x)}=2.
\]
Namely, $\overline{\eta}_0$ belongs to ${\cal S}_0$ in the sense of
\cite[p.\ 297, the beginning of Section 3]{P07}.
Since Theorem \ref{thm:sol-exist} (i)
says that
the distribution of $X(t,x)$ is infinitely divisible with L\'evy measure $\eta$,
we can apply \cite[Theorem 3.3]{P07} with $\gamma=0$ to see that
\[
\lim_{r\rightarrow\infty}\frac{P(X(t,x)>r)}{\overline{\eta}(r)}=1.
\]

For $r>0$, let
\[
S_r=\left\{(s,y,z)\in (0,t]\times \R^d\times (0,\infty) : p_s(y)z>r\right\}.
\]
Since $\nu(S_r)=\overline{\eta}(r)$, we have
\[
P(\overline{X}(t)\le r)=P(\mu(S_r)=0)=e^{-\overline{\eta}(r)}
\]
and thus
\[
P(\overline{X}(t)>r)=1-e^{-\overline{\eta}(r)}\sim \overline{\eta}(r),
\quad r\rightarrow\infty.
\]
Hence the proof is complete.
\end{proof}

The following statement indicates that
one can deduce the asymptotic behavior of $\overline{\eta}(r)$ as $r\rightarrow\infty$ directly
from that of $\overline{\lambda}(r)$ under some mild conditions.
\begin{lem}\label{lem:eta-express-2}
Suppose that \eqref{eq:small-eta-conv} and \eqref{eq:big-conv} are satisfied. Then the following statements hold.
\begin{enumerate}
\item[{\rm (i)}]
If $\int_{(1,
\infty)}z^{1+\alpha/d}\,\lambda(\d z)<\infty$,
then as $r\rightarrow\infty$,
\[
\overline{\eta}(r)\sim \frac{\alpha\omega_d}{d}\frac{1}{r^{1+\alpha/d}}
\int_0^{\infty}u^{\alpha/d}\left(\int_{(u/g(0),\infty)}g^{-1}\left(\frac{u}{z}\right)^d\,\lambda(\d z)\right)\,\d u.
\]
\item[{\rm (ii)}]
Assume that for some $\beta\in [d/(d+\alpha),1+\alpha/d]$ and slowly varying function $l$
at infinity,
\begin{equation}\label{eq:lambda-slowly}
\overline{\lambda}(r)
\asymp
\frac{l(r)}{r^{\beta}}, \quad r\ge 1.
\end{equation}
\begin{enumerate}
\item[{\rm (ii-a)}]
If $d/(d+\alpha)<\beta<1+\alpha/d$, then
$\overline{\eta}(r)\asymp\overline{\lambda}(r)$ as $r\rightarrow\infty$.

\item[{\rm (ii-b)}]
If $\beta=d/(d+\alpha)$ and $\int_1^{\infty}l(r)/r\,\d r<\infty$,
then as $r\rightarrow\infty$,
\[
\overline{\eta}(r)\asymp \frac{1}{r^{\alpha/(d+\alpha)}}\int_r^{\infty}\frac{l(u)}{u}\, \d u.
\]
\item[{\rm (ii-c)}]
If $\beta=1+\alpha/d$, then as $r\rightarrow\infty$,
\[
\overline{\eta}(r)\asymp \frac{1}{r^{1+\alpha/d}}\int_1^r\frac{l(u)}{u}\, \d u.
\]
\end{enumerate}
\end{enumerate}
\end{lem}

\begin{proof}
(1) Let $M=g(0)$.
Assume that $
\int_{(0,\infty)}z^{1+\alpha/d}\,\lambda(\d z)<\infty$.
Then by the Fubini theorem,
\begin{align*}
&\int_0^{\infty}
\left(\int_{(u/M,\infty)}g^{-1}\left(\frac{u}{z}\right)^d \,\lambda(\d z)\right)
u^{\alpha/d}\,\d u
=\int_{(0,\infty)}
\left(\int_0^{Mz}
u^{\alpha/d}g^{-1}\left(\frac{u}{z}\right)^d\,\d u\right)\,\lambda(\d z)\\
&=\int_{(0,\infty)}
\left(\int_{Mz/2}^{Mz}
u^{\alpha/d}g^{-1}\left(\frac{u}{z}\right)^d \,\d u\right)\,\lambda(\d z)
 +\int_{(0,\infty)}
\left(\int_0^{Mz/2}
u^{\alpha/d}g^{-1}\left(\frac{u}{z}\right)^d \,\d u\right)\,\lambda(\d z)\\
&={\rm (I)}+{\rm (II)}.
\end{align*}
Since there exists $c_1>0$ such that $g^{-1}(s/z)^d\le c_1$ for $s\ge Mz/2$,
we have
\[
{\rm (I)}\le c_1\int_{(0,\infty)}
\left(\int_0^{Mz}u^{\alpha/d}\,\d u\right)\,\lambda(\d z)
=\frac{c_1M^{1+
\alpha/d}}{1+\alpha/d}\int_{(0,\infty)}z^{1+\alpha/d}\,\lambda(\d z)<\infty.
\]
By \eqref{eq:g^{-1}-asymp}, there exist positive constants $c_2$ and $c_3$ such that
\[
{\rm (II)}\le c_2\int_{(0,\infty)}
\left(\int_0^{Mz/2}u^{\alpha/d}\left(\frac{z}{u}\right)^{d/(d+\alpha)}\,\d u\right)\,\lambda(\d z)
=c_3\int_{(0,\infty)}z^{1+\alpha/d}\,\lambda(\d z)<\infty.
\]
Therefore,
\[
\int_0^{\infty}
\left(\int_{(u/M,\infty)}g^{-1}\left(\frac{u}{z}\right)^d \,\lambda(\d z)\right)u^{\alpha/d}\,\d u<\infty.
\]
Then (i) follows by \eqref{eq:eta-expression}.

(2)
Since $g^{-1}(r)\preceq r^{-1/(d+\alpha)} \ (r>0)$ by \eqref{eq:g^{-1}-asymp},
we have by \eqref{eq:eta-expression},
\[\overline{\eta}(r)
\preceq
\frac{1}{r^{1+\alpha/d}}\int_0^{t^{d/\alpha}r}
\left(\int_{(u/M,\infty)}z^{d/(d+\alpha)}\,\lambda(\d z)\right)\,u^{\alpha/d-d/(d+\alpha)}\,\d u.
\]
Let
\begin{equation*}
\begin{split}
&\int_0^{t^{d/\alpha}r}
\left(\int_{(u/M,\infty)}z^{d/(d+\alpha)}\,\lambda(\d z)\right)\,u^{\alpha/d-d/(d+\alpha)}\,\d u\\
&=
\int_0^{1}
\left(\int_{(u/M,\infty)}z^{d/(d+\alpha)} \,\lambda(\d z)\right)
u^{\alpha/d-d/(d+\alpha)}\,\d u\\
&\quad +\int_1^{t^{d/\alpha}r}
\left(\int_{(u/M,\infty)}z^{d/(d+\alpha)} \,\lambda(\d z)\right)
u^{\alpha/d-d/(d+\alpha)}\,\d u\\
&={\rm (III)}+{\rm (IV)}.
\end{split}
\end{equation*}
Then by the Fubini theorem with \eqref{eq:small-eta-conv} and \eqref{eq:big-conv},
\begin{equation*}
\begin{split}
{\rm (III)}&=\int_{(0,\infty)}
\left(\int_0^{1\wedge(zM)}
u^{\alpha/d-d/(d+\alpha)}\,\d u\right)z^{d/(d+\alpha)} \,\lambda(\d z)\\
&\preceq
\int_{(0,1]}
z^{1+\alpha/d}\, \lambda(\d z)+\int_{(1,\infty)}
z^{d/(d+\alpha)} \,\lambda(\d z)\asymp 1.
\end{split}
\end{equation*}
Since the Fubini theorem yields
\begin{align*}
&\int_{(r,\infty)}z^{d/(d+\alpha)} \,\lambda(\d z)
=\frac{d}{d+\alpha}\int_{(r,\infty)}\left(\int_0^z v^{-\alpha/(d+\alpha)}\, \d v\right) \,\lambda(\d z)\\
&=\frac{d}{d+\alpha}
\left\{\int_{(r,\infty)}\left(\int_0^r v^{-\alpha/(d+\alpha)}\,\d v\right) \,\lambda(\d z)
+\int_{(r,\infty)}\left(\int_r^z v^{-\alpha/(d+\alpha)}\,\d v\right) \,\lambda(\d z)\right\}\\
&=r^{d/(d+\alpha)}\overline{\lambda}(r)
+\frac{d}{d+\alpha}\int_r^{\infty} v^{-\alpha/(d+\alpha)}\overline{\lambda}(v)\,\d v,
\end{align*}
we obtain
\begin{align*}
{\rm (IV)}
&\asymp \int_1^{t^{d/\alpha}r} u^{\alpha/d}
\overline{\lambda}\left(\frac{u}{M}\right)\,\d u
+\int_1^{t^{d/\alpha}r}
\left(\int_{u/M}^{\infty} v^{-\alpha/(d+\alpha)}\overline{\lambda}(v)\,\d v\right)
u^{\alpha/d-d/(d+\alpha)}\,\d u\\
&={\rm (IV)_1}+{\rm (IV)_2}.
\end{align*}

As the function $l$ is slowly varying at infinity,
we have the following:
\begin{itemize}
\item If $d/(d+\alpha)<\beta<1+\alpha/d$, then
\[
{\rm (IV)_1}\asymp {\rm (IV)_2}\asymp r^{1+\alpha/d}\overline{\lambda}(r), \quad r\rightarrow\infty,
\]
so that ${\rm (IV)}\asymp r^{1+\alpha/d}\overline{\lambda}(r)$ as $r\rightarrow\infty$.
\item For $\beta=d/(d+\alpha)$ and $\int_1^{\infty}l(r)/r\,\d r<\infty$,
we see by \cite[p.\ 27, Proposition 1.5.9b]{BGT89} that for any $c>0$,
\begin{equation}\label{eq:int-slow}
\int_{cr}^{\infty} \frac{l(u)}{u}\,\d u\sim \int_r^{\infty}\frac{l(u)}{u}\,\d u.
\end{equation}
Hence
\[
{\rm (IV)_1}\asymp r^{1+\alpha/d}\overline{\lambda}(r)=r^{1+\alpha/d-d/(d+\alpha)}l(r),
\quad {\rm (IV)_2}\asymp r^{1+\alpha/d-d/(d+\alpha)}\int_r^{\infty}\frac{l(u)}{u}\,\d u, \quad r\rightarrow\infty.
\]
Moreover, since
\[
\lim_{r\rightarrow\infty}\frac{1}{l(r)}\int_r^{\infty}\frac{l(u)}{u}\,\d u=\infty
\]
by \cite[p.\ 27, Proposition 1.5.9b]{BGT89} again,
we have
${\rm (IV)}\asymp r^{1+\alpha/d-d/(d+\alpha)}\int_r^{\infty}l(u)/u\,\d u$ as $r\rightarrow\infty$.
\item For $\beta=1+\alpha/d$, note that for any $c>0$, we have as $r\rightarrow\infty$,
thanks to \cite[p.\ 26, Proposition 1.5.9a]{BGT89},
\begin{equation}\label{eq:int-slow}
\int_1^{cr}\frac{l(u)}{u}\,\d u\sim \int_1^r\frac{l(u)}{u}\,\d u.
\end{equation}
Therefore,
\[
{\rm (IV)_1}\asymp \int_1^r\frac{l(u)}{u}\, \d u,
\quad {\rm (IV)_2}\asymp r^{1+\alpha/d}\overline{\lambda}(r)=l(r), \quad r\rightarrow\infty.
\]
Moreover, since  \[
\lim_{r\rightarrow\infty}\frac{1}{l(r)}\int_1^r \frac{l(u)}{u}\,\d u=\infty
\]
by \cite[p.\ 26, Proposition 1.5.9a]{BGT89} again,
we have
${\rm (IV)}\asymp \int_1^rl(u)/u\,\d u$ as $r\rightarrow\infty$.
\end{itemize}
Noting that
\[
\overline{\eta}(r)\preceq \frac{1}{r^{1+d/\alpha}}({\rm (III)}+{\rm (IV)})
\asymp \frac{{\rm (IV)}}{r^{1+d/\alpha}},
\]
we obtain the desired upper bounds of $\overline{\eta}(r)$ for (ii-a)--(ii-c).
By \eqref{eq:eta-lower-0} and the same argument as before,
we also get the desired lower bounds of $\overline{\eta}(r)$ for (ii-a)--(ii-c).
\end{proof}

\section{Tails of the spatial supremum}\label{section4}
In this section, we
consider the tail asymptotics of the local supremum of the mild
solution $X(t,x)$ to \eqref{eq:fractional-she}.

\subsection{Tail of the mild solution}
Fix $A\in {\cal B}(\R^d)$.
For $B\in
{\cal B}((0,\infty))$, define
\begin{equation}\label{eq:eta_A}
\eta_A(B)=\nu\left(\left\{(s,y,z)\in (0,t]\times \R^d\times (0,\infty)
: \frac{z}{(t-s)^{d/\alpha}}g\left(\frac{d(y,A)}{(t-s)^{1/\alpha}}\right)
\in B
\right\}\right),
\end{equation}
where $d(x,A)=\inf_{y\in A}|x-y|$ for any $x\in \R^d.$
Note that by the definition of $\nu$,
\[
\eta_A(B)=\nu\left(\left\{(s,y,z)\in (0,t]\times \R^d\times (0,\infty)
: \frac{z}{s^{d/\alpha}}g\left(\frac{d(y,A)}{s^{1/\alpha}}\right)
\in B
\right\}\right).
\]
For any $r>0$,
let
$\overline{\eta_A}(r)=\eta_A((r,\infty))$.

\begin{prop}\label{thm:sup-tail}
Assume that $\lambda$ satisfies   \eqref{eq:big-conv} and
\begin{equation}\label{eq:xa-conv-1}
\begin{dcases}
\int_{(0,1]}z \,\lambda(\d z)<\infty,&\quad \alpha\le d,\\
\int_{(0,1]}z^{\gamma}\,\lambda(\d z)<\infty,&\quad\alpha>d=1,
\end{dcases}
\end{equation} where $\gamma\in ((1+\alpha)/2, \alpha)$.
For any bounded Borel set $A\subset \R^d$, if the normalization of $\overline{\eta_A}(r)$ is subexponential, then
\begin{equation}\label{eq:xa-tail}
P\left(\sup_{x\in A}X(t,x)>r\right)\sim\overline{\eta_A}(r), \quad r\rightarrow\infty.
\end{equation}
\end{prop}

\begin{proof}
(1) We first note that for the proof of \eqref{eq:xa-tail},
it is enough to prove that $X(t,\cdot)$ has a continuous modification,
and for any bounded Borel set $A\subset \R^d$,
\begin{equation}\label{eq:x-bound}
P\left(\sup_{x\in A}|X(t,x)|<\infty\right)=1.
\end{equation}
To adjust the notations of \cite[Section 3]{RS93},
we define
$S=(0,t]\times \R^d$, $M(\d t\,\d x)=\Lambda(\d t \, \d x)$ and $f_t(x,s,y)=p(t-s,x-y)$.
For any  function $\alpha:\R^d\to \R$, we also define
\[
\phi(\alpha)=\sup_{x\in A\cap \Q^d}\alpha(x), \quad q(\alpha)=\sup_{x\in A\cap \Q^d}|\alpha(x)|,
\]
and
\[
H(r)=\overline{\eta_A}(r)
=\nu\left(\left\{(s,y,z)\in (0,t]\times \R^d\times (0,\infty) : \phi(zf_t(\cdot,s,y))>r\right\}\right).
\]
Then
$X(t,x)=\int_S f_t(x,s,y)\,M(\d s\,\d y)$, and
the continuity of $X(t,\cdot)$
would yield
\[
\phi(X(t,\cdot))=\sup_{x\in A}X(t,x), \quad q(X(t,\cdot))=\sup_{x\in A}|X(t,x)|.
\]
Hence if the normalization of
$\overline{\eta_A}$ is subexponential and \eqref{eq:x-bound} holds,
then we have \eqref{eq:xa-tail} by \cite[Theorem 3.1]{RS93} applied to $X(t,\cdot)$.

(2)
We next
prove that $X(t,\cdot)$ has a continuous modification
and \eqref{eq:x-bound} holds for any bounded Borel set $A\subset\R^d$.
Assume first that $\alpha>d=1$.
Then by  definition,
\begin{equation}\label{eq:x-decomposition}
\begin{split}
X(t,x)
&=mt+\int_{(0,t]\times \R
\times (0,\infty)}
p_{t-s}(x-y)z\left({\bf 1}_{\{z\le p_{t-s}(0)^{-1}\}}-{\bf 1}_{\{z\le 1\}}\right)\,\nu(\d s\,\d y\,\d z)\\
&\quad +\int_{(0,t]\times \R
\times (0,\infty)}
p_{t-s}(x-y)z {\bf 1}_{\{z\le p_{t-s}(0)^{-1}\}}\,(\mu-\nu)(\d s\,\d y\,\d z)\\
&\quad +\int_{(0,t]\times \R
\times (0,\infty)}
p_{t-s}(x-y)z {\bf 1}_{\{z>p_{t-s}(0)^{-1}\}}\,\mu(\d s\,\d y\,\d z)\\
&=:X_1(t,x)+X_2(t,x)+X_3(t,x).
\end{split}
\end{equation}

Let
$M=g(0)$ and
\begin{align*}
X_1(t,x)
&=mt+\int_{(0,t]\times \R
\times (0,\infty)}
p_s(x-y)z{\bf 1}_{\{1<z\le p_s(0)^{-1}\}}\,\d s\,\d y\,\lambda(\d z)\\
&\quad -\int_{(0,t]\times \R
\times (0,\infty)}p_s(x-y)z {\bf 1}_{\{p_s(0)^{-1}<z\le 1\}}\,\d s\,\d y\,\lambda(\d z)\\
&=:mt+X_{11}(t,x)-X_{12}(t,x).
\end{align*}
Then by the Fubini theorem,
\begin{align*}
X_{11}(t,x)
&=\int_0^t \left(\int_{(0,\infty)}z{\bf 1}_{\{1<z\le p_s(0)^{-1}\}}\,\lambda(\d z)\right)\,\d s
=\int_{(1,\infty)} \left(\int_0^t
{\bf 1}_{\{s>(Mz)^{\alpha}\}}
\,\d s\right)z\,\lambda(\d z)\\
&=\int_{(1,\infty)}
(t-(Mz)^{\alpha}){\bf 1}_{\{z\le t^{1/\alpha}/M\}}z
\,\lambda(\d z)
\le \frac{t^{1+1/\alpha}}{M}
\overline{\lambda}(1)<\infty
\end{align*}
and
\begin{align*}
X_{12}(t,x)
&=\int_{(0,1]} \left(\int_0^t {\bf 1}_{\{s\le (Mz)^{\alpha}\}}\,\d s\right)z\,\lambda(\d z)
=\int_{(0,1]} (t\wedge (Mz)^{\alpha})z\,\lambda(\d z)\\
&\le M^{\alpha}\int_{(0,1]} z^{1+\alpha}\,\lambda(\d z)<\infty.
\end{align*}
Therefore, $X_1(t,x)$ is independent of $x$,
and there exists  $M_0\in (0,\infty)$ such that
\begin{equation}\label{eq:x_1-bound}
P\left(\sup_{x\in \R^d}|X_1(t,x)|\le M_0\right)=1.
\end{equation}

For $x,x'\in \R$,
\begin{align*}
&X_2(t,x)-X_2(t,x')\\
&=\int_{(0,t]\times \R\times (0,\infty)}
\left(p_{t-s}(x-y)-p_{t-s}(x'-y)\right)z {\bf 1}_{\{z\le (t-s)^{1/\alpha}/M\}}\,(\mu-\nu)(\d s\,\d y\,\d z).
\end{align*}
Then by \cite[Theorem 1]{MR14} with $\alpha=p=2$,
\begin{align*}
&E\left[\left(X_2(t,x)-X_2(t,x')\right)^2\right]\\
&=E\left[\left(\int_{(0,t]\times \R\times (0,\infty)}
\left(p_{t-s}(x-y)-p_{t-s}(x'-y)\right)z {\bf 1}_{\{z\le (t-s)^{1/\alpha}/M\}}\,(\mu-\nu)(\d s\,\d y\,\d z)
\right)^2\right]\\
&\le \int_{(0,t]\times \R\times (0,\infty)}
\left(p_{t-s}(x-y)-p_{t-s}(x'-y)\right)^2 z^2
{\bf 1}_{\{z\le (t-s)^{1/\alpha}/M\}}\,\nu (\d s\,\d y\,\d z)\\
&=\int_{(0,t]\times \R\times (0,\infty)}
\left(p_s(x-y)-p_s(x'-y)\right)^2 z^2
{\bf 1}_{\{z\le s^{1/\alpha}/M\}}\,\d s\,\d y\,\lambda(\d z).
\end{align*}
Since $g(r)\le g(0)=M$ for any $r>0$,
we have for any $w\in \R$,
$z>0$ and $s\in (0,t]$,
\[
p_s(w)z{\bf 1}_{\{z\le s^{1/\alpha}/M\}}
=\frac{1}{s^{1/\alpha}}g\left(\frac{|w|}{s^{1/\alpha}}\right)z{\bf 1}_{\{z\le s^{1/\alpha}/M\}}
\le \frac{M}{s^{1/\alpha}}\frac{s^{1/\alpha}}{M}=1.
\]
This implies that for any $\gamma\in (0,2)$,
\begin{align*}
&\left(p_s(x-y)-p_s(x'-y)\right)^2 z^2
{\bf 1}_{\{z\le s^{1
/\alpha}/M \}}\\
&\le (p_s(x-y)+p_s(x'-y))^{2-\gamma}z^{2-\gamma}
|p_s(x-y)-p_s(x'-y)|^{\gamma} z^{\gamma}
{\bf 1}_{\{z\le s^{1
/\alpha}/M \}}\\
&\le 2^{2-\gamma}|p_s(x-y)-p_s(x'-y)|^{\gamma}
z^{\gamma} {\bf 1}_{\{z\le t^{1
/\alpha}/M \}}.
\end{align*}
In particular, if we take $\gamma\in (1,2)$ so that
\begin{equation}\label{eq:gamma-cond}
\frac{1+\alpha}{2}<\gamma<\alpha,
\end{equation}
then by Lemma \ref{lem:heat-kernel} (ii),
\begin{align*}
&\int_{(0,t]\times \R\times (0,\infty)}
\left(p_s(x-y)-p_s(x'-y)\right)^2 z^2
{\bf 1}_{\{z\le s^{1/\alpha}/M\}}\,\d s\,\d y\,\lambda(\d z)\\
&\preceq \int_{(0,t]\times \R}
|p_s(x-y)-p_s(x'-y)|^{\gamma}\,\d s\,\d y
\int_{(0,t^{1/\alpha}/M]}z^{\gamma} \,\lambda(\d z)\\
&\preceq |x-x'|^{(1-\gamma)+\alpha}\int_{(0,t^{1/\alpha}/M]}z^{\gamma} \,\lambda(\d z).
\end{align*}
Moreover, since
$(1-\gamma)+\alpha>1$,
we see by \cite[Theorem 4.3]{K09} that $X_2(t,\cdot)$ has a continuous modification,
and for any
$\theta\in [0,(\alpha-\gamma)/2)$,
\[
E\left[
\sup_{x,x'\in \R,\, x\ne x'}\left(\frac{|X_2(t,x)-X_2(t,x')|}{|x-x'|^{\theta}}\right)^2\right]<\infty.
\]
Hence for any $a\in A$ and $\gamma\in
(1,2)$ with \eqref{eq:gamma-cond},
we have by the triangle inequality and \cite[Theorem 1]{MR14} again,
\begin{align*}
E\left[\sup_{x\in A}X_2(t,x)^2\right]
&\le 2E\left[\sup_{x, x'\in A}(X_2(t,x)-X_2(t,x'))^2\right]+2E\left[X_2(t,a)^2\right]\\
&\preceq \sup_{x,x'\in A}|x-x'|^{\theta}
+\int_{(0,t]\times \R\times (0,\infty)}p_s(a-y)^2 z^2
{\bf 1}_{\{z\le s^{1/\alpha}/M\}}\,\d s\,\d y\,\lambda(\d z)\\
&\preceq \sup_{x,x'\in A}|x-x'|^{\theta}
+\int_{(0,t]\times \R}p_s(y)^{\gamma}\,\d s\,\d y
\int_{(0, t^{1/\alpha}/M]}z^{\gamma}\,\lambda(\d z).
\end{align*}
Then by assumption,
$\int_{(0, t^{1/\alpha}/M]}z^{\gamma}\,\lambda(\d z)<\infty$.
Since, for $\gamma>1$,
\begin{equation}\label{eq:heat-gamma}
\begin{split}
&\int_{(0,t]\times \R}p_s(y)^{\gamma}\,\d s\,\d y
\le \int_0^t \left(\frac{M}{s^{1/\alpha}}\right)^{\gamma-1}\,\d s\int_{\R} p_s(y)\,\d y
\preceq t^{1+(1-\gamma)/\alpha},
\end{split}
\end{equation}
we have for any $t>0$ and bounded Borel set
$A\subset \R
$,
\begin{equation}\label{eq:x_2-bound}
E\left[\sup_{x\in A}X_2(t,x)^2\right]<\infty.
\end{equation}

Let $A\subset \R
$ be a bounded Borel set such that
$A\subset B(r)$ for some $r>0$.
Then for any $x\in A$,
\begin{align*}
X_3(t,x)
&=\int_{(0,t]\times B(2r)\times (0,\infty)}
\frac{1}{(t-s)^{1/\alpha}}g\left(\frac{|x-y|}{(t-s)^{1/\alpha}}\right)
z{\bf 1}_{\{z>(t-s)^{1/\alpha}M\}}\,\mu(\d s\, \d y\, \d z)\\
&\quad +\int_{(0,t]\times B(2r)^c\times (0,\infty)}\frac{1}{(t-s)^{1/\alpha}}g\left(\frac{|x-y|}{(t-s)^{1/\alpha}}\right)
z{\bf 1}_{\{z>(t-s)^{1/\alpha}/M\}}\,\mu(\d s\, \d y\, \d z)\\
&={\rm (I)}+{\rm (II)}.
\end{align*}

Since $g(r)\le g(0)=M$ for any $r>0$, we have
\begin{equation}\label{eq:x_3-I}
{\rm (I)}\le
\int_{(0,t]\times B(2r)\times (0,\infty)}
\frac{Mz}{(t-s)^{1/\alpha}}
{\bf 1}_{\{z>(t-s)^{1/\alpha}/M\}}\,\mu(\d s\, \d y\, \d z).
\end{equation}
Then
\begin{align*}
&\int_{(0,t]\times B(2r)\times (0,\infty)}
\left\{1\wedge \left(\frac{Mz}{s^{1/\alpha}}{\bf 1}_{\{z>s^{1/\alpha}/M\}}\right)\right\}
\,\d s\,\d y\,\lambda(\d z)\\
&=\int_{(0,t]\times B(2r)\times (0,\infty)}
{\bf 1}_{\{z>s^{1/\alpha}/M\}}\,\d s\,\d y\,\lambda(\d z)
\asymp \int_{(0,\infty)}(z^{\alpha}\wedge t)\,\lambda(\d z).
\end{align*}
As
the last integral
above is convergent
by  \eqref{eq:xa-conv-1},
the right hand side of \eqref{eq:x_3-I} and thus ${\rm (I)}$ are convergent almost surely
by \cite[p.\ 43, Theorem 2.7 (i)]{K14}.

We also note that for any $x\in A$ and $y\in B(2r)^c$, $|y|/2\le |y-x|\le 3|y|/2
$.
Then by \eqref{eq:g-asymp},
\begin{equation}\label{eq:x_3-II}
\begin{split}
{\rm (II)}
&\preceq
\int_{(0,t]\times B(2r)^c\times (0,\infty)}
\frac{1}{(t-s)^{1
/\alpha}}g\left(\frac{|y|}{(t-s)^{1/\alpha}}\right)
z
{\bf 1}_{\{z>(t-s)^{1/\alpha}/M\}}
\,\mu(\d s\, \d y\, \d z)\\
&\preceq
\int_{(0,t]\times \R
\times (0,\infty)}
p_{t-s}(y)z {\bf 1}_{\{z>(t-s)^{1
/\alpha}/M\}}
\,\mu(\d s\, \d y\, \d z).
\end{split}
\end{equation}
Since $1=d<\alpha<2$ by assumption, \eqref{eq:xa-conv-1} implies \eqref{eq:small-conv-0}.
As \eqref{eq:big-conv} holds by assumption,
we can follow the proof of Theorem \ref{thm:sol-exist}
to verify that
\[
\int_{(0,t]\times \R \times(0,\infty)}
\left\{1\wedge \left(p_s(y)z{\bf 1}_{\{z>s^{1/\alpha}/M\}}\right)\right\}\,\d s\d y\lambda(\d z)<\infty.
\]
Indeed, $z>s^{1/\alpha}/M$ if and only if $s<H_1(z)$,
and so the assertion above follows from the arguments  for \eqref{eq:a_1-equiv} and \eqref{eq:b_1-equiv}.
Then by \cite[p.\ 43, Theorem 2.7 (i)]{K14},
the last integral in \eqref{eq:x_3-II} and thus ${\rm (II)}$ are convergent almost surely.

By the argument above,
the upper bounds of ${\rm (I)}$ and ${\rm (II)}$ are independent of $x\in A$;
that is,
\begin{equation}\label{eq:x_3-bound}
P\left(\sup_{x\in A}X_3(t,x)<\infty\right)=1.
\end{equation}
Moreover, $X_3(t,\cdot)$ is continuous by the continuity of $g$
and the dominated convergence
theorem.

To summarize the argument above,
if $\alpha>d=1$, then we obtain
\begin{itemize}
\item $X_1(t,x)$ is independent of $x\in \R^d$;
\item $X_2(t,\cdot)$ has a continuous modification;
\item $X_3(t,\cdot)$ is continuous.
\end{itemize}
Therefore, $X(t,\cdot)$ also has a continuous modification,
for which we use the same notation.
We also have \eqref{eq:x-bound}  by \eqref{eq:x_1-bound}, \eqref{eq:x_2-bound} and \eqref{eq:x_3-bound}.

On the other hand, if $d\ge \alpha$, then, by \eqref{eq:xa-conv-1},
\eqref{eq:first-moment} is satisfied
and so \eqref{eq:decomp-1} holds.
Thus, we can follow the argument for $X_3(t,x)$
(as well as the proof of Theorem \ref{thm:sol-exist})
to prove \eqref{eq:x-bound}.
\end{proof}

\subsection{Tail behaviors of measures}
Let $A\subset \R^d$ be a
bounded
Borel set with $0<|\overline{A}|<\infty$ (here and in what follows, $\bar A$ denotes the closure of $A$).
We proved in Theorem \ref{thm:sup-tail} that, under some assumptions,
the measure $\eta_A$
determines
the asymptotic tail distribution of $\sup_{x\in A}X(t,x)$.
On the other hand, since the function $g(r)$ is decreasing on $(0,\infty)$,
by the definition of $\eta_A$ in \eqref{eq:eta_A},
the main contribution to the mass of $\eta_A$ comes from  the points
$(s,y,z)\in (0,t]\times \R^d\times (0,\infty)$ with $y\in \overline{A}$.
In other words, we expect that $\overline{\eta}_A(r)$ is comparable to
\begin{align*}
&\nu\left(\left\{(s,y,z)\in (0,t]\times \overline{A}\times (0,\infty) :
\frac{z}{
(t-s)^{d/\alpha} 
}g\left(\frac{d(y,A)}{(t-s)^{d/\alpha}}\right)>r\right\}\right)\\
&=\nu\left(\left\{(s,y,z)\in (0,t]\times \overline{A}\times (0,\infty) :
\frac{z}{s^{d/\alpha}}g(0)>r\right\}\right)=|\overline{A}|\overline{\tau}(r/g(0)).
\end{align*}
Here $m$ is the Lebesgue measure on ${\cal B}(\R^d)$
and $\tau$ is the measure on
${\cal B}((0,\infty))$ defined by
\begin{equation}\label{eq:def-tau}
\tau(B)=(m\otimes \lambda)\left(\left\{(s,z)\in (0,t]\times (0,\infty) :
z/s^{d/\alpha}\in B\right\}\right), \quad B\in
{\cal B}((0,\infty)).
\end{equation}

Our purpose in this subsection is to reveal the relation between
$\overline{\eta_A}(r)$ and $\overline{\tau}(r)$
with the aid of $\overline{\lambda}(r)$,
which yields the subexponentiality of $\overline{\eta_A}(r)$.
We first prove basic properties of $\overline{\tau}(r)$.
\begin{lem}\label{lem:tau-finite}
$\overline{\tau}(r)<\infty$ for any $r>0$, if and only if
\begin{equation}\label{eq:tau-conv}
\int_{(0,1]}z^{\alpha/d}\,\lambda(\d z)<\infty.
\end{equation}
Under this condition,
$r\mapsto\overline{\tau}(r)$ is continuous and decreasing on $(0,\infty)$ such that
\begin{equation}\label{eq:liminf-tau}
\liminf_{r\rightarrow\infty}r^{\alpha/d}\overline{\tau}(r)>0;
\end{equation}
moreover, $\overline{\tau}(r)$ is of extended regular variation at infinity.
\end{lem}

\begin{proof}
By definition, we have for $r>0$,
\begin{equation}\label{eq:tau-decomp}
\begin{split}
\overline{\tau}(r)
&=m\otimes\lambda\left(\left\{(s,z)\in (0,t]\times (0,\infty) :
z/s^{d/\alpha}
>r\right\}\right)
=\int_{(0,\infty)}\left(t\wedge \left(\frac{z}{r}\right)^{\alpha/d}\right)\,\lambda(\d z)\\
&=\frac{1}{r^{\alpha/d}}\int_{(0,rt^{d/\alpha}]}z^{\alpha/d}\,\lambda(\d z)
+t\overline{\lambda}(rt^{d/\alpha}).
\end{split}
\end{equation}
Therefore, the first assertion follows.

Assume that $\overline{\tau}(r)<\infty$ for any $r>0$.
Then for any $r_0>0$,
$\lim_{r\rightarrow r_0+}\overline{\tau}(r)=\overline{\tau}(r_0)$ by definition.
Since the Fubini theorem implies that
\[
\overline{\tau}(r)=\int_0^t\left(\int_{(rs^{\alpha/d},\infty)}
\,\lambda(\d z)\right)\,\d s
=\int_0^t\left(\int_{[rs^{\alpha/d},\infty)}
\,\lambda(\d z)\right)\,\d s=\tau([r,\infty)),
\]
we also have for any $r_0>0$,
$\lim_{r\rightarrow r_0-}\overline{\tau}(r)=\overline{\tau}(r_0)$.
Hence  $\overline{\tau}$ is continuous.
The decreasing property of $\overline{\tau}$ is obvious,
and
\eqref{eq:liminf-tau}
follows from \eqref{eq:tau-decomp}.

The proof of the last assertion is similar to that of Lemma \ref{lem:eta-express} (ii).
Let
$$f(r)=\frac{\alpha}{d}
\int_0^{rt^{d/\alpha}}u^{\alpha/d-1}\overline{\lambda}(u)\,\d u.$$
Then by definition,
\begin{align*}
\overline{\tau}(r)
& =m\otimes\lambda\left(\left\{(s,z)\in (0,t]\times (0,\infty) : z/s^{d/\alpha}
>r\right\}\right)\\
&=\int_{0}^t\overline{\lambda}(rs^{d/\alpha})\,\d s
=\frac{\alpha}{d}\frac{1}{r^{\alpha/d}}\int_0^{rt^{d/\alpha}}u^{\alpha/d-1}\overline{\lambda}(u)\,\d u
=\frac{f(r)}{r^{\alpha/d}}.
\end{align*}
Since
\[
\log f(r)-\log f(1)=\int_1^r\frac{f'(s)}{f(s)}\, \d s
=\int_1^r
\frac{s^{\alpha/d-1}t \overline{\lambda}(st^{d/\alpha})}{\int_0^{st^{d/\alpha}}u^{\alpha/d-1}\overline{\lambda}(u)\,\d u}
\,\d s,
\]
we have
\begin{equation}\label{eq:log-tau}
\begin{split}
\log \overline{\tau}(r)
=-\frac{\alpha}{d}\log r+\log f(r)
&=-\frac{\alpha}{d}\log r+\log f(1)
+\int_1^r
\frac{s^{\alpha/d-1}t\overline{\lambda}(st^{d/\alpha})}
{\int_0^{st^{d/\alpha}}u^{\alpha/d-1}\overline{\lambda}(u)\,\d u}\,\d s\\
&=\log f(1)
+\int_1^r\left(\frac{s^{\alpha/d}t\overline{\lambda}(st^{d/\alpha})}
{\int_0^{st^{d/\alpha}}u^{\alpha/d-1}\overline{\lambda}(u)\,\d u}-\frac{\alpha}{d}\right)\frac{1}{s}\,\d s.
\end{split}
\end{equation}

Let
\[
\xi(s)=\frac{s^{\alpha/d}t\overline{\lambda}(st^{d/\alpha})}
{\int_0^{st^{d/\alpha}}u^{\alpha/d-1}\overline{\lambda}(u)\,\d u}-\frac{\alpha}{d}.
\]
Since
\[
\int_0^{st^{d/\alpha}}u^{\alpha/d-1}\overline{\lambda}(u)\,\d u
\ge \overline{\lambda}(st^{d/\alpha})\int_0^{st^{d/\alpha}}u^{\alpha/d-1}\,\d u
=\frac{d}{\alpha}s^{\alpha/d}t \overline{\lambda}(st^{d/\alpha}),
\]
we get
\[
-\frac{\alpha}{d}<\xi(s)\le 0, \quad s\ge 0.
\]
Hence
by \eqref{eq:log-tau} and \cite[p.\ 74, Theorem 2.2.6]{BGT89}$,
\overline{\tau}$ is of extended regular
variation at infinity
so that (ii) follows.
\end{proof}

In
Subsection
\ref{section:tau} of Appendix below,
we will discuss the connection of the measure $\tau$ defined by \eqref{eq:def-tau} with
a functional of the Poisson random measure.
In particular,
we will point out there
that if \eqref{eq:tau-conv} fails, then for any $x\in \R^d$ and $r>0$,
\[
\sup_{y\in B(x,r)}X(t,y)=\infty, \quad \text{$P$-a.s.}
\]

We next reveal the relation between $\overline{\tau}(r)$ and $\overline{\lambda}(r)$.
\begin{lem}\label{lem:tau}
Suppose that \eqref{eq:tau-conv} holds. Then we have the following statements.
\begin{enumerate}
\item[{\rm (i)}]
If $\int_{(1,
\infty)}z^{\alpha/d}\,\lambda(\d z)<\infty$,
then as $r\rightarrow\infty$,
\[
\overline{\tau}(r)\sim\frac{1}{r^{\alpha/d}}\int_{(0,\infty)}z^{\alpha/d}\,\lambda(\d z).
\]

\item[{\rm (ii)}]
Suppose that $\overline{\lambda}(r)=l(r)/r^{\beta}$ for $r\ge 1$,
where $\beta\in [0,\alpha/d]$ and  slowly varying function $l(r)$ at infinity.
Then as $r\rightarrow\infty$,
\[
\overline{\tau}(r)\sim
\begin{dcases}
\frac{1}{r^{\alpha/d}}\left(\int_{(0,1]}z^{\alpha/d}\,\lambda(\d z)+\overline{\lambda}(1)
+\frac{\alpha}{d}\int_1^r
\frac{l(u)}{u}\,\d u \right), & \beta=\frac{\alpha}{d},\\
\frac{\alpha}{\alpha-d\beta}t^{1-d\beta/\alpha}\overline{\lambda}(r), & 0\le \beta<\frac{\alpha}{d}.
\end{dcases}
\]
In particular,
if $\beta=\alpha/d$ and $\int_1^{\infty}l(u)/u\,\d u<\infty$,
then as $r\rightarrow\infty$,
\[
\overline{\tau}(r)\sim
\frac{1}{r^{\alpha/d}}\left(\int_{(0,1]}z^{\alpha/d}\,\lambda(\d z)+\overline{\lambda}(1)
+\frac{\alpha}{d}\int_1^{\infty}
\frac{l(u)}{u}\,\d u \right).
\]
On the other hand,
if $\beta=\alpha/d$ and $\int_1^{\infty}l(u)/u\,\d u=\infty$,
then as $r\rightarrow\infty$,
\[
\overline{\tau}(r)\sim
\frac{\alpha}{d}\frac{1}{r^{\alpha/d}}
\int_1^r \frac{l(u)}{u}\,\d u.
\]
\end{enumerate}
\end{lem}

\begin{proof}
(1) By \eqref{eq:tau-decomp} and
\[
t\overline{\lambda}(rt^{d/\alpha})
=\frac{1}{r^{\alpha/d}}r^{\alpha/d}t\overline{\lambda}(rt^{d/\alpha})
\le \frac{1}{r^{\alpha/d}}\int_{(rt^{d/\alpha},\infty)}z^{\alpha/d}\,\lambda(\d z),
\]
we have (i).

(2) Let $\overline{\lambda}(r)$ satisfy the condition in (ii).
For  $r>1$,
\begin{align*}
\overline{\tau}(r)
&=m\otimes\lambda\left(\left\{(s,z)\in (0,t]\times (0,1] :
z/s^{d/\alpha}
>r\right\}\right)\\
&\quad +m\otimes\lambda\left(\left\{(s,z)\in (0,t]\times (1,\infty) :
z/s^{d/\alpha}
>r\right\}\right)\\
&=\int_{(0,1]}\left(t\wedge \left(\frac{z}{r}\right)^{\alpha/d}\right)\,\lambda(\d z)
+\int_0^t \overline{\lambda}(rs^{d/\alpha}\vee 1)\,\d s\\
&=\frac{1}{r^{\alpha/d}}\int_{(0,rt^{d/\alpha}\wedge1]}z^{\alpha/d}\,\lambda(\d z)
+t\lambda((rt^{d/\alpha}\wedge1,1])\\
&+\int_0^{t\wedge (1/r)^{\alpha/d}
}\overline{\lambda}(1)\,\d s
+\int_{t\wedge (1/r)^{\alpha/d}
}^t\overline{\lambda}(rs^{d/\alpha})\,\d s.
\end{align*}

Since
$$\int_{t\wedge (1/r)^{\alpha/d}
}^t\overline{\lambda}(rs^{d/\alpha})\,\d s
=\frac{\alpha}{d}\frac{1}{r^{\alpha/d}}
\int_{1\wedge (r t^{d/\alpha})
}^{rt^{d/\alpha}}u^{\alpha/d-1}\overline{\lambda}(u)\,\d u,$$
we see that if $\beta\in [0,\alpha/d)$, then  as $r\rightarrow\infty$,
\[
\overline{\tau}(r)
\sim \frac{\alpha/d}{\alpha/d-\beta}\frac{1}{r^{\alpha/d}}(rt^{d/\alpha})^{\alpha/d-\beta}l(rt^{d/\alpha})
\sim \frac{\alpha}{\alpha-d\beta}t^{1-d\beta/\alpha}\frac{l(r)}{r^{\beta}}
=\frac{\alpha}{\alpha-d\beta}t^{1-d\beta/\alpha}\overline{\lambda}(r).
\]
If $\beta=\alpha/d$, then the function
\[
f(r):=\int_1^{rt^{d/\alpha}}u^{\alpha/d-1}\overline{\lambda}(u)\,\d u
=\int_1^{rt^{d/\alpha}}\frac{l(u)}{u}\,\d u
\]
is slowly varying at infinity by \cite[p.26, Proposition 1.5.9a]{BGT89}.
Hence the assertion also follows from the arguments above.
\end{proof}

\begin{lem}\label{lem:eta-tail}
Suppose that \eqref{eq:tau-conv} holds.
Let $l(r)$ be a slowly varying function at infinity, and
let $A\subset \R^d$ be a bounded Borel set with
$0<|\overline{A}|<\infty$.
\begin{enumerate}
\item[{\rm (i)}]
Assume that either of the following conditions holds{\rm :}
\begin{enumerate}
\item[{\rm (a)}]
$\int_{(1,\infty)}z^{(\alpha/d)\vee (d/(d+\alpha))}\,\lambda(\d z)<\infty${\rm ;}
\item[{\rm (b)}]
$d/(d+\alpha)<\alpha/d$ and
$\overline{\lambda}(r)=l(r)/r^{\alpha/d}$ for $r\ge 1$.
\end{enumerate}
Then
\begin{equation}\label{eq:lim-ratio-}
\lim_{r\rightarrow\infty}\frac{\overline{\eta_A}(r)}{\overline{\tau}(r)}=|\overline{A}|g(0)^{\alpha/d}.
\end{equation}
\item[{\rm (ii)}]
Let $d/(d+\alpha)<\alpha/d$.
Assume that for some $\beta\in (d/(d+\alpha),\alpha/d)$,
$\overline{\lambda}(r)=l(r)/r^{\beta} $ for $r\ge 1$.
Then
\begin{equation}\label{eq:lim-ratio}
\lim_{r\rightarrow\infty}\frac{\overline{\eta_A}(r)}{\overline{\tau}(r)}
=|\overline{A}|g(0)^{\beta}+\frac{1}{t^{1-d\beta/\alpha}}\left(1-\frac{d\beta}{\alpha}\right)
\int_{(0,t]\times (\overline{A})^c}s^{-d\beta/\alpha}
g\left(\frac{d(y,A)}{s^{1/\alpha}}\right)^{\beta}\,\d s\,\d y.
\end{equation}
\end{enumerate}
\end{lem}

\begin{proof}
(1) We assume that (a) or (b) holds.
Fix a bounded Borel
set $A\subset \R^d$
with $0<|\overline{A}|<\infty$.
For $\varepsilon>0$, let $A^{\varepsilon}=\{y\in \R^d : d(y,A)<\varepsilon\}$.
Then for any $r>1$,
\begin{align*}
\overline{\eta_A}(r)
&=\nu\left(\left\{(s,y,z)\in (0,t]\times
A^{\varepsilon}
\times (0,\infty)
: \frac{z}{s^{d/\alpha}}g\left(\frac{d(y,A)}{s^{1/\alpha}}\right)>r
\right\}\right)\\
&\quad+\nu\left(\left\{(s,y,z)\in (0,t]\times
(A^{\varepsilon})^c
\times (0,\infty)
: \frac{z}{s^{d/\alpha}}g\left(\frac{d(y,A)}{s^{1/\alpha}}\right)>r
\right\}\right)\\
&={\rm (I)}+{\rm (II)}.
\end{align*}
Since $0\le d(y,A)<\varepsilon$ for any $y\in A^{\varepsilon}$,
we have
\[
{\rm (I)}\le \nu\left(\left\{(s,y,z)\in (0,t]\times A^{\varepsilon}
\times (0,\infty)
: \frac{z}{s^{d/\alpha}}>\frac{r}{g(0)}
\right\}\right)
=|A^{\varepsilon}|\overline{\tau}\left(\frac{r}{g(0)}\right).
\]
Then by  Lemma \ref{lem:tau},
\begin{equation}\label{eq:i-sup}
\limsup_{r\rightarrow\infty}\frac{{\rm (I)}}{\overline{\tau}(r)}
\le |A^{\varepsilon}|\limsup_{r\rightarrow\infty}\frac{\overline{\tau}(r/g(0))}{\overline{\tau}(r)}
=|A^{\varepsilon}|g(0)^{\alpha/d}\rightarrow|\overline{A}|g(0)^{\alpha/d},
\quad \varepsilon\rightarrow0.
\end{equation}
We here note that if (a) holds, then
\[
\int_{(1,\infty)}z^{\alpha/d}\,\lambda(\d z)
\le \int_{(1,\infty)}z^{(\alpha/d)\vee (d/(d+\alpha))}\,\lambda(\d z)<\infty,
\]
and so
Lemma \ref{lem:tau} (i) is applicable;
if (b) holds with $\int_1^\infty l(u)/u\,\d u<\infty$, then Lemma \ref{lem:tau} (ii) applies;
if (b) holds with $\int_1^\infty l(u)/u\,\d u=\infty$,
then the desired assertion follows from Lemma \ref{lem:tau} (ii) and \cite[p.26, Proposition 1.5.9a]{BGT89}.

On the other hand, since
$d(y,A)=0$ for any $y\in \overline{A}$, we have
\begin{equation*}
\begin{split}
\overline{\eta_A}(r)
&\ge \nu\left(\left\{(s,y,z)\in (0,t]\times \overline{A} \times (0,\infty)
: \frac{z}{s^{d/\alpha}}g\left(\frac{d(y,A)}{s^{1/\alpha}}\right)>r\right\}\right)\\
&=\nu\left(\left\{(s,y,z)\in (0,t]\times \overline{A} \times (0,\infty)
: \frac{z}{s^{d/\alpha}}g(0)>r\right\}\right)
=|\overline{A}|\overline{\tau}\left(\frac{r}{g(0)}\right).
\end{split}
\end{equation*}
Then by Lemma \ref{lem:tau} again,
\begin{equation}\label{eq:lim-inf}
\liminf_{r\rightarrow\infty}\frac{\overline{\eta_A}(r)}{\overline{\tau}(r)}
\ge |\overline{A}|g(0)^{\alpha/d}.
\end{equation}
Hence if we can prove that
\begin{equation}\label{eq:ii-lim}
\lim_{r\rightarrow\infty}\frac{{\rm (II)}}{\overline{\tau}(r)}=0,
\end{equation}
then by \eqref{eq:i-sup}, we will obtain
\[
\limsup_{r\rightarrow\infty}\frac{\overline{\eta_A}(r)}{\overline{\tau}(r)}
\le |\overline{A}|g(0)^{\alpha/d}.
\]
Combining this with \eqref{eq:lim-inf}, we
will
get \eqref{eq:lim-ratio-}.

Let us prove \eqref{eq:ii-lim}.
We first assume (a) with $d/(d+\alpha)\ge \alpha/d$
so that $\int_{(1,\infty)}z^{d/(d+\alpha)}\,\lambda(\d z)<\infty$.
For $R>0$, let $B(R)=\{y\in \R^d : |y|\le R\}$ and
\begin{align*}
{\rm (II)}
&=\nu\left(\left\{(s,y,z)\in (0,t]\times (A^{\varepsilon})^c\times (0,1]
: \frac{z}{s^{d/\alpha}}g\left(\frac{d(y,A)}{s^{1/\alpha}}\right)>r
\right\}\right)\\
&\quad+\nu\left(\left\{(s,y,z)\in (0,t]\times ((A^{\varepsilon})^c\cap B(R))\times (1,\infty)
: \frac{z}{s^{d/\alpha}}g\left(\frac{d(y,A)}{s^{1/\alpha}}\right)>r
\right\}\right)\\
&\quad+\nu\left(\left\{(s,y,z)\in (0,t]\times ((A^{\varepsilon})^c\cap B(R)^c)\times (1,\infty)
: \frac{z}{s^{d/\alpha}}g\left(\frac{d(y,A)}{s^{1/\alpha}}\right)>r
\right\}\right)\\
&={\rm (II)}_1+{\rm (II)}_2+{\rm (II)}_3.
\end{align*}

For $\theta>0$, we have by the Chebyshev inequality and \eqref{eq:g^{-1}-asymp},
\begin{equation*}
\begin{split}
{\rm (II)}_2
&\le \frac{1}{r^{\theta}}
\int_{(0,t]\times ((A^{\varepsilon})^c\cap B(R))\times (1,\infty)}
\left(\frac{z}{s^{d/\alpha}}g\left(\frac{d(y,A)}{s^{1/\alpha}}\right)\right)^{\theta}
{\bf 1}_{\{zg(d(y,A)/s^{1/\alpha})/s^{d/\alpha}>r\}}\,\d s\,\d y\lambda(\d z)\\
&\preceq \frac{1}{r^{\theta}}
\int_{(0,t]\times ((A^{\varepsilon})^c\cap B(R))\times (1,\infty)}
\left(\frac{z}{s^{d/\alpha}}\frac{s^{1+d/\alpha}}{d(y,A)^{d+\alpha}}\right)^{\theta}
{\bf 1}_{\{zg(0)/s^{d/\alpha}>r\}}\,\d s\,\d y\lambda(\d z)\\
&\preceq \frac{1}{\varepsilon^{\theta(d+\alpha)}}
\frac{1}{r^{\theta}}
\int_{(1,\infty)}\left(\int_0^{t\wedge (zg(0)/r)^{\alpha/d}}s^{\theta}\,\d s\right)z^{\theta}\,\lambda(\d z).
\end{split}
\end{equation*}
In particular, if we take $\theta=d/(d+\alpha)$, then
\[
{\rm (II)}_2\preceq
\frac{1}{\varepsilon^d}
\frac{1}{r^{d/(d+\alpha)}}
\int_{(1,\infty)}\left(\int_0^{t\wedge (zg(0)/r)^{\alpha/d}}s^{d/(d+\alpha)}\,\d s\right)
z^{d/(d+\alpha)}\,\lambda(\d z)=o(r^{-d/(d+\alpha)}).
\]

If $R>0$ is large enough, then for any $y\in B(R)^c$, $d(y,A)\asymp |y|$.
Hence by the Chebyshev inequality again, we have for any $\theta\in (0,d/(d+\alpha))$,
\begin{equation*}
\begin{split}
{\rm (II)}_3
&\le \frac{1}{r^{\theta}}
\int_{(0,t]\times ((A^{\varepsilon})^c\cap B(R)^c)\times (1,\infty)}
\left(\frac{z}{s^{d/\alpha}}g\left(\frac{d(y,A)}{s^{1/\alpha}}\right)\right)^{\theta}
{\bf 1}_{\{zg(d(y,A)/s^{1/\alpha})/s^{d/\alpha}>r\}}\,\d s\,\d y\lambda(\d z)\\
&\preceq \frac{1}{r^{\theta}}
\int_{(0,t]\times ((A^{\varepsilon})^c\cap B(R)^c)\times (1,\infty)}
\left(\frac{z}{s^{d/\alpha}}\frac{s^{1+d/\alpha}}{|y|^{d+\alpha}}\right)^{\theta}
{\bf 1}_{\{c_1sz/|y|^{d+\alpha}>r\}}{\bf 1}_{\{zg(0)/s^{d/\alpha}>r\}}\,\d s\,\d y\lambda(\d z)\\
&\preceq \frac{1}{\varepsilon^{\theta(d+\alpha)}}
\frac{1}{r^{\theta}}
\int_{(1,\infty)}\left\{\int_0^{t\wedge (zg(0)/r)^{\alpha/d}}
\left(\int_{|y|<(c_1 sz/r)^{1/(d+\alpha)}}\frac{1}{|y|^{\theta(d+\alpha)}}\,\d y\right)
s^{\theta}\,\d s\right\}z^{\theta}\,\lambda(\d z)\\
&\asymp \frac{1}{\varepsilon^{\theta(d+\alpha)}}
\frac{1}{r^{d/(d+\alpha)}}
\int_{(1,\infty)}\left(\int_0^{t\wedge (zg(0)/r)^{\alpha/d}}s^{d/(d+\alpha)}
\,\d s\right)z^{d/(d+\alpha)}\,\lambda(\d z)\\
&=o(r^{-d/(d+\alpha)}).
\end{split}
\end{equation*}
Following the calculation for ${\rm (II)}_2$ and ${\rm (II)}_3$,
we also obtain for any $\theta\in (0,d/(d+\alpha))$,
\begin{equation*}
\begin{split}
{\rm (II)}_1
&\preceq
\frac{1}{\varepsilon^d}
\frac{1}{r^{d/(d+\alpha)}}
\int_{(0,1]}\left(\int_0^{t\wedge (zg(0)/r)^{\alpha/d}}s^{d/(d+\alpha)}\,\d s\right)
z^{d/(d+\alpha)}\,\lambda(\d z)\\
&\quad +\frac{1}{\varepsilon^{\theta(d+\alpha)}}
\frac{1}{r^{d/(d+\alpha)}}
\int_{(0,1]}\left(\int_0^{t\wedge (zg(0)/r)^{\alpha/d}}s^{d/(d+\alpha)}
\,\d s\right)z^{d/(d+\alpha)}\,\lambda(\d z)\\
&\asymp \frac{1}{r^{1+\alpha/d}}\int_{(0,1]}z^{1+\alpha/d}\,\lambda(\d z)
=o(r^{-d/(d+\alpha)}).
\end{split}
\end{equation*}
Therefore,
\[
{\rm (II)}=o(r^{-d/(d+\alpha)}), \quad r\rightarrow\infty.
\]
Furthermore, since $d/(d+\alpha)\ge \alpha/d$ by assumption,
Lemma \ref{lem:tau} (i) yields \eqref{eq:ii-lim}.

We next assume (a) with $d/(d+\alpha)< \alpha/d$, and so
$\int_{(1,\infty)}z^{\alpha/d}\,\lambda(\d z)<\infty$.
By the Schwarz inequality,
\begin{equation}\label{eq:schwarz}
\begin{split}
{\rm (II)}
&\le \frac{1}{r^{\alpha/d}}
\int_{(0,t]\times (A^{\varepsilon})^c
\times (0,\infty)}
\frac{z^{\alpha/d}}{s}
g\left(\frac{d(y,A)}{s^{d/\alpha}}\right)^{\alpha/d}
{\bf 1}_{\{zg(d(y,A)/s^{d
/\alpha})/s>r\}}
\,\d s\,\d y\,\lambda(\d z)\\
&\le \frac{1}{r^{\alpha/d}}\int_{(0,\infty)}z^{\alpha/d}\,\lambda(\d z)
\left(\int_{(0,t]\times
(A^{\varepsilon})^c}
\frac{1}{s}g\left(\frac{d(y,A)}{s^{d/\alpha}}\right)^{\alpha/d}
{\bf 1}_{\{
s<zg(0)/r\}}\,\d s\,\d y\right).
\end{split}
\end{equation}
Since
$$
\frac{d(y,A)}{s^{1/\alpha}}\ge \frac{\varepsilon}{t^{1/\alpha}}, \quad y\in (A^{\varepsilon})^c, \, s\in (0,t],
$$
\eqref{eq:g-asymp} yields
\begin{equation}\label{eq:g-compare}
g\left(\frac{d(y,A)}{s^{1/\alpha}}\right)
\asymp
\frac{s^{1+d/\alpha}}{d(y,A)^{d+\alpha}},
\quad y\in (A^{\varepsilon})^c, \, s\in (0,t].
\end{equation}
Namely,
\begin{equation}\label{eq:g-compare-1}
\int_{(0,t]\times (A^{\varepsilon})^c}
\frac{1}{s}g\left(\frac{d(y,A)}{s^{1/\alpha}}\right)^{\alpha/d}
\,\d s\,\d y
\asymp
\int_0^t s^{\alpha/d}\,\d s
\int_{(A^{\varepsilon})^c}\frac{1}{d(y,A)^{\alpha(d+\alpha)/d}}\,\d y.
\end{equation}
Furthermore, since $A$ is bounded by assumption, there exists $R_0>0$ such that
$\overline{A}\subset B(R_0)$.
Then for any $y\in B(2R_0)^c$,
since
\[
|y|\le d(y,A)+d(0,A)\le d(y,A)+R_0\le d(y,A)+\frac{1}{2}|y|,
\]
we have $d(y,A)\ge |y|/2$.
Note also that $d(y,A)\ge \varepsilon$ for any $y\in (A^{\varepsilon})^c$.
Therefore,
\begin{equation}\label{eq:dist-int}
\begin{split}
&\int_{(A^{\varepsilon})^c}\frac{1}{d(y,A)^{\alpha(d+\alpha)/d}}\,\d y\\
&=\int_{(A^{\varepsilon})^c\cap B(2R_0)}
\frac{1}{d(y,A)^{\alpha(d+\alpha)/d}}\,\d y
+\int_{(A^{\varepsilon})^c\cap B(2R_0)^c}
\frac{1}{d(y,A)^{\alpha(d+\alpha)/d}}\,\d y\\
&\le \frac{|B(2R_0)|}
{\varepsilon^{\alpha(d+\alpha)/d}}
+2
^{\alpha(d+\alpha)/d}\int_{B(2R_0)^c}
\frac{1}{|y|^{\alpha(d+\alpha)/d}}\,\d y
<\infty.
\end{split}
\end{equation}
The integrability at the last inequality  follows by the condition $d/(d+\alpha)<\alpha/d$.
Hence by the dominated convergence theorem with \eqref{eq:schwarz}, \eqref{eq:g-compare-1}
and \eqref{eq:dist-int},
we have ${\rm (II)}=o(r^{-\alpha/d})$ as $r\rightarrow\infty$.
Combining this with Lemma \ref{lem:tau} (ii), we arrve at \eqref{eq:ii-lim}.

We finally assume (b).
Then
\begin{equation}\label{eq:ii-calc}
\begin{split}
{\rm (II)}
&
\asymp
\int_{(0,t]\times (A^{\varepsilon})^c}
\overline{\lambda}\left(\frac{s^{d/\alpha}r}{g\left(d(y,A)/s^{1/\alpha}\right)}\right)
\,\d s\,\d y\\
&=\frac{1}{r^{\alpha/d}}
\int_{(0,t]\times (A^{\varepsilon})^c}
l\left(\frac{s^{d/\alpha}r}{g\left(d(y,A)/s^{1/\alpha}\right)}\right)
\frac{1}{s}g\left(\frac{d(y,A)}{s^{1/\alpha}}\right)^{\alpha/d}\,\d s\, \d y.
\end{split}
\end{equation}

By the Potter bound (\cite[p.25, Theorem 1.5.6]{BGT89}),
for any $\delta>0$ and $C>1$,
there exists $c>0$ such that for any $x,y\in \R$ with $x,y\ge c$,
\[
\frac{l(y)}{l(x)}\le C\left(\left(\frac{y}{x}\right)^{\delta}\vee\left(\frac{y}{x}\right)^{-\delta}\right).
\]
On the other hand, we have by
\eqref{eq:g-compare},
\[
\frac{s^{d/\alpha}r}{g(d(y,A)/s^{1/\alpha})}
\asymp \frac{d(y,A)^{d+\alpha}r}{s}
\succeq
\frac{\varepsilon^{d+\alpha}r}{t}
, \quad y\in (A^{\varepsilon})^c, \, s\in (0,t], \, r>1.
\]
Hence for any $y\in (A^{\varepsilon})^c$, $s\in (0,t]$
and $r>1$, (by taking $C$ large if necessary),
\begin{equation}\label{eq:potter-bound}
l\left(\frac{s^{d/\alpha}r}{g(d(y,A)/s^{1/\alpha})}\right)/l(r)
\le C\left(\left(\frac{s^{d/\alpha}}{g(d(y,A)/s^{1/\alpha})}\right)^{\delta}\vee
\left(\frac{s^{d/\alpha}}{g(d(y,A)/s^{1/\alpha})}\right)^{-\delta}\right).
\end{equation}
Note that the right hand side above is independent of $r$.
By \eqref{eq:g-compare}, we
have for any $\varepsilon>0$ and $\delta\in {\mathbb R}$,
\[
\int_{(0,t]\times (A^{\varepsilon})^c}
\left(
\frac{1}{s}g\left(\frac{d(y,A)}{s^{1/\alpha}}\right)^{\alpha/d}\right)^{1-\delta d/\alpha}\,\d s\, \d y
\asymp
\int_{(0,t]\times (A^{\varepsilon})^c}
\left(
\frac{s^{\alpha/d}}{d(y,A)^{\alpha(d+\alpha)/d}}\right)^{1-\delta d/\alpha}\,\d s\, \d y.
\]
As in \eqref{eq:dist-int}, we can show that if $d<\alpha(d+\alpha)/d$,
then there exists $\delta_0>0$ such that the last integral above is convergent
for any $\delta\in \R$ with $|\delta|<\delta_0$.

Combining the argument above  with \eqref{eq:potter-bound}
and
$$
l\left(\frac{s^{d/\alpha}r}{g(d(y,A)/s^{1/\alpha})}\right)/l(r)\rightarrow 1, \quad r\rightarrow\infty,
$$
we can apply the dominated convergence theorem for \eqref{eq:ii-calc} to obtain
\begin{equation}\label{eq:ii-II}
{\rm (II)}
\asymp
\frac{l(r)}{r^{\alpha/d}}
\int_{(0,t]\times (A^{\varepsilon})^c}
\frac{1}{s}g\left(\frac{d(y,A)}{s^{1/\alpha}}\right)^{\alpha/d}\,\d s\,\d y, \quad r\rightarrow\infty.
\end{equation}
Since it follows by \cite[p.\ 26, Proposition 1.5.9a]{BGT89}
that
$$
\lim_{r\rightarrow\infty}\frac{1}{l(r)}\int_1^r\frac{l(u)}{u}\,\d u=\infty,
$$
Lemma \ref{lem:tau} (ii)
yields \eqref{eq:ii-lim}.
The proof is complete under the condition (i).

(2)
Assume the condition in (ii).
Then as in \eqref{eq:ii-calc},
\[
{\rm (II)}
=\overline{\lambda}(r)
\int_{(0,t]\times(A^{\varepsilon})^c}
\frac{1}{l(r)}l\left(s^{d/\alpha}g\left(\frac{d(y,A)}{s^{1/\alpha}}\right)^{-1}r\right)
\frac{1}{s^{d\beta/\alpha}}g\left(\frac{d(y,A)}{s^{1/\alpha}}\right)^{\beta}\,\d s\,\d y.
\]
Since $\beta>d/(d+\alpha)$,
we follow the proof of \eqref{eq:ii-II} to see that
\[
{\rm (II)}
\sim \overline{\lambda}(r)\int_{(0,t]\times(A^{\varepsilon})^c}
\frac{1}{s^{d\beta/\alpha}}g\left(\frac{d(y,A)}{s^{1/\alpha}}\right)^{\beta}\,\d s\,\d y,
\quad r\rightarrow\infty,
\]
whence
\[
\lim_{\varepsilon\rightarrow +0}\lim_{r\rightarrow\infty}\frac{{\rm (II)}}{\overline{\lambda}(r)}
=\int_{(0,t]\times(\overline{A})^c}
\frac{1}{s^{d\beta/\alpha}}g\left(\frac{d(y,A)}{s^{1/\alpha}}\right)^{\beta}\,\d s\,\d y.
\]
Since $g$ is bounded and $\beta<\alpha/d$, we
also
have in the same way as above
\[
\lim_{\varepsilon\rightarrow +0}
\lim_{r\rightarrow\infty}\frac{{\rm (I)}}{\overline{\lambda}(r)}
=\int_{(0,t]\times \overline{A}}
\frac{1}{s^{d\beta/\alpha}}g\left(\frac{d(y,A)}{s^{1/\alpha}}\right)^{\beta}\,\d s\,\d y.
\]
These two equalities above yield  as $r\rightarrow \infty$ and then $\varepsilon\rightarrow+0$,
\begin{equation*}
\begin{split}
\frac{\overline{\eta_A}(r)}{\overline{\lambda}(r)}
&=\frac{{\rm (I)}}{\overline{\lambda}(r)}+\frac{{\rm (II)}}{\overline{\lambda}(r)}
\rightarrow \int_{(0,t]\times \R^d}
\frac{1}{s^{d\beta/\alpha}}g\left(\frac{d(y,A)}{s^{1/\alpha}}\right)^{\beta}\,\d s\,\d y\\
&=\frac{\alpha}{\alpha-d\beta}t^{1-d\beta/\alpha}|\overline{A}|g(0)^{\beta}
+
\int_{(0,t]\times (\overline{A})^c}s^{-d\beta/\alpha}g\left(\frac{d(y,A)}{s^{1/\alpha}}\right)^{\beta}\,\d s\,\d y.
\end{split}
\end{equation*}
Combining this with Lemma \ref{lem:tau} (ii),
we complete the proof under the condition in (ii).
\end{proof}

By Proposition \ref{thm:sup-tail} with Lemmas \ref{lem:tau-finite}
and \ref{lem:eta-tail},
we have
\begin{thm}\label{cor:sup-tail}
Let $A\subset \R^d$ be a bounded Borel set with $0<|\overline{A}|<\infty$.
Suppose that
\[
\begin{dcases}
\int_{(0,1]}z^{\alpha/d} \,\lambda(\d z)<\infty,&\quad \alpha\le d,\\
\int_{(0,1]}z^{\gamma}\,\lambda(\d z)<\infty,&\quad\alpha>d=1
\end{dcases}
\]
with $\gamma\in ((1+\alpha)/2,\alpha)$,
and that either of the following conditions holds{\rm :}
\begin{itemize}
\item[{\rm (a)}] $\int_{(1,\infty)}z^{(\alpha/d)\vee (d/(d+\alpha))}\,\lambda(\d z)<\infty${\rm ;}
\item[{\rm (b)}]
$d/(d+\alpha)<\alpha/d$ and
$\overline{\lambda}(r)=l(r)/r^{\beta}$ for $r\ge 1$ with
$\beta\in (d/(d+\alpha),\alpha/d]$ and $l(r)$ being a slowly varying function at infinity.
\end{itemize}
Then,
the normalization of
$\overline{\eta_A}$ is subexponential and there is a constant $c_A>0$
such
that
\[
P\left(\sup_{x\in A}X(t,x)>r\right)\sim \overline{\eta_A}(r)\sim c_A\overline{\tau}(r), \quad r\rightarrow\infty.
\]
\end{thm}

\begin{proof}
Since
\[
z^{d/(d+\alpha)}
=\frac{d}{d+\alpha}\int_0^z\frac{1}{u^{\alpha/(d+\alpha)}}\,\d u,
\]
we have by the Fubini theorem,
\begin{align*}
&
\int_{(1,\infty)}
z^{d/(d+\alpha)}\,\lambda(\d z)
=\frac{d}{d+\alpha}
\int_{(1,\infty)}
\left(\int_0^z \frac{1}{u^{\alpha/(d+\alpha)}}\,\d u\right)\,\lambda(\d z)\\
&=\frac{d}{d+\alpha}\int_{1}^{\infty}\frac{1}{u^{\alpha/(d+\alpha)}}\left(\int_{(u,\infty)}\,\lambda(\d z)\right)\,\d u
+\frac{d}{d+\alpha}\int_0^{1}\frac{1}{u^{\alpha/(d+\alpha)}}\left(
\int_{(1,\infty)}
\lambda(\d z)\right)\,\d u\\
&=\frac{d}{d+\alpha}\int_1^{\infty}\frac{\overline{\lambda}(u)}{u^{\alpha/(d+\alpha)}}\,\d u
+ \overline{\lambda}(1).
\end{align*}
So, under (b),  \eqref{eq:big-conv} is satisfied.
Hence, under the
current assumptions,
\eqref{eq:tau-conv} and \eqref{eq:big-conv} are valid.
In particular, $\overline{\tau}(r)<\infty$ for any $r>0$.
Then by Lemma \ref{lem:eta-tail},
we have for each $s>0$,
\[
\frac{\overline{\eta_A}(r+s)}{\overline{\eta_A}(r)}
=\frac{\overline{\tau}(r)}{\overline{\eta_A}(r)}
\frac{\overline{\eta_A}(r+s)}{\overline{\tau}(r+s)}
\frac{\overline{\tau}(r+s)}{\overline{\tau}(r)}
\sim \frac{\overline{\tau}(r+s)}{\overline{\tau}(r)}, \quad r\rightarrow\infty.
\]
Since
$\overline{\tau}(r)$ is of extended regular variation at infinity by
Lemma \ref{lem:tau-finite}
(ii),
we can show that $\overline{\eta_A}$ is subexponential
by following the proof of Theorem \ref{thm:eta-tail}.
Hence by Proposition \ref{thm:sup-tail} and Lemma \ref{lem:eta-tail},
we have the desired conclusion.
\end{proof}

\section{Limiting behaviors}\label{section5}
In this section, we study the limiting behavior in space of the mild solution
$X(t,x)$ to \eqref{eq:fractional-she} with L\'evy space-time white noise.

\subsection{Growth order in space of the local supremum}
We first reveal the growth order in space of the local supremum
in terms of the measure $\tau$.
\begin{thm}\label{thm:limsup-whole}
Let $f:(0,\infty)\to (0,\infty)$ be
nondecreasing.
Suppose that
$$
\begin{dcases}
\int_{(0,1]}z^{\alpha/d} \,\lambda(\d z)<\infty,&\quad \alpha\le d,\\
\int_{(0,1]}z^{\gamma}\,\lambda(\d z)<\infty,&\quad\alpha>d=1\end{dcases}
$$
with $\gamma\in ((1+\alpha)/2,\alpha)$, and that
 either of the following conditions holds:
 \begin{itemize}
 \item[{\rm (a)}] $\int_{(1,\infty)}z^{(\alpha/d)\vee (d/(d+\alpha))}\,\lambda(\d z)<\infty${\rm ;}
\item[{\rm (b)}]
$d/(d+\alpha)<\alpha/d$ and
$\overline{\lambda}(r)=l(r)/r^{\beta}$ for $r\ge 1$
with $\beta\in (d/(d+\alpha),\alpha/d]$ and $l(r)$ being a slowly varying function at infinity.
\end{itemize}
Then, \[
\lim_{r\rightarrow\infty}\frac{\sup_{|x|\le r}X(t,x)}{f(r)}=0, \quad \text{$P$-a.s.}
\] or
\[
\limsup_{r\rightarrow\infty}\frac{\sup_{|x|\le r}X(t,x)}{f(r)}=\infty, \quad \text{$P$-a.s.}
\]
according as
the integral
$\int_1^{\infty}r^{d-1}\overline{\tau}(f(r))\,\d r$ is convergent or divergent,
where
$\overline{\tau}(r)=\tau((r,\infty))$
with $\tau$ being a measure on
$(0,\infty)$ defined by
\eqref{eq:def-tau}.
\end{thm}

\begin{proof} (1)
We first assume that
\[
\int_1^{\infty}r^{d-1}\overline{\tau}(f(r))\,\d r<\infty.
\]
Let $B(n,n+1)=\{y\in \R^d : n\le |y|<n+1\}$.
Then for any $n\ge 1$,
there exist $m_n\ge 2$ and a positive sequence $\{l_k^{(n)}\}_{0\le k\le m_n}$ such that
$m_n=O(n^{d-1})$, $l_k^{(n)}-l_{k-1}^{(n)}=1$ with $1\le k\le m_n$, and
\[
B(n,n+1)\subset \bigcup_{k=1}^{m_n}[l_{k-1}^{(n)},l_k^{(n)}]^d.
\]
Since $X(t,x)$ is stationary in $x\in \R^d$, we have
for any $K>0$,
\[
P\left(\sup_{x\in [l_{k-1}^{(n)},l_k^{(n)}]^d}X(t,x)>\frac{f(n)}{K}\right)
=P\left(\sup_{x\in [0,1]^d}X(t,x)>\frac{f(n)}{K}\right).
\]
Recall that by Lemma \ref{lem:tau-finite},
$\overline{\tau}$ is of extended regular variation at infinity.
Then by
Theorem \ref{cor:sup-tail},
\begin{equation}\label{eq:sup-tail}
\begin{split}
&P\left(\sup_{x\in B(n,n+1)}X(t,x)>\frac{f(n)}{K}\right)
\le \sum_{k=1}^{m_n}P\left(\sup_{x\in [l_{k-1}^{(n)},l_k^{(n)}]^d}X(t,x)>\frac{f(n)}{K}\right)\\
&=m_n P\left(\sup_{x\in [0,1]^d}X(t,x)>\frac{f(n)}{K}\right)
\asymp n^{d-1}
P\left(\sup_{x\in [0,1]^d}X(t,x)>\frac{f(n)}{K}\right)
\asymp
n^{d-1}
\overline{\tau}(f(n)).
\end{split}
\end{equation}
As $f$ is nondecreasing and $\overline{\tau}$ is decreasing,
\eqref{eq:sup-tail} implies that
$$
\sum_{n=1}^{\infty}P\left(\sup_{x\in B(n,n+1)}X(t,x)>\frac{f(n)}{K}\right)
\preceq
\sum_{n=1}^{\infty}n^{d-1}\overline{\tau}(f(n))
\preceq
\int_1^{\infty}r^{d-1}\overline{\tau}(f(r))\,\d r<\infty.
$$
Hence by the Borel-Cantelli lemma, we get for each $K\ge 1$,
$$
P\left(\text{there exists $N_0\ge 1$ such that for all $n\ge N_0$,
$\sup_{x\in B(n,n+1)}X(t,x)\le\frac{f(n)}{K}$}\right)=1.
$$
Then for each $K\ge 1$, $P$-a.s., we have for any $r\ge N_0$,
\[
\sup_{|x|\le r}X(t,x)\le\sup_{|x|\le [r]+1}X(t,x)
=\max_{1\le k\le [r]+1}\left(\sup_{x\in B(k-1,k)}X(t,x)\right),
\]
and thus
\begin{align*}
\frac{\sup_{|x|\le r}X(t,x)}{f(r)}
&\le\frac{\sup_{|x|\le N_0}X(t,x)}{f(r)}
\vee \max_{N_0+1\le k\le [r]+1}\left(\frac{\sup_{x\in B(k-1,k)}X(t,x)}{f(k-1)}\right)\\
&\le \frac{\sup_{|x|\le N_0}X(t,x)}{f(r)}\vee \frac{1}{K}.
\end{align*}
Letting $r\rightarrow\infty$, we get
\[
\limsup_{r\rightarrow\infty}\frac{\sup_{|x|\le r}X(t,x)}{f(r)}
\le\frac{1}{K}, \quad \text{$P$-a.s.\ for each $K\ge 1$.}
\]
Moreover,
by letting $K\rightarrow\infty$ along $\Q$,
we have
\begin{equation}\label{eq:sup-lim}
\lim_{r\rightarrow\infty}\frac{\sup_{|x|\le r}X(t,x)}{f(r)}=0, \quad \text{$P$-a.s.}
\end{equation}

(2) We next assume that
\[
\int_1^{\infty}r^{d-1}\overline{\tau}(f(r))\,\d r=\infty.
\]
For  the moment, we suppose that  $d\ge \alpha$.
Since \eqref{eq:first-moment} holds by assumption, $X(t,x)$ is expressed as \eqref{eq:decomp-1}.
For $n\in {\mathbb N}$
and $K>0$,
define $T_n=T_{B(n,n+1)}(Kf(n+1))$, that is,
\[
T_n=\left\{(s,y,z)\in (0,t]\times B(n,n+1)
\times (0,\infty) :
z/(t-s)^{d/\alpha}>Kf(n+1)
\right\}.
\]
For $n\in {\mathbb N}$,
let
$
A_n:=\left\{\mu(T_n)\ge 1\right\}
$.
Since $\{T_n\}_{n\ge1}$ are disjoint,
$\{A_n\}_{n\ge1}$ are independent and
so
\[
P(A_n)=1-e^{-\nu(T_n)}=1-\exp\left(-\overline{\tau}(Kf(n+1))|B(n,n+1)|
\right).
\]
In particular, if
$\limsup_{n\rightarrow\infty}\overline{\tau}(f(n+1))|B(n,n+1)|>0$,
then we have $\sum_{n=1}^{\infty}P(A_n)=\infty$ so that
$P(A_n \ \text{i.o.})=1$
by the second Borel-Cantelli lemma.

On the other hand, if
$\lim_{n\rightarrow\infty}\overline{\tau}(f(n+1))|B(n,n+1)|=0$
(which implies that $f(r)\to\infty $ as $r\to \infty$),
then
$\lim_{n\rightarrow\infty}\overline{\tau}(Kf(n+1))|B(n,n+1)|=0$,
because $\overline{\tau}$ is of extended regular variation at infinity by
Lemma \ref{lem:eta-express}.
Therefore,
\[
P(A_n)\sim \overline{\tau}(Kf(n+1))|B(n,n+1)|
\sim d \omega_d n^{d-1}\overline{\tau}(Kf(n+1)), \quad n\rightarrow\infty.
\]
Hence there exists
$c_1>0$
such that for any $n\ge 1$,
\[
P(A_n)\ge c_1 n^{d-1}\overline{\tau}(f(n+1)),
\]
which yields
\[
\sum_{n=1}^{\infty}P(A_n)
\ge c_1\sum_{n=1}^{\infty}n^{d-1}\overline{\tau}(f(n+1))
 \succeq
\int_1^{\infty}r^{d-1}\overline{\tau}(f(r))\,\d r=\infty.
\]
We thus have $P(A_n \ \text{i.o.})=1$ by the second Borel-Cantelli lemma again.

Let $P(A_n \ \text{i.o.})=1$ hold.
Then $P$-a.s.,
there exists
a random increasing sequence $\{n_l\}_{l\ge1}$
such that for any $l\in {\mathbb N}$, there exists
$(\tau,\zeta,\xi)\in T_{n_l}$ such that
\begin{equation*}
\begin{split}
\sup_{x\in B(n_l,n_l+1)}X(t,x)
\ge \sup_{x\in B(n_l,n_l+1)}p_{t-\tau}(x-\zeta)\xi
=
p_{t-\tau}(0)\xi
=\frac{g(0)\xi}{(t-\tau)^{d/\alpha}}\ge K g(0)f(n_l+1).
\end{split}
\end{equation*}
Hence for any $l\in {\mathbb N}$,
\begin{equation*}
\begin{split}
\frac{\sup_{|x|\le n_l}X(t,x)}{f(n_l)}
\ge \frac{\sup_{x\in B(n_l-1,n_l)}X(t,x)}{f(n_l+1)}
\ge Kg(0), \quad \text{$P$-a.s.,}
\end{split}
\end{equation*}
which yields
\[
\limsup_{r\rightarrow\infty}\frac{\sup_{|x|\le r}X(t,x)}{f(r)}\ge Kg(0), \quad \text{$P$-a.s.}
\]
Letting $K\rightarrow\infty$ along ${\mathbb Q}$, we have
\begin{equation}\label{eq:x-limsup}
\limsup_{r\rightarrow\infty}\frac{\sup_{|x|\le r}X(t,x)}{f(r)}=\infty, \quad \text{$P$-a.s.}
\end{equation}

We now suppose that $\alpha>d=1$.
Let $X_1(t,x)$, $X_2(t,x)$ and $X_3(t,x)$ be as in \eqref{eq:x-decomposition}.
For any $n\ge 1$,
\begin{align*}
&\sup_{x\in B(n,n+1)}X_3(t,x)
=\sup_{{x\in B(n,n+1)}}\sum_{i\ge 1:\, \tau_i\le t}p_{t-\tau_i}(x-\eta_i)
\zeta_i {\bf 1}_{\{\zeta_i>p_{t-\tau_i}(0)^{-1}\}}\\
&\ge \sup_{x\in B(n,n+1)}\sum_{i\ge 1:\, \tau_i\le t}p_{t-\tau_i}(x-\eta_i)
\zeta_i {\bf 1}_{\{\zeta_i>p_{t-\tau_i}(0)^{-1}, \eta_i \in \overline{B(n,n+1)}\}}\\
&\ge g(0)\sup\left\{\zeta_i/(t-\tau_i)^{d/\alpha}:
i\ge 1, \, \zeta_i/(t-\tau_i)^{d/\alpha}>1/g(0), \, \eta_i\in \overline{B(n,n+1)}\right\}
=:g(0)Y_n(t).
\end{align*}
Then, for any
$r\ge \max\{1/g(0),1\}$,
\[
P\left(Y_n(t)>r\right)
=1-e^{-|B(n,n+1)|\overline{\tau}(r)}
=1-e^{-2\overline{\tau}(r)}.
\]
See also \eqref{eq:xa-poisson} below for the details.

On the other hand, note again that now we consider $\alpha>d=1$
and assume that $\int_1^\infty \bar\tau (f(r))\,\d r=\infty$.
Then, according to \cite[Lemma 3.4]{CK20},
$$\int_1^\infty \bar\tau (f(r)\vee \bar\tau^{-1}(1/r))\,\d r=\infty,$$
where
$\bar\tau^{-1}(r)=\inf\{s>0 : \bar\tau(s)<r\}$ is the right continuous inverse of $\bar\tau$
and $a\vee b=\max\{a,b\}$.
Thus, for every $K>0$, we have for all large $n\ge 1$,
\begin{align*}
P\left(Y_n(t)>K(f(n)\vee\bar\tau^{-1}(1/n))\right)
&=1-e^{-2\overline{\tau}(K(f(n)\vee\bar\tau^{-1}(1/n)))}\\
&
\asymp\overline{\tau}(K(f(n)\vee\bar\tau^{-1}(1/n)))\\
&\asymp \overline{\tau}(f(n)\vee\bar\tau^{-1}(1/n)).
\end{align*}
Since $\{Y_n(t)\}_{n\ge0}$ are independent and
$\sum_{n=1}^{\infty}
\overline{\tau}(f(n)\vee\bar\tau^{-1}(1/n))=\infty$ by assumption,
we get by the second Borel-Cantelli lemma,
\[
P\left(\sup_{x\in B(n,n+1)}X_3(t,x)> K(f(n)\vee \bar\tau^{-1}(1/n)), \, \text{i.o.}\right)=1.
\]
Following the argument for \eqref{eq:x-limsup}, we obtain
\begin{equation}\label{eq:x_3-limsup}
\limsup_{r\rightarrow\infty}\frac{\sup_{|x|\le r}X_3(t,x)}{f(r)\vee \bar\tau^{-1}(1/r)}=\infty, \quad \text{$P$-a.s.}
\end{equation}

On the other hand, by \eqref{eq:sup-tail} applied to $X_2(t,x)$,
and by
\eqref{eq:x_2-bound} with $A=[0,1]$,
there exists $c_1>0$
such that for any $r>0$,
\[
P\left(\sup_{x\in B(n,n+1)}|X_2(t,x)|>r \right)
\le c_1P\left(\sup_{x\in [0,1]}|X_2(t,x)|>r \right)\\
\le \frac{c_1}{r^2}E\left[\sup_{x\in [0,1]}|X_2(t,x)|^2\right].
\]
Therefore,
\[
P\left(\sup_{x\in B(n,n+1)}|X_2(t,x)|>f(n)\vee\bar\tau^{-1}(1/n)\right)=o\left(1/(f(n)\vee\bar\tau^{-1}(1/n))^2\right).
\]
Note here that,  as  $\bar\tau(\bar\tau^{-1}(r))=r$ for any $r>0$ and $\lim_{s\rightarrow\infty}\bar\tau^{-1}(1/s)=\infty$,
we have by Lemma \ref{lem:tau-finite},
\[
\liminf_{s\rightarrow\infty}\bar\tau^{-1}(1/s)^{\alpha}/s
=\liminf_{s\rightarrow\infty}\bar\tau^{-1}(1/s)^{\alpha}\bar\tau(\bar\tau^{-1}(1/s))>0.
\]
Moreover, since $1<\alpha< 2$,
for every $K>0$, we have by assumption,
\begin{align*}
&
\sum_{n=1}^{\infty}P\left(\sup_{x\in B(n,n+1)}|X_2(t,x)|>K(f(n)\vee\bar\tau^{-1}(1/n))\right)
\preceq \sum_{n=1}^{\infty}
(f(n)\vee\bar\tau^{-1}(1/n))^{-2}\\
&\preceq \int_1^{\infty}
(f(r)\vee\bar\tau^{-1}(1/r))^{-2}\,\d r
\preceq \int_1^{\infty}
(\bar\tau^{-1}(1/r))^{-2}\,\d r\preceq \int_1^{\infty}r^{-2/\alpha}\,\d r<\infty.
\end{align*}
Hence by the Borel-Cantelli lemma
and the argument for \eqref{eq:sup-lim},
$$
\lim_{r\rightarrow\infty}
\frac{\sup_{|x|\le r}|X_2(t,x)|}{f(r)\vee\bar\tau^{-1}(1/r)}=0, \quad \text{$P$-a.s.}
$$
Combining this with \eqref{eq:x_1-bound} and \eqref{eq:x_3-limsup},
we finally obtain
\[
\limsup_{r\rightarrow\infty}\frac{\sup_{|x|\le r}X(t,x)}{f(r)\vee\bar\tau^{-1}(1/r)}=\infty,
\quad \text{$P$-a.s.},
\]
and so
the proof is complete.
\end{proof}

\subsection{Growth order in space of the local supremum on the lattice}

We next reveal the growth order of the local supremum on the lattice of $X(t,x) $
in order to study the attainability of the local supremum.
For $t>0$ and $x\in \R^d$, we define
\[
X_*(t,x)=
\begin{dcases}
\int_{(0,t]\times \R^d\times (0,\infty)}
p_{t-s}(x-y)z{\bf 1}_{\{|x-y|\le 1/2, \, p_{t-s}(x-y)z>1\}}\,\mu(\d s\, \d y\, \d z),
& d<\alpha,\\
\int_{(0,t]\times \R^d\times (0,\infty)}
p_{t-s}(x-y)z{\bf 1}_{\{|x-y|\le 1/2\}}\,\mu(\d s\, \d y\, \d z),
& d\ge \alpha.
\end{dcases}
\]
Let $\eta_0$ be the L\'evy measure on $(0,\infty)$ associated with $X_*(t,x)$.
Then, for any $B\in {\cal B}((0,\infty))$,
\begin{equation}\label{eq:eta-0}
\eta_0(B)=\nu\left(\left\{(s,y,z)\in (0,t]\times \R^d\times (0,\infty)
: |y|\le \frac{1}{2}, \, p_s(y)z\in B\cap \left({\bf 1}_{\{d<\alpha\}},\infty\right)\right\}\right).
\end{equation}
By the definitions
of $\eta$ and $\eta_0$, it is clear that for any $r>0$, $\overline{\eta_0}(r)\le \overline{\eta}(r)$.

We first present existence condition and asymptotic behavior of $\overline{\eta_0}(r)$,
and then relate the tail distribution of $X_*(t,x)$ with $\overline{\eta_0}(r)$.
\begin{lem}\label{lem:eta_0}
\begin{enumerate}
\item[{\rm (i)}] $\overline{\eta_0}(r)<\infty$ for any $r>0$
if and only if \eqref{eq:small-eta-conv} holds, i.e.,
$$\int_{(0,1]}z^{1+\alpha/d}\,\lambda(\d z)<\infty.$$  Under this condition, for any $r>0$,
\[
\overline{\eta_0}(r)
=\frac{\omega_d}{r^{1+\alpha/d}}
\int_0^{tr^{\alpha/d}}\left\{\int_0^{r^{1/d}/2}
\overline{\lambda}\left(\frac{s^{d/\alpha} }{g(l/s^{1/\alpha})}\right)
l^{d-1}\,\d l\right\}\,\d s.
\]
In particular, $r\mapsto \overline{\eta_0}(r)$ is continuous and decreasing on $(0,\infty)$ so that
$$\liminf_{r\rightarrow\infty}r^{1+\alpha/d}\overline{\eta_0}(r)>0.$$
Moreover,
$\overline{\eta_0}(r)$ is of extended regular variation at infinity.
\item[{\rm (ii)}] Under \eqref{eq:small-eta-conv}, for each $t>0$ and $x\in \R^d$,
\begin{equation}\label{eq:x_0-tail}
P(X_*(t,x)>r)\sim\overline{\eta_0}(r), \quad r\rightarrow\infty.
\end{equation}
\end{enumerate}
\end{lem}
\begin{proof}

(1) For $\kappa>0$, define
\[
\Lambda_{\kappa}(r)
=\frac{1}{r^{1+\alpha/d}}\int_{(0,\kappa r]}z^{1+\alpha/d}\,\lambda(\d z)
+\overline{\lambda}(\kappa r),\quad r>0.
\]
We first claim that, under \eqref{eq:small-eta-conv},
for any positive constants $\kappa_1$ and $\kappa_2$,
$\Lambda_{\kappa_1}(r)\asymp \Lambda_{\kappa_2}(r)$ as $r\rightarrow\infty$.
Indeed, let $\kappa_1$ and $\kappa_2$ be positive constants such that $\kappa_1<\kappa_2$.
Since
\[
\overline{\lambda}(\kappa_1 r)
=\int_{(\kappa_1 r,\kappa_2 r]}\lambda(\d z)+\overline{\lambda}(\kappa_2 r)
\le \frac{1}{(\kappa_1 r)^{1+\alpha/d}}\int_{(0,\kappa_2 r]}z^{1+\alpha/d}\,\lambda(\d z)
+\overline{\lambda}(\kappa_2 r),
\]
we have
\[
\Lambda_{\kappa_1}(r)
\le \frac{1}{r^{1+\alpha/d}}\left(1+\frac{1}{\kappa_1^{1+\alpha/d}}\right)
\int_{(0,\kappa_2 r]}z^{1+\alpha/d}\,\lambda(\d z)+\overline{\lambda}(\kappa_2 r)
\le \left(1+\frac{1}{\kappa_1^{1+\alpha/d}}\right)\Lambda_{\kappa_2}(r).
\]
We also see that
$\overline{\lambda}(\kappa_2 r)\le \overline{\lambda}(\kappa_1 r)$ and
\begin{equation*}
\begin{split}
\frac{1}{r^{1+\alpha/d}}\int_{(0,\kappa_2 r]}z^{1+\alpha/d}\,\lambda(\d z)
&=\frac{1}{r^{1+\alpha/d}}\int_{(0,\kappa_1 r]}z^{1+\alpha/d}\,\lambda(\d z)
+\frac{1}{r^{1+\alpha/d}}\int_{(\kappa_1 r,\kappa_2 r]}z^{1+\alpha/d}\,\lambda(\d z)\\
&\le \frac{1}{r^{1+\alpha/d}}\int_{(0,\kappa_1 r]}z^{1+\alpha/d}\,\lambda(\d z)
+\kappa_2^{1+\alpha/d}\overline{\lambda}(\kappa_1 r),
\end{split}
\end{equation*}
which implies that
\[
\Lambda_{\kappa_2}(r)\le (1+\kappa_2^{1+\alpha/d})\Lambda_{\kappa_1}(r).
\]
Therefore, $\Lambda_{\kappa_1}(r)\asymp \Lambda_{\kappa_2}(r)$.

(2)
Let $M=g(0)$ and $c_0\ge \max_{0\le u\le M}g^{-1}(u)^du$. Then by definition,
\begin{equation}\label{eq:eta_0-1}
\begin{split}
\overline{\eta_0}(r)&=\eta_0((r,\infty))\\
&=\omega_d\int_0^t\left\{\int_0^{1/2}
\left(\int_{g(l/s^{1/\alpha})z/s^{d/\alpha}>r}\,\lambda(\d z)\right)
l^{d-1}\,\d l\right\}\,\d s\\
&=\omega_d\int_{(0,\infty)}
\left\{\int_0^t
\left(\int_0^{1/2}l^{d-1}{\bf 1}_{\{l<g^{-1}(s^{d/\alpha }r/z)s^{1/\alpha}, \, s^{d/\alpha}r/z\le M\}}\,\d l\right)
\,\d s\right\}\,\lambda(\d z)\\
&=\frac{\omega_d}{d}\int_{(0,\infty)}
\left[
\int_0^{t\wedge (Mz/r)^{\alpha/d}}
\left\{\frac{1}{2}\wedge \left(g^{-1}\left(\frac{s^{d/\alpha}r}{z}\right)s^{1/\alpha}\right)\right\}^d
\,\d s\right]\,\lambda(\d z)\\
&=\frac{\omega_d}{d}
\Biggl(\int_{(0,r/(2c_0)]}
\left[\int_0^{t\wedge (Mz/r)^{\alpha/d}}
\left\{\frac{1}{2}\wedge \left(g^{-1}\left(\frac{s^{d/\alpha}r}{z}\right)s^{1/\alpha}\right)\right\}^d\, \d s\right]
\,\lambda(\d z)\\
&\qquad\quad+\int_{(r/(2c_0),\infty)}
\left[\int_0^{t\wedge (Mz/r)^{\alpha/d}}
\left\{\frac{1}{2}\wedge \left(g^{-1}\left(\frac{s^{d/\alpha}r}{z}\right)s^{1/\alpha}\right)\right\}^d\, \d s\right]
\,\lambda(\d z)\Biggr)\\
&=\frac{\omega_d}{d}\left({\rm (I)}+{\rm (II)}\right).
\end{split}
\end{equation}

If $z\le r/(2c_0)$ and $s\le (Mz/r)^{\alpha/d}$, then
\[
g^{-1}\left(\frac{s^{d/\alpha}r}{z}\right)^ds^{d/\alpha}
=g^{-1}\left(\frac{s^{d/\alpha}r}{z}\right)^d\frac{s^{d/\alpha}r}{z}\cdot\frac{z}{r}
\le c_0\cdot\frac{1}{2c_0}
=\frac{1}{2},
\]
and so
\begin{align*}
{\rm (I)}
&=\int_{(0,r/(2c_0)]}
\left(\int_0^{t\wedge
(Mz/r)^{\alpha/d}}
g^{-1}\left(\frac{s^{d/\alpha}r}{z}\right)^ds^{d/\alpha}\, \d s\right)\,\lambda(\d z)\\
&=\int_{(0,r/(2c_0)]}
\left(\int_0^{t\wedge
(Mz/(2r))^{\alpha/d}}
g^{-1}\left(\frac{s^{d/\alpha}r}{z}\right)^ds^{d/\alpha}\, \d s\right)\,\lambda(\d z)\\
&\quad+\int_{(0,r/(2c_0)]}
\left(\int_{t\wedge
(Mz/(2r))^{\alpha/d}}^{t\wedge
(Mz/r)^{\alpha/d}}
g^{-1}\left(\frac{s^{d/\alpha}r}{z}\right)^ds^{d/\alpha}\, \d s\right)\,\lambda(\d z)\\
&={\rm (I)}_1+{\rm (I)}_2.
\end{align*}
Then by
applying \eqref{eq:g^{-1}-asymp}, we have
\begin{align*}
{\rm (I)}_1
&\asymp \int_{(0,r/(2c_0)]}
\left(\int_0^{(Mz/(2r))^{\alpha/d}}
\left(\frac{z}{s^{d/\alpha}r}\right)^{d/(d+\alpha)}s^{d/\alpha}\, \d s\right)\,\lambda(\d z)\\
&=\frac{d+\alpha}{2d+\alpha}\left(\frac{M}{2}\right)^{\alpha/d+\alpha/(d+\alpha)}
\frac{1}{r^{1+\alpha/d}}\int_{(0,r/(2c_0)]}z^{1+\alpha/d}\,\lambda(\d z)
\end{align*}
and
\begin{align*}
{\rm (I)}_2&\le M^d\int_{(0,r/(2c_0)]}
\left(\int_0^{(Mz/r
)^{\alpha/d}}
s^{d/\alpha}\, \d s\right)\,\lambda(\d z)\\
&=\frac{dM^d}{d+\alpha}\left(\frac{M}{2}\right)^{1+\alpha/d}
\frac{1}{r^{1+\alpha/d}}
\int_{(0,r/(2c_0)]}z^{1+\alpha/d}\,\lambda(\d z).
\end{align*}
Thus
\begin{equation}\label{eq:eta_0-1-asymp}
{\rm (I)}\asymp \frac{1}{r^{1+\alpha/d}}\int_{(0,r/(2c_0)]}z^{1+\alpha/d}\,\lambda(\d z).
\end{equation}

We next take $c_1>0$ so small that
$c_1\le (M/(2t^{d/\alpha}))\wedge (2c_0)$
and
\begin{align*}
{\rm (II)}
&=\int_{(r/(2c_0), r/c_1]}
\left[\int_0^{t\wedge (Mz/r)^{\alpha/d}}
\left\{\frac{1}{2}\wedge \left(g^{-1}\left(\frac{s^{d/\alpha}r}{z}\right)s^{1/\alpha}\right)\right\}^d\, \d s\right]
\,\lambda(\d z)\\
&\quad +\int_{(r/c_1,\infty)}
\left[\int_0^{t\wedge (Mz/r)^{\alpha/d}}
\left\{\frac{1}{2}\wedge \left(g^{-1}\left(\frac{s^{d/\alpha}r}{z}\right)s^{1/\alpha}\right)\right\}^d\, \d s\right]
\,\lambda(\d z)\\
&={\rm (II)}_1+{\rm (II)}_2.
\end{align*}
Then
\[
{\rm (II)}_1\preceq
\int_{(r/(2c_0), \infty)
}\,\lambda(\d z)
=
\overline{\lambda}\left(\frac{r}{2c_0}\right).
\]
If we take $c_2>0$
so large that $c_2>M/(c_1 t^{d/\alpha})$, then $t>(M/(c_1c_2))^{\alpha/d}$ and thus
\begin{align*}
{\rm (II)}_1
&\ge
\int_{(r/(2c_0), r/c_1]}
\left[\int_{(Mz/(2c_2r))^{\alpha/d}}^{(Mz/(c_2r))^{\alpha/d}}
\left\{\frac{1}{2}\wedge \left(g^{-1}\left(\frac{s^{d/\alpha}r}{z}\right)s^{1/\alpha}\right)\right\}^d\, \d s\right]
\,\lambda(\d z)\\
&\asymp\frac{1}{r^{\alpha/d}}\int_{(r/(2c_0), r/c_1]}z^{\alpha/d}
\,\lambda(\d z)
\asymp
\frac{1}{r^{1+\alpha/d}}\int_{(r/(2c_0), r/c_1]}z^{1+\alpha/d}\,\lambda(\d z).
\end{align*}
If $0\le s\le t$ and $z>r/c_1$, then by \eqref{eq:g^{-1}-asymp},
\[
g^{-1}\left(\frac{s^{d/\alpha}r}{z}\right)^ds^{d/\alpha}
\ge g^{-1}(c_1s^{d/\alpha})^ds^{d/\alpha}
\asymp \frac{1}{(s^{d/\alpha})^{d/(d+\alpha)}}\cdot s^{d/\alpha}
=s^{d/(d+\alpha)},
\]
which yields
\begin{align*}
{\rm (II)}_2
&=\int_{(r/c_1,\infty)}
\left[\int_0^t
\left\{\frac{1}{2}\wedge \left(g^{-1}\left(\frac{s^{d/\alpha}r}{z}\right)s^{1/\alpha}\right)\right\}^d\, \d s\right]
\,\lambda(\d z)
\asymp
\overline{\lambda}\left(\frac{r}{c_1}\right).
\end{align*}
By the argument above, we get
\[
\frac{1}{r^{1+\alpha/d}}\int_{(r/(2c_0), r/c_1]}z^{1+\alpha/d}\,\lambda(\d z)
+\overline{\lambda}\left(\frac{r}{c_1}\right)
\preceq {\rm (II)}
\preceq
\overline{\lambda}\left(\frac{r}{2c_0}\right).
\]
Combining this with \eqref{eq:eta_0-1} and \eqref{eq:eta_0-1-asymp},
we have
\begin{equation}\label{e:llffpp}
\Lambda_{1/c_1}(r)
\preceq \overline{\eta_0}(r)
\preceq \Lambda_{1/(2c_0)}(r)
.\end{equation}
Therefore,
by the assertion in (1), $\overline{\eta_0}(r)<\infty$ for any $r>0$
if and only if \eqref{eq:small-eta-conv} holds.
Moreover, under this condition, for each $\kappa>0$,
$\overline{\eta_0}(r)\asymp \Lambda_{\kappa}(r)$ as $r\rightarrow\infty$
and $\liminf_{r\rightarrow\infty}r^{1+\alpha/d}\overline{\eta_0}(r)>0$.

(3) By the definition of $p_s(y)$, we
obtain
\begin{align*}
\overline{\eta_0}(r)=\eta_0((r,\infty))
&=\int_0^t\left\{\int_{|y|\le 1/2}
\overline{\lambda}\left(\frac{r}{p_s(y)}\right)
\,\d y\right\}\,\d s\\
&=\omega_d\int_0^t\left\{\int_0^{1/2}
\overline{\lambda}\left(\frac{s^{d/\alpha} r}{g(l/s^{1/\alpha})}\right)
l^{d-1}\,\d l\right\}\,\d s\\
&=\frac{\omega_d}{r^{1+\alpha/d}}
\int_0^{tr^{\alpha/d}}\left\{\int_0^{r^{1/d}/2}
\overline{\lambda}\left(\frac{u^{d/\alpha} }{g(v/u^{1/\alpha})}\right)
v^{d-1}\,\d v\right\}\,\d u.
\end{align*}
At the last equation, we used the change of
variables
formula
with $u=sr^{\alpha/d}$ and $v=lr^{1/d}$.
Hence $r\mapsto\overline{\eta_0}(r)$ is continuous on $(0,\infty)$.

(4) For $s>0$ and $l>0$, we define
\[
\Lambda(s,l)=\overline{\lambda}\left(\frac{s^{d/\alpha} }{g(l/s^{1/\alpha})}\right).
\]
We also define
\[
\Theta(r)=\int_0^{tr^{\alpha/d}}
\left(\int_0^{r^{1/d}/2}\Lambda(s,l)l^{d-1}\,\d l\right)\,\d s,
\]
and so
$\overline{\eta_0}(r)=\omega_d\Theta(r)/r^{1+\alpha/d}$.
Then
\begin{equation}\label{eq:eta-deri-1}
\frac{\overline{\eta_0}'(r)}{\overline{\eta_0}(r)}=\frac{\Theta'(r)}{\Theta(r)}
-\left(1+\frac{\alpha}{d}\right)\frac{1}{r}
\end{equation}
and
\begin{equation}\label{eq:theta-deri}
\Theta'(r)
=\frac{\alpha}{d}tr^{\alpha/d-1}\int_0^{r^{1/d}/2}\Lambda(tr^{\alpha/d},l)l^{d-1}\,\d l
+\frac{1}{2^d d}\int_0^{tr^{\alpha/d}}\Lambda\left(s,\frac{r^{1/d}}{2}\right)\,\d s.
\end{equation}

Since $\Lambda(s,l)$ is decreasing in $l$, we have
\[
\Theta(r)
\ge  \int_0^{tr^{\alpha/d}}\left(\int_0^{r^{1/d}/2}l^{d-1}\,\d l\right)\Lambda\left(s, \frac{r^{1/d}}{2}\right)\,\d s
=\frac{r}{2^d d}\int_0^{tr^{\alpha/d}}\Lambda\left(s, \frac{r^{1/d}}{2}\right)\,\d s.
\]
As
\[
\Lambda(s,s^{1/\alpha}u)=
\overline{\lambda}\left(\frac{s^{d/\alpha}}{g(u)}\right)
\]
is decreasing in $s$,
we also get by the change of variables formula with $l=s^{1/\alpha}u$ and
$v=t^{1/\alpha}r^{1/d}u$,
\begin{equation*}
\begin{split}
\Theta(r)
&=\int_0^{tr^{\alpha/d}}
\left(\int_0^{r^{1/d}/(2s^{1/\alpha})}\Lambda(s,s^{1/\alpha}u)u^{d-1}\,\d u\right)s^{d/\alpha}\,\d s\\
&\ge \int_0^{tr^{\alpha/d}}
\left(\int_0^{r^{1/d}/(2s^{1/\alpha})}\Lambda(tr^{\alpha/d},t^{1/\alpha} r^{1/d}u)u^{d-1}\,\d u\right)
s^{d/\alpha}\,\d s\\
&\ge \int_0^{tr^{\alpha/d}}s^{d/\alpha}\,\d s
\int_0^{1/(2t^{1/\alpha})}\Lambda(tr^{\alpha/d},t^{1/\alpha} r^{1/d}u)u^{d-1}\,\d u\\
&=\frac{tr^{\alpha/d}}{1+d/\alpha}
\int_0^{r^{1/d}/2}\Lambda(tr^{\alpha/d},v)v^{d-1}\,\d v.
\end{split}
\end{equation*}
Thus
\[
\Theta(r)\ge
\left(\frac{r}{2^d d}\int_0^{tr^{\alpha/d}}\Lambda\left(s, \frac{r^{1/d}}{2}\right)\,\d s\right)
\vee \left(\frac{tr^{\alpha/d}}{1+d/\alpha}
\int_0^{r^{1/d}/2}\Lambda(tr^{\alpha/d},v)v^{d-1}\,\d v\right).
\]
Then by \eqref{eq:theta-deri}, we obtain
\[
0\le \frac{\Theta'(r)}{\Theta(r)}\le
\frac{1}{r}\left(\frac{\alpha}{d}\left(1+\frac{d}{\alpha}\right)+1\right)
=\frac{1}{r}\left(2+\frac{\alpha}{d}\right).
\]
Hence by \eqref{eq:eta-deri-1},
\[
-\left(1+\frac{\alpha}{d}\right)\le r \frac{\overline{\eta_0}'
(r)}{\overline{\eta_0}(r)}
\le \left(2+\frac{\alpha}{d}\right)-\left(1+\frac{\alpha}{d}\right)=1.
\]
With these two inequalities at hand,
we can follow the proofs of Lemma \ref{lem:eta-express} and
Theorem \ref{thm:eta-tail} to prove
the last assertion in (i) as well as the assertion (ii).
\end{proof}

We finally determine the growth rate of the local supremum of $X(t,x)$
on the lattice.
\begin{thm}\label{thm:local-growth}
Suppose that \eqref{eq:small-conv-0} and \eqref{eq:big-conv} hold.
Let $f:(0,\infty)\to (0,\infty)$ be nondecreasing.
\begin{enumerate}
\item[{\rm (i)}]
Let $\bar\eta(r)=\eta((r,\infty))$ for all $r>0$, where $\eta$ is defined by \eqref{eq:eta}.
If
$\int_1^{\infty}r^{d-1}\overline{\eta}(f(r))\,\d r<\infty$,
then
\[
\lim_{r\rightarrow\infty}\frac{\sup_{x\in {\mathbb Z}^d, \, |x|\le r}X(t,x)}{f(r)}=0, \quad \text{$P$-a.s.}
\]

\item[{\rm (ii)}]
Assume further either of the next conditions{\rm :}
\begin{enumerate}
\item[{\rm (a)}] For $\alpha\le d$,
$\int_{(0,1]}z\,\lambda(\d z)<\infty${\rm ;}
\item[{\rm (b)}] For $\alpha>d=1$,
there exists $\gamma>1/(1+\alpha)$
such that
$\int_{(1,\infty)}z^{\gamma}\,\lambda(\d z)<\infty$.
\end{enumerate}
Let $\bar\eta_0(r)=\eta_0((r,\infty))$ for all $r>0$,
where $\eta_0$ is defined by \eqref{eq:eta-0}. If $\int_1^{\infty}r^{d-1}\overline{\eta}_0(f(r))\,\d r=\infty$, then
\[
\limsup_{r\rightarrow\infty}\frac{\sup_{x\in {\mathbb Z}^d, \, |x|\le r}X(t,x)}{f(r)}=\infty, \quad \text{$P$-a.s.}
\]
\end{enumerate}
\end{thm}

\begin{proof} (1)
We first prove (i).
Following \eqref{eq:sup-tail},
we see
by Theorem \ref{thm:eta-tail}
that for any $K>0$, there exist positive constants $c_1$, $c_2$, $c_3$ such that
\[
P\left(\sup_{z\in {\mathbb Z}^d\cap B(n,n+1)}X(t,x)>\frac{f(n)}{K}\right)
\le c_1n^{d-1}P\left(X(t,0)>\frac{f(n)}{K}\right)
\le c_2n^{d-1}\overline{\eta}(f(n))
\]
and so
\[
\sum_{n=1}^{\infty}P\left(\sup_{z\in {\mathbb Z}^d\cap B(n,n+1)}X(t,x)>\frac{f(n)}{K}\right)
\le c_2\sum_{n=1}^{\infty}n^{d-1}\overline{\eta}(f(n))\le
c_3\int_1^{\infty}r^{d-1}\overline{\eta}(f(r))\,\d r.
\]
Hence by the same argument for \eqref{eq:sup-lim}, the proof of (i) is complete.

(2) We next prove (ii) under the condition (a).
We set as in \eqref{eq:decomp-1},
\[
X(t,x)=m_0t
+\int_{(0,t]\times \R^d\times (0,\infty)}p_{t-s}(x-y)z\,\mu({\rm d}s\,{\rm d}y\,{\rm d}z)
:=m_0 t
+X''_2(t,x).
\]

By definition,
\begin{equation}\label{eq:x_2-lower}
X''_2(t,x)\ge X_*(t,x), \quad \text{$P$-a.s.}
\end{equation}
Since $\overline{\eta_0}$ is of extended regular variation at infinity by Lemma \ref{lem:eta_0}
(i)
and $\{X_*(t,x)\}_{x\in {\mathbb Z}^d}$ are identically distributed,
we have by \eqref{eq:x_0-tail},
\begin{equation*}
\begin{split}
\sum_{n=1}^{\infty}\sum_{x
\in {\mathbb Z}^d\cap B(n,n+1)}P(X_*(t,x)>Kf(n+1))
&\ge c_1\sum_{n=1}^{\infty} n^{d-1}\overline{\eta_0}(
f(n+1))\\
&\ge c_2\int_1^{\infty}r^{d-1}\overline{\eta_0}(f(r))\,\d r=\infty.
\end{split}
\end{equation*}
Moreover,
as $\{X_*(t,x)\}_{x\in {\mathbb Z}^d}$ are independent, we get by the second Borel-Cantelli lemma,
\[
P\left(\begin{varwidth}[c]{20cm}
\text{there exists a sequence $\{(n_l,x_l)\}_{l\ge 1}\subset {\mathbb N}\times {\mathbb Z}^d$}
\text{such that
$n_l\rightarrow\infty$} as $l\rightarrow\infty$, \\
\text{$x_l\in B(n_l,n_{l+1})$ and
$X_*
(t,x_l)>Kf(n_l+1)$ for all $l\ge 1$}
\end{varwidth}\right)=1.
\]
Namely,
\[
P\left(\begin{varwidth}[c]{20cm}
\text{there exists a sequence $\{n_l\}_{l=1}^{\infty}$ such that $n_l\rightarrow\infty$ as $l\rightarrow\infty$ and } \\
\text{$\sup_{y\in {\mathbb Z}^d, \,
n_l\le |y|<n_l+1
} X_*
(t,x)>Kf(n_l+1)$ for all $l\ge 1$}
\end{varwidth}\right)
=1.
\]
This yields for any $K>0$,
\[
\limsup_{r\rightarrow\infty}\frac{\sup_{x\in {\mathbb Z}^d, |x|\le r}X_*(t,x)}{f(r)}>K, \quad \text{$P$-a.s.}
\]
and therefore,
\[
\limsup_{r\rightarrow\infty}\frac{\sup_{x\in {\mathbb Z}^d, |x|\le r}X_*(t,x)}{f(r)}=\infty, \quad \text{$P$-a.s.}
\]
By \eqref{eq:x_2-lower}, we further obtain
\begin{equation}\label{eq:limsup-x_2}
\limsup_{r\rightarrow\infty}\frac{\sup_{x\in {\mathbb Z}^d, |x|\le r}X''_2(t,x)}{f(r)}=\infty, \quad \text{$P$-a.s.}
\end{equation}
Hence
the proof is complete under the condition (a).

(3) We finally prove (ii) under the condition (b).
In particular, we consider $\alpha>d=1$ and assume that $\int_1^{\infty}\overline{\eta}_0(f(r))\,\d r=\infty$.
Then, according to \cite[Lemma 3.4]{CK20},
$$\int_1^\infty \overline{\eta}_0(f(r)\vee \overline{\eta}_0^{-1}(1/r))\,\d r=\infty,$$
where $a\vee b=\max\{a,b\}$ and $\overline{\eta}_0^{-1}(r)$ is
the right continuous inverse
of the function $\overline{\eta}_0(r)$.
Let
\begin{equation*}
\begin{split}
X(t,x)
&=
m_1
+\int_{(0,t]\times \R\times (0,\infty)}p_{t-s}(x-y)z{\bf 1}_{\{p_{t-s}(x-y)z\le 1\}}\,(\mu-\nu)(\d s\,\d y\, \d z)\\
&\quad+\int_{(0,t]\times \R\times (0,\infty)}p_{t-s}(x-y)z{\bf 1}_{\{p_{t-s}(x-y)z>1\}}\,\mu(\d s\,\d y\, \d z)\\
&=
m_1
+X_1''(t,x)+X_2''(t,x)
\end{split}
\end{equation*}
with
\[
m_1=
mt
+\int_{(0,t]\times \R\times (0,\infty)}p_s(y)z
\left({\bf 1}_{\{p_s(y)z\le 1\}}-{\bf 1}_{\{z\le 1\}}\right)\,\d s\,\d y\, \lambda(\d z)
\]
and
\[
E\left[\exp\left(i\theta X(t,x)\right)\right]
=\exp\left(i\theta
m_1
+\int_0^{\infty}\left(e^{i\theta u}-1-i\theta u{\bf 1}_{\{0\le u\le 1\}}\right)\,\eta(\d u)\right),
\quad \theta\in \R.
\]
For $X_2''(t,x)$, we can follow the proof of
(2) to verify that
\begin{equation}\label{eq:limsup-x_2--}
\limsup_{r\rightarrow\infty}\frac{\sup_{x\in {\mathbb Z}^d, |x|\le r}X''_2(t,x)}{f(r)\vee \overline{\eta}_0^{-1}(1/r)}=\infty, \quad \text{$P$-a.s.}
\end{equation}.

We turn to the estimate of $X_1''(t,x)$.
By \cite[Theorem 1]{MR14} with $p=4$ and $\alpha=2$,
\begin{align*}
&E\left[\sup_{x\in [0,1]\cap {\mathbb Z}}|X''_1(t,x)|^4\right]
\le 2 E\left[|X''_1(t,0)|^4\right] \\
&=2E\left[
\left(\int_{(0,t]\times \R\times (0,\infty)}p_{t-s}(y)z{\bf 1}_{\{p_{t-s}(y)z\le 1\}}
\,(\mu-\nu)(\d s\,\d y\,\d z)\right)^4\right]\\
&\le
c_1
\Bigg[\left(\int_{(0,t]\times \R\times (0,\infty)}p_{t-s}(y)^2 z^2{\bf 1}_{\{p_{t-s}(y)z\le 1\}}
\, \d s\,\d y\,\lambda(\d z)\right)^2\\
&\qquad+ \int_{(0,t]\times \R\times (0,\infty)}p_{t-s}(y)^4 z^4{\bf 1}_{\{p_{t-s}(y)z\le 1\}}
\, \d s\,\d y\,\lambda(\d z) \Bigg]\\
&\le c_1\left({\rm (I)}^2+{\rm (I)}\right),
\end{align*}
where
\begin{align*} {\rm (I)}=& \int_{(0,t]\times \R\times (0,1]}p_s(y)^2 z^2{\bf 1}_{\{p_{s
}(y)z\le 1\}}
\, \d s\,\d y\,\lambda(\d z)\\
&+\int_{(0,t]\times \R\times (1,\infty)}p_s(y)^2 z^2{\bf 1}_{\{p_{s
}(y)z\le 1\}}
\, \d s\,\d y\,\lambda(\d z)\\
&={\rm (I)}_1+{\rm (I)}_2.
\end{align*}

Furthermore,
\begin{align*}
{\rm (I)}_1
&\le
\int_{(0,t]\times \R\times (0,1]}g(0)s^{-1/\alpha} p_s(y) z^2{\bf 1}_{\{p_s(y)z\le 1\}}
\, \d s\,\d y\,\lambda(\d z)\\
&\le
c_2
t^{1-1/\alpha}\int_{(0,1]}z^2\,\lambda(\d z).
\end{align*}
On the other hand,
we have for
all $\gamma\in (1/(1+\alpha),1)$,
\begin{align*}
{\rm (I)}_2&=
\int_{(0,t]\times \R\times(1,\infty)}
(p_s(y) z)^{\gamma}(p_s(y)z)^{2-\gamma}{\bf 1}_{\{p_{s
}(y)z\le 1\}}
\, \d s\,\d y\,\lambda(\d z)\\
&\le
c_1
\int_{(0,t]\times \R}p_s(y)^{\gamma}\,\d s\,\d y\int_{(1,\infty)}z^{\gamma}\,\lambda(\d z)
\asymp t^{1+(1-\gamma)/\alpha},
\end{align*}
where in the last inequality we used the fact that for all $\gamma\in (1/(1+\alpha),1)$,
\begin{align*}
&\int_{(0,t]\times \R}p_s(y)^{\gamma}\,\d s\,\d y
\asymp \int_{(0,t]\times \R}\left(\frac{1}{s^{1/\alpha}}\wedge \frac{s}{|y|^{1+\alpha}}\right)^{\gamma}\,\d s\,\d y\\
&=\int_0^t\left(\int_{|y|<s^{1/\alpha}}s^{-\gamma/\alpha}\,\d y\right)\,\d s
+\int_0^t\left\{\int_{|y|\ge s^{1/\alpha}}\left(\frac{s}{|y|^{1+\alpha}}\right)^{\gamma}\,\d y\right\}\,\d s\\
&=2\int_0^t s^{(1-\gamma)/\alpha}\,\d s
+\frac{2}{\gamma(1+\alpha)-1}
\int_0^t s^{\gamma} s^{(1-\gamma(1+\alpha))/\alpha}\,\d s\\
&  \preceq t^{1+(1-\gamma)/\alpha}.
\end{align*}
We thus have
\[
E\left[\sup_{x\in [0,1]\cap {\mathbb Z}}|X''_1(t,x)|^4\right]<\infty.
\]
Furthermore, according to Lemma \ref{lem:eta_0} (i),
$$\liminf_{r\rightarrow\infty}
r^{1+\alpha}\overline{\eta_0}(r)>0$$
and so
$r^{1/(1+\alpha)}\preceq \overline{\eta_0}^{-1}(1/r)$ for all $r>1$.
This implies that
$$\int_1^\infty
(f(r)\vee \overline{\eta}_0^{-1}(1/r))^{-4}
\,\d r
\le \int_1^\infty (\overline{\eta}_0^{-1}(1/r))^{-4}\,\d r\preceq \int_1^\infty r^{-4/(1+\alpha)}\,\d r<\infty.$$
Therefore,
\begin{align*}
&\sum_{n=1}^{\infty}P\left(
\sup_{x\in [n,n+1]\cap {\mathbb Z}}
|X''_1(t,x)|>f(n)\vee \overline{\eta}_0^{-1}(1/n)\right)\\
&=\sum_{n=1}^{\infty}P\left(
\sup_{x\in [0,1]\cap {\mathbb Z}}
|X''_1(t,x)|>f(n)\vee \overline{\eta}_0^{-1}(1/n)\right)\\
&\le E\left[
\sup_{x\in [0,1]\cap {\mathbb Z}}
|X''_1(t,x)|^4
\right]
\sum_{n=1}^{\infty}\frac{1}{(f(n)\vee \overline{\eta}_0^{-1}(1/n))^4}<\infty.
\end{align*}
Following the proof of \eqref{eq:sup-lim}, we have
\[
\lim_{r\rightarrow\infty}\frac{
\sup_{x\in {\mathbb Z}, \, |x|\le r}
|X_1''(t,x)|
}{f(r)\vee \overline{\eta}_0^{-1}(1/r)}=0, \quad \text{$P$-a.s.}
\]
Combining this with
\eqref{eq:limsup-x_2--},
we obtain that
\[
\limsup_{r\rightarrow\infty}
\frac{\sup_{x\in {\mathbb Z}, \, |x|\le r}X(t,x)}{f(r)}
\ge \limsup
_{r\rightarrow\infty}\frac{
\sup_{x\in {\mathbb Z}, \, |x|\le r}
X(t,x)}{f(r)\vee \overline{\eta}_0^{-1}(1/r)}=
\infty,
\quad \text{$P$-a.s.,}
\] which completes the proof of (ii).
\end{proof}

As mentioned above, $\overline{\eta_0}(r)\le \overline{\eta}(r)$ for any $r>0$.
The next lemma provides more precise asymptotic relation between
$\overline{\eta_0}(r)$ and $\overline{\eta}(r)$.
Recall that, according to Lemma \ref{lem:eta-express} (i),
$\overline{\eta}(r)<\infty$ for any $r>0$ if and only if
\eqref{eq:small-eta-conv} and \eqref{eq:big-conv} hold.
\begin{lem}\label{lem:eta-eta_0}
Suppose that \eqref{eq:small-eta-conv} and \eqref{eq:big-conv} hold.
Then the following statements hold.
\begin{enumerate}
\item Suppose
that either
of the following conditions holds{\rm :}
\begin{enumerate}
\item[{\rm (a)}]
$\displaystyle\int_{(1,
\infty)}z^{1+\alpha/d}\,\lambda(\d z)<\infty${\rm ;}
\item[{\rm (b)}]
There exist constants $\delta>d/(d+\alpha)$,
$c>0$ and $M>0$ such that
\begin{equation}\label{eq:lambda-ratio}
\frac{\overline{\lambda}(ry)}{\overline{\lambda}(r)}\le cy^{-\delta}, \quad r>M, \, y>M.
\end{equation}
\end{enumerate}
Then $\overline{\eta_0}(r)\asymp \overline{\eta}(r)$ as $r\rightarrow\infty$.

\item[{\rm (ii)}]
Suppose that ${\rm supp}[\lambda]\subset [1,\infty)$ and
$\overline{\lambda}(r)
\asymp
l(r)/r^{d/(d+\alpha)}$ for $r\ge 1$, where $l(r)$ is a
slowly varying function at infinity with $\int_1^{\infty}l(r)/r\,\d r<\infty$.
Then $\overline{\eta_0}(r)\asymp \overline{\lambda}(r)$
and $\overline{\eta_0}(r)=o(\overline{\eta}(r))$ as $r\rightarrow\infty$.
\end{enumerate}
\end{lem}

\begin{proof}
(1) We first prove (i) under the condition (a).
By Lemma \ref{lem:eta-express-2}(i) and Lemma \ref{lem:eta_0}(i)
\[
\overline{\eta}(r)\asymp \frac{1}{r^{1+\alpha/d}} \preceq \overline{\eta_0}(r)\le \overline{\eta}(r).
\]
Hence the proof is complete.

(2) We next prove (i) under the condition (b).
Let
$c_1
\ge 2t^{d/\alpha}/M$.
Then by \eqref{eq:eta-express},
\begin{align*}
\overline{\eta}(r)
&=\omega_d \Biggl\{\int_{(0,c_1r]}
\left(\int_0^{t\wedge (Mz/r)^{\alpha/d}}
g^{-1}\left(\frac{s^{d/\alpha}r}{z}\right)^d s^{d/\alpha}\,\d s\right)\,\lambda(\d z)\\
&\qquad\quad +\int_{(c_1r,\infty)}
\left(\int_0^{t\wedge (Mz/r)^{\alpha/d}}
g^{-1}\left(\frac{s^{d/\alpha}r}{z}\right)^d s^{d/\alpha}\,\d s\right)\,\lambda(\d z)\Biggr\}\\
&=\omega_d({\rm (I)}+{\rm (II)}).
\end{align*}

According to the argument for
\eqref{eq:eta_0-1-asymp},
\begin{equation}\label{eq:int-small-1}
{\rm (I)}
\preceq \frac{1}{r^{1+\alpha/d}}\int_{(0,c_1r]}z^{1+\alpha/d}\,\lambda(\d z).
\end{equation}
We obtain
\begin{equation}\label{eq:int-large-1}
\begin{split}
{\rm (II)}
&=\int_{(c_1 r,\infty)}
\left(\int_0^t
g^{-1}\left(\frac{s^{d/\alpha}r}{z}\right)^d s^{d/\alpha}\,\d s\right)\,\lambda(\d z)\\
&\asymp \int_{(c_1r,\infty)}
\left(\int_0^t
\left(\frac{z}{s^{d/\alpha}r}\right)^{d/(d+\alpha)} s^{d/\alpha}
\,\d s\right)\,\lambda(\d z)\\
&=\frac{d+\alpha}{2d+\alpha}\frac{t^{1+d/(d+\alpha)}}{r^{d/(d+\alpha)}}
\int_{(c_1r,\infty)}z^{d/(d+\alpha)}\,\lambda(\d z).
\end{split}
\end{equation}
Since \[
z^{d/(d+\alpha)}
=\frac{d}{d+\alpha}\int_0^z\frac{1}{u^{\alpha/(d+\alpha)}}\,\d u,
\]
we have by the Fubini theorem,
\begin{align*}
&\int_{(c_1r,\infty)}z^{d/(d+\alpha)}\,\lambda(\d z)
=\frac{d}{d+\alpha}\int_{(c_1r,\infty)}\left(\int_0^z \frac{1}{u^{\alpha/(d+\alpha)}}\,\d u\right)\,\lambda(\d z)\\
&=\frac{d}{d+\alpha}\int_{c_1r}^{\infty}\frac{1}{u^{\alpha/(d+\alpha)}}\left(\int_{(u,\infty)}\,\lambda(\d z)\right)\,\d u
+\frac{d}{d+\alpha}\int_0^{c_1r}\frac{1}{u^{\alpha/(d+\alpha)}}\left(\int_{(c_1r,\infty)}\,\lambda(\d z)\right)\,\d u\\
&=\frac{d}{d+\alpha}\int_{c_1r}^{\infty}\frac{\overline{\lambda}(u)}{u^{\alpha/(d+\alpha)}}\,\d u
+\left(c_1r\right)^{d/(d+\alpha)}\overline{\lambda}\left(c_1r\right).
\end{align*}
Then by \eqref{eq:lambda-ratio},
\begin{align*}
\int_{c_1r}^{\infty}\frac{\overline{\lambda}(u)}{u^{\alpha/(d+\alpha)}}\,\d u
&\preceq
\overline{\lambda}\left(c_1r\right)\int_{c_1r}^{\infty}\frac{1}{u^{\alpha/(d+\alpha)}}
\left(\frac{c_1r}{u}\right)^{\delta}\,\d u\\
&=
c_1^{d/(d+\alpha)}
\left(\delta-\frac{d}{d+\alpha}\right)^{-1}r^{d/(d+\alpha)}\overline{\lambda}\left(c_1r\right),
\end{align*}
which yields
\[
\int_{(c_1r,\infty)}z^{d/(d+\alpha)}\,\lambda(\d z)
  \preceq r^{d/(d+\alpha)}\overline{\lambda}\left(c_1r\right).
\]
Hence by \eqref{eq:int-large-1},
\[
{\rm (II)}\preceq \overline{\lambda}\left(c_1r\right).
\]
Combining this with \eqref{eq:int-small-1},
we obtain by
\eqref{e:llffpp},
\begin{equation}\label{eq:eta-upper}
\overline{\eta}(r)
\preceq \frac{1}{r^{1+\alpha/d}}\int_{(0,c_1r]}z^{1+\alpha/d}\,\lambda(\d z)
+\overline{\lambda}\left(c_1r\right)\asymp
\bar\eta_0(r).
\end{equation}

(3) We finally prove (ii). Under the condition (ii),
Lemma \ref{lem:eta-express-2} (ii-b) and \cite[p.\ 27, Proposition 1.5.9b]{BGT89}
imply that
\[
\lim_{r\rightarrow\infty}\frac{\overline{\eta}(r)}{\overline{\lambda}(r)}
=\lim_{r\rightarrow\infty}\frac{1}{l(r)}\int_r^{\infty}\frac{l(u)}{u}\,\d u=\infty,
\]
whence $\overline{\lambda}(r)=o(\overline{\eta}(r))$ as $r\rightarrow\infty$.

On the other hand,
it follows by \cite[Proposition 1.5.8]{BGT89} that as $r\rightarrow\infty$,
\[
\int_1^r u^{\alpha/d}\overline{\lambda}(u)\,\d u
\asymp\int_1^r\frac{l(u)}{u^{d/(d+\alpha)-\alpha/d}}\,\d u
\asymp r^{\alpha/d+\alpha/(d+\alpha)}l(r)
\asymp r^{1+\alpha/d}\overline{\lambda}(r).
\]
Then as $r\rightarrow\infty$,
\begin{align*}
\int_{(1,r]}z^{1+\alpha/d}\,\lambda(\d z)
&
\preceq \int_{(1,r]}\left(\int_1^z u^{\alpha/d}\,\d u\right)\,\lambda(\d z)
\preceq\int_1^r \left(\int_{(u,r]} \,\lambda(\d z)\right)\,u^{\alpha/d}\,\d u\\
&\preceq\int_1^r u^{\alpha/d}\overline{\lambda}(u)\,\d u
\preceq r^{1+\alpha/d}\overline{\lambda}(r).
\end{align*}
Furthermore, since ${\rm supp}[\lambda]\subset[1,\infty)$,
we have $\int_{(0,1]}z^{1+\alpha/d}\,\lambda(\d z)=0$.
Hence by
\eqref{e:llffpp},
$\overline{\eta_0}(r)\asymp \overline{\lambda}(r)$ as
$r\rightarrow\infty$.
By the argument above, the proof is complete.
\end{proof}

\subsection{Examples}\label{S:example}
In this subsection, we calculate the growth order of the local supremum of $X(t,x)$
for a large class of
concrete L\'evy
noises.
Let $k(z)$ be a nonnegative Borel measurable function on
$(0,\infty)$ such that
for some $\kappa\in (0,\alpha/d)$ and $\beta>d/(d+\alpha)$,
\[
k(z)\preceq  \frac{1}{z^{1+\kappa}},\quad 0<z<1
\]
and
\begin{equation}\label{eq:k-large}
k(z)=\frac{l(z)}{z^{1+\beta}}, \quad z\ge 1
\end{equation}
with $l(z)$ being a slowly varying function at infinity.
Assume that the L\'evy measure $\lambda(\d z)$
associated with the L\'evy space-time white noise $\Lambda(t,x)$
in the fractional stochastic heat equation \eqref{eq:fractional-she} is given by $\lambda(\d z)=k(z)\,\d z$.
Define
\[
L(r)=r^{\beta}\int_r^{\infty}\frac{l(z)}{z^{1+\beta}}\,\d z.\]
Then
\[
\bar{\lambda}(r)=\int_r^{\infty}\frac{l(z)}{z^{1+\beta}}\,\d z
=\frac{1}{r^{\beta}}r^{\beta}\int_r^{\infty}\frac{l(z)}{z^{1+\beta}}\,\d z=\frac{L(r)}{r^{\beta}}.
\]
Since $L(r)$ is a slowly varying function at infinity by \cite[Proposition 1.5.10]{BGT89},
$\lambda$ satisfies the full conditions in Theorems \ref{thm:limsup-whole} and \ref{thm:local-growth}.

In what follows, we take $l(z)=\beta$ in \eqref{eq:k-large} for simplicity;
which yields $\bar{\lambda}(r)=1/r^{\beta}$ for $r\ge 1$.
\begin{enumerate}
\item[(1)]
We first calculate the growth order of $\sup_{|x|\le r}X(t,x)$ as $r\rightarrow\infty$.
\begin{enumerate}
\item[(a)] Let  $\beta\ne \alpha/d$.
Then, by Lemma \ref{lem:tau}, $\overline{\tau}(r)\asymp 1/r^{\gamma}$ with $\gamma=\beta\wedge (\alpha/d)$, as $r\rightarrow\infty$. Therefore, by Theorem \ref{thm:limsup-whole}, for any nondecreasing function $f:(0,\infty)\to (0,\infty)$,
$$\limsup_{r\to\infty}\frac{\sup_{|x|\le r}X(t,x)}{f(r)}=\infty \quad \text {or} \quad\limsup_{r\to\infty}\frac{\sup_{|x|\le r}X(t,x)}{f(r)}=0$$ according to whether the integral $$\int_1^\infty r^{d-1}f(r)^{-\gamma}\,dr$$ diverges or converges.
In particular, when $f(r)=r^{d/\gamma}(\log r)^p$ for some $p>0$,
\[
\int_1^{\infty}r^{d-1}f(r)^{-\gamma}\,\d r<\infty \iff p>\frac{1}{\gamma}.
\]
Hence, Theorem \ref{thm:limsup-whole} implies the following:
\begin{itemize}
\item if $p>1/\gamma$, then
\begin{equation}\label{eq:tau-p-large}
\lim_{r\rightarrow\infty}\frac{\sup_{|x|\le r}X(t,x)}{r^{d/\gamma}(\log r)^p}
=0, \quad \text{$P$-a.s.;}
\end{equation}
\item if $0<p\le 1/\gamma$, then
\begin{equation}\label{eq:tau-p-small}
\limsup_{r\rightarrow\infty}
\frac{\sup_{|x|\le r}X(t,x)}{r^{d/\gamma}(\log r)^p}
=\infty, \quad \text{$P$-a.s.}
\end{equation}
\end{itemize}

\item[(b)] Suppose that $\beta=\alpha/d>d/(d+\alpha)$.
Then, by Lemma \ref{lem:tau} (ii),
\[
\overline{\tau}(r)\asymp \frac{\log r}{r^{\alpha/d}}, \quad
r\rightarrow\infty.
\]Therefore, by Theorem \ref{thm:limsup-whole}, for any nondecreasing function $f:(0,\infty)\to (0,\infty)$,
$$\limsup_{r\to\infty}\frac{\sup_{|x|\le r}X(t,x)}{f(r)}=\infty \quad \text {or} \quad\limsup_{r\to\infty}\frac{\sup_{|x|\le r}X(t,x)}{f(r)}=0$$ according to whether the integral $$\int_1^\infty r^{d-1}f(r)^{-\alpha/d}\log f(r)\,dr$$ diverges or converges.
In particular, for the test function $f(r)=r^{d^2/\alpha}(\log r)^p$ with $p>0$,
\[
\int_1^{\infty}r^{d-1}f(r)^{-\alpha/d}\log f(r)\,\d r<\infty \iff p>\frac{2d}{\alpha}.
\]
Thus, Theorem \ref{thm:limsup-whole} yields that  \begin{itemize}
\item if $p>2d/\alpha$, then \eqref{eq:tau-p-large} holds with $\gamma=\beta=\alpha/d$;
\item
if $0<p\le 2d/\alpha$, then \eqref{eq:tau-p-small} holds with $\gamma=\alpha/d$. \end{itemize}

\end{enumerate}

\item[(2)] We next calculate the growth order of $\sup_{x\in {\mathbb Z}^d, \, |x|\le r}X(t,x)$.
\begin{enumerate}
\item[(a)] Let $\beta\ne 1+\alpha/d$.
Then, by Lemma \ref{lem:eta-express-2} and Lemma \ref{lem:eta-eta_0},
$
\overline{\eta}(r)\asymp \overline{\eta_0}(r)\asymp {r^{-\delta}}
$ with $\delta=\beta\wedge (1+\alpha/d)$, as $r\to \infty$.
Therefore, by Theorem \ref{thm:local-growth},
for any nondecreasing function $f:(0,\infty)\to (0,\infty)$,
$$\limsup_{r\to\infty}\frac{\sup_{x\in \mathbb Z^d, |x|\le r}X(t,x)}{f(r)}=\infty
\quad \text {or} \quad\limsup_{r\to\infty}\frac{\sup_{x\in \mathbb Z^d, |x|\le r}X(t,x)}{f(r)}=0$$
according to whether the integral
$$\int_1^\infty r^{d-1}f(r)^{-\delta}\,dr$$
diverges or converges.
In particular, Theorem \ref{thm:local-growth} implies the following:
\begin{itemize}
\item if $p>1/\delta$, then
\begin{equation}\label{eq:eta-p-large}
\lim_{r\rightarrow\infty}
\frac{\sup_{x\in {\mathbb Z}^d, \, |x|\le r}X(t,x)}{r^{d/\delta}(\log r)^p}=0,
\quad \text{$P$-a.s.;}
\end{equation}
\item if $0<p\le 1/\delta$, then
\begin{equation}\label{eq:eta-p-small}
\limsup_{r\rightarrow\infty}
\frac{\sup_{x\in {\mathbb Z}^d, \, |x|\le r}X(t,x)}{r^{d/\delta}(\log r)^p}=\infty,
\quad \text{$P$-a.s.}
\end{equation}
\end{itemize}
\item[(b)] Let $\beta=1+\alpha/d$.
Then, by Lemma \ref{lem:eta-express-2} (ii) and Lemma \ref{lem:eta-eta_0},
\[
\overline{\eta}(r)\asymp \overline{\eta_0}(r)\asymp\frac{\log r}{r^{1+\alpha/d}},\quad r\to \infty.
\]
Therefore,  by Theorem \ref{thm:local-growth},
for any nondecreasing function $f:(0,\infty)\to (0,\infty)$,
$$\limsup_{r\to\infty}\frac{\sup_{x\in \mathbb Z^d, |x|\le r}X(t,x)}{f(r)}=\infty
\quad \text {or} \quad\limsup_{r\to\infty}\frac{\sup_{x\in \mathbb Z^d, |x|\le r}X(t,x)}{f(r)}=0$$
according to whether the integral $$\int_1^\infty r^{d-1}f(r)^{-(1+\alpha/d)}\log f(r)\,dr$$
diverges or converges.
In particular, by Theorem \ref{thm:local-growth},
\begin{itemize}
\item
if $p>2d/(d+\alpha)$, then
\eqref{eq:eta-p-large} holds with $\delta=\beta$;
\item
if $0<p\le 2d/(d+\alpha)$, then
\eqref{eq:eta-p-small} holds with $\delta=\beta$.\end{itemize}
\end{enumerate}
\end{enumerate}

\begin{rem}\rm
We explain the consequence of the assertions (1) and (2) above.
Recall that $\gamma=\beta\wedge (\alpha/d)$ and $\delta=\beta\wedge (1+\alpha/d)$.
 \begin{enumerate}
\item If $d/(d+\alpha)<\beta<\alpha/d$, then $\gamma=\delta=\beta$, and so,
 by
\eqref{eq:tau-p-large}--\eqref{eq:eta-p-small}, $\sup_{|x|\le r}X(t,x)$ has the same growth order
with that of $\sup_{x\in {\mathbb Z}^d, \, |x|\le r}X(t,x)$.

\item If $\beta>\alpha/d>d/(d+\alpha)$, then
$\gamma=\alpha/d<\delta$,
and $\sup_{|x|\le r}X(t,x)$ has the higher polynomial growth order
than that of  $\sup_{x\in {\mathbb Z}^d, \, |x|\le r}X(t,x)$.
\item Suppose that $d/(d+\alpha)<\alpha/d$ and $\beta=\alpha/d$.
Then $\gamma=\beta=\alpha/d$, and we have the following statements:

if $p>2d/\alpha$, then
\[
\lim_{r\rightarrow\infty}\frac{\sup_{|x|\le r}X(t,x)}{r^{d^2/\alpha}(\log r)^p}
=\lim_{r\rightarrow\infty}\frac{\sup_{x\in {\mathbb Z}^d, \, |x|\le r}X(t,x)}{r^{d^2/\alpha}(\log r)^p}
=0, \quad \text{$P$-a.s.}
\]

if $d/\alpha<p\le 2d/\alpha$, then
\[
\lim_{r\rightarrow\infty}\frac{\sup_{|x|\le r}X(t,x)}{r^{d^2/\alpha}(\log r)^p}=\infty,
\quad \lim_{r\rightarrow\infty}\frac{\sup_{x\in {\mathbb Z}^d, \, |x|\le r}X(t,x)}{r^{d^2/\alpha}(\log r)^p}
=0, \quad \text{$P$-a.s.}
\]

if $0<p\le d/\alpha$, then
\[
\limsup_{r\rightarrow\infty}\frac{\sup_{|x|\le r}X(t,x)}{r^{d^2/\alpha}(\log r)^p}
=\limsup_{r\rightarrow\infty}\frac{\sup_{x\in {\mathbb Z}^d, \, |x|\le r}X(t,x)}{r^{d^2/\alpha}(\log r)^p}
=\infty, \quad \text{$P$-a.s.}
\]
 \end{enumerate}
 \end{rem}

\section{Appendix}
\subsection{Heat kernels of symmetric stable-L\'evy process}
For $\alpha\in (0,2)$,
let $Z=(\{Z_t\}_{t\ge 0}, \{P_x\}_{x\in {\mathbb R}^d})$ be a (rotationally) symmetric stable-L\'evy process on ${\mathbb R}^d$
generated by $-(-\Delta)^{\alpha/2}$.
Then there exists a positive Borel measurable function
$
p_t(x)=p(t,x):(0,\infty)\times {\mathbb R}^d\to (0,\infty)$ such that
\[
p(t,-x)=p(t,x), \quad t>0, \ x\in {\mathbb R}^d
\]
and
\[
P_x(Z_t\in A)=\int_Ap(t,x-y)\,{\rm d}y, \quad t>0, \ x\in {\mathbb R}^d, \ A\in {\cal B}({\mathbb R}^d).
\]
In particular, there exists a positive, continuous and strictly decreasing function $g(r)$ on $[0,\infty)$ such that
\[
g(r)\asymp 1\wedge \frac{1}{r^{d+\alpha}}, \quad r>0
\]
and
\begin{equation}\label{eq:bound1}
p(t,x)=\frac{1}{t^{d/\alpha}}g\left(\frac{|x|}{t^{1/\alpha}}\right), \quad t>0, \ x\in {\mathbb R}^d;
\end{equation}
that is,
$$p(t,x)\asymp \frac{t}{|x|^{d+\alpha}}\wedge \frac{1}{t^{d/\alpha}}.$$
Then for any $c>0$,
\begin{equation}\label{eq:scale}
p(t,cx)=\frac{1}{c^d}p\left(\frac{t}{c^{\alpha}},x\right), \quad t>0, \ x\in {\mathbb R}^d.
\end{equation}

For $\gamma>0$ and $x,y\in \R^d$, let
\[
Q_{\gamma}(t,x,y)=\int_0^t\int_{\R^d}|p(s,x-z)-p(s,y-z)|^{\gamma}\,\d z\,\d s.
\]

\begin{lem}\label{lem:heat-kernel}
\begin{enumerate}
\item For any
$\gamma\in [1,1+\alpha/d)$,
there exists $c_1>0$ such that
for any unit vector $e\in {\mathbb R}^d$,
$$Q_\gamma(t,0,\pm e)\le c_1.$$
\item Suppose that $\alpha>d=1$ and $\gamma\ge 1$. Let $\gamma_0=(d+\alpha)/(d+1)$ and $T>0$.
Then, for any $t\in [0,T]$ and $x,y\in {\mathbb R}^d$ with $|x-y|\le 1$,
\[
Q_{\gamma}(t,x,y)
\preceq
\begin{dcases}
|x-y|^{\gamma}, & 1\le \gamma<\gamma_0,\\
|x-y|^{\gamma}\log(1+t/|x-y|^{\alpha}), & \gamma=\gamma_0,\\
|x-y|^{d(1-\gamma)+\alpha}, &
\gamma_0<\gamma<1+{\alpha}/{d}.
\end{dcases}
\]
\end{enumerate}
\end{lem}

\begin{proof}
(1)
By the translation invariance of the Lebesgue measure, for any $\gamma>0$,
there exists a positive constant $c_1:=c_1(\gamma)$ such that
for any unit vector $e\in {\mathbb R}^d$,
\begin{align*}
\int_0^1\int_{{\mathbb R}^d}
\left|p\left(s,z\pm e\right)-p(s,z)\right|^{\gamma}\,{\rm d}z{\rm d}s
&\le c_1\int_0^1\int_{{\mathbb R}^d}\left(p_s(z\pm e)^{\gamma}+p_s(z)^{\gamma}\right)\,{\rm d}z\,{\rm d}s\\
&=2c_1\int_0^1\int_{{\mathbb R}^d}p_s(z)^{\gamma}\,{\rm d}z\,{\rm d}s.
\end{align*}
On the other hand, according to \eqref{eq:bound1},
$$\int_0^1 \int_{{\mathbb R}^d}p_s(z)^{\gamma}\,{\rm d}z\,{\rm d}s
\le \int_0^1 \left(\frac{g(0)}{s^{d/\alpha}}\right)^{\gamma-1}
\int_{{\mathbb R}^d}p_s(z)
\,{\rm d}z\,{\rm d}s
=\int_0^1 \left(\frac{g(0)}{s^{d/\alpha}}\right)^{\gamma-1}\,\d s<\infty,$$
thanks to $1\le \gamma<1+{\alpha}/{d}$.
Thus, the first assertion (i) follows.

(2) We next prove (ii) with $T=1$.
For any $0<t\le 1$ and
$\gamma>0$,
\begin{equation}\label{eq:unit}
\begin{split}
Q_{\gamma}(t,x,y)
&=\int_0^t\int_{{\mathbb R}^d}|p(s,x-y+z)-p(s,z)|^{\gamma}\,{\rm d}z\,{\rm d}s\\
&=|x-y|^d\int_0^t\int_{{\mathbb R}^d}
\left|p\left(s,|x-y|\left(\frac{x-y}{|x-y|}+z\right)\right)-p(s,|x-y|z)\right|^{\gamma}\,{\rm d}z\,{\rm d}s.
\end{split}
\end{equation}
Then, by \eqref{eq:scale},
\begin{equation*}
\begin{split}
&\int_0^t\int_{{\mathbb R}^d}
\left|p\left(s,|x-y|\left(\frac{x-y}{|x-y|}+z\right)\right)-p(s,|x-y|z)\right|^{\gamma}\,{\rm d}z\,{\rm d}s\\
&=\frac{1}{|x-y|^{d\gamma}}\int_0^t
\int_{{\mathbb R}^d}
\left|p\left(\frac{s}{|x-y|^{\alpha}},\frac{x-y}{|x-y|}+z\right)
-p\left(\frac{s}{|x-y|^{\alpha}},z\right)\right|^{\gamma}\,{\rm d}z\,{\rm d}s\\
&=|x-y|^{\alpha-d\gamma}\int_0^{t/|x-y|^{\alpha}}
\int_{{\mathbb R}^d}
\left|p\left(s,\frac{x-y}{|x-y|}+z\right)
-p\left(s,z\right)\right|^{\gamma}\,{\rm d}z\,{\rm d}s.
\end{split}
\end{equation*}
Hence if we let $e_{xy}=(x-y)/|x-y|$, then by \eqref{eq:unit},
$$
Q_{\gamma}(t,x,y)
=|x-y|^{d(1-\gamma)+\alpha}
\int_0^{t/|x-y|^{\alpha}}
\int_{{\mathbb R}^d}
\left|p\left(s,z+e_{xy}\right)
-p\left(s,z\right)\right|^{\gamma}\,{\rm d}z\,{\rm d}s.
$$

Let $\gamma\ge 1$.
We first suppose  that $|x-y|\le t^{1/\alpha}$.
Since
\[
p_s(z+e_{xy})-p_s(z)=\int_0^1\langle \nabla p_s(z+ue_{xy}),e_{xy}\rangle\,{\rm d}u,
\]
we have
\[
|p_s(z+e_{xy})-p_s(z)|^{\gamma}
\le\left(\int_0^1|\nabla p_s(z+ue_{xy})|\,{\rm d}u\right)^{\gamma}
\le \int_0^1|\nabla p_s(z+ue_{xy})|^{\gamma}\,{\rm d}u.
\]
Then by the Fubini theorem,
\begin{align*}
\int_1^{t/|x-y|^{\alpha}}
\int_{{\mathbb R}^d}|p_s(z+e_{xy})-p_s(z)|^{\gamma}\,{\rm d}z\,{\rm d}s
&\le \int_1^{t/|x-y|^{\alpha}}
\left(\int_{{\mathbb R}^d}
\int_0^1|\nabla p_s(z+ue_{xy})|^{\gamma}\,{\rm d}u\,{\rm d}z\right)\,{\rm d}s\\
&=\int_1^{t/|x-y|^{\alpha}}
\left(\int_0^1
\int_{{\mathbb R}^d}|\nabla p_s(z+ue_{xy})|^{\gamma}\,{\rm d}z\,{\rm d}u\right)\,{\rm d}s\\
&=\int_1^{t/|x-y|^{\alpha}}\left(\int_0^1
\int_{{\mathbb R}^d}|\nabla p_s(z)|^{\gamma}\,{\rm d}z\,{\rm d}u\right)\,{\rm d}s\\
&=\int_1^{t/|x-y|^{\alpha}}
\int_{{\mathbb R}^d}|\nabla p_s(z)|^{\gamma}\,{\rm d}z{\rm d}s.
\end{align*}
Furthermore,
it follows from  \cite[Lemma 5]{BJ07} that
\[
|\nabla p_s(z)|\asymp |z|\left(\frac{s}{|z|^{d+2+\alpha}}\wedge \frac{1}{s^{(d+2)/\alpha}}\right),
\]
and
so we get
\begin{align*}
&\int_1^{t/|x-y|^{\alpha}}
\int_{{\mathbb R}^d}|\nabla p_s(z)|^{\gamma}\,{\rm d}z\,{\rm d}s
\asymp
\int_1^{t/|x-y|^{\alpha}}
\int_{{\mathbb R}^d} |z|^{\gamma}
\left(\frac{s}{|z|^{d+2+\alpha}}\wedge \frac{1}{s^{(d+2)/\alpha}}\right)^{\gamma}\,{\rm d}z\,{\rm d}s\\
&=
\int_1^{t/|x-y|^{\alpha}}\frac{1}{s^{(d+2)\gamma/\alpha}}
\left(\int_{|z|\le s^{1/\alpha}} |z|^{\gamma}\,{\rm d}z\right)\,{\rm d}s
+\int_1^{t/|x-y|^{\alpha}}s^{\gamma}
\int_{|z|>s^{1/\alpha}}
\left(\frac{|z|^{\gamma}}{|z|^{(d+2+\alpha)\gamma}}\right)\,{\rm d}z\,{\rm d}s\\
&=J(t,x,y).
\end{align*}
For any $t>0$ and $x,y\in {\mathbb R}^d$ with $|x-y|\le
t^{1/\alpha}$,
\[J(t,x,y)\asymp
\begin{dcases}
(t/|x-y|^{\alpha})^{1+(d-\gamma(d+1))/\alpha}, & 1\le \gamma<\gamma_0,\\
\log(1+t/|x-y|^{\alpha}), & \gamma=\gamma_0,\\
1, & \gamma>\gamma_0.
\end{dcases}
\]
Therefore,
according to all the conclusions above and the first assertion (i),
for any $t>0$ and $x,y\in {\mathbb R}^d$ with $|x-y|\le t^{1/\alpha}$,
\begin{align*}
&|x-y|^{d(1-\gamma)+\alpha}
\int_0^{t/|x-y|^{\alpha}}\int_{{\mathbb R}^d}|p_s(z+e_{xy})-p_s(z)|^{\gamma}\,{\rm d}z{\rm d}s\\
&\le |x-y|^{d(1-\gamma)+\alpha}(J(t,x,y)+c_0)
\preceq
\begin{dcases}
|x-y|^{\gamma}, & 1\le \gamma<\gamma_0,\\
|x-y|^{\gamma}\log(1+ t/|x-y|^{\alpha}), & \gamma=\gamma_0,\\
|x-y|^{d(1-\gamma)+\alpha},
&\gamma_0<\gamma<1+\alpha/d.
\end{dcases}
\end{align*}

We next suppose that $t^{1/\alpha}<|x-y|\le 1$.
Then, according the first assertion (i),
\begin{align*}
&|x-y|^{d(1-\gamma)+\alpha}
\int_0^{t/|x-y|^{\alpha}}
\int_{{\mathbb R}^d}
\left|p\left(s,z+e_{xy}\right)
-p\left(s,z\right)\right|^{\gamma}\,{\rm d}z\,{\rm d}s\\
&\le |x-y|^{d(1-\gamma)+\alpha}
\int_0^1\int_{{\mathbb R}^d}
|p(s,z+e_{xy})-p(s,z)|^{\gamma}
\,{\rm d}z{\rm d}s\\
&\preceq |x-y|^{d(1-\gamma)+\alpha}
\preceq
\begin{dcases}
|x-y|^{\gamma}, & 1\le
\gamma<\gamma_0,\\
|x-y|^{\gamma}\log(1+t/|x-y|^{\alpha}), & \gamma=\gamma_0,\\
|x-y|^{d(1-\gamma)+\alpha}, & \gamma_0<\gamma<1+\alpha/d.
\end{dcases}
\end{align*}
The proof is complete.
\end{proof}

\subsection{Poissonian functional associated with $\tau$}\label{section:tau}
In this subsection, we introduce a functional of the Poisson random measure
associated with the measure $\tau$
defined by \eqref{eq:def-tau}.
Let $A\subset\R^d$ be a
bounded
Borel set with $0<|\bar A|<\infty$, and define
\begin{equation}\label{eq:X_A}
X_A(t):=
\sum_{i=1}^{\infty}\frac{\zeta_i}{(t-\tau_i)^{d/\alpha}}{\bf 1}_{\{\eta_i\in \overline{A}, \, \tau_i\le t\}}.
\end{equation}
Clearly, $X_A(t)$ is a functional of the
Poisson random measure.

We first consider the existence of $X_A(t)$.

\begin{prop}\label{thm:xa-conv}
Let $A\subset\R^d$ be a  Borel set with $0<|\bar A|<\infty$.
Then, for any $t>0$, $X_A(t)$ is convergent P-a.s.
if and only if
\begin{equation}\label{eq:xa-conv-1-}
\int_{(0,1]}z^{(1\wedge (\alpha/d))}|\log z|^{{\bf 1}_{\{d=\alpha\}}}\,\lambda(\d z)<\infty.
\end{equation}
In this case, for any $\theta\in \R$,
\[
E[e^{i\theta X_A(t)}]=\exp\left(|\overline{A}|\int_{(0,\infty)}(e^{i\theta u}-1)\,\tau({\rm d}u)\right).
\]
\end{prop}

\begin{proof}
By \cite[p.43, Theorem 2.7 (i)]{K14},
$X_A(t)$ is convergent a.s.\ if and only if
\begin{align*}
\int_{(0,t]\times \R^d \times (0,\infty)}
\left(1\wedge \frac{z}{(t-s)^{d/\alpha}}{\bf 1}_{\{y\in \overline{A}\}}\right)\,
\d s\, \d y\,\lambda(\d z)
&=|\overline{A}|\int_{(0,\infty)}
\left\{\int_0^t\left(1\wedge \frac{z}{s^{d/\alpha}}\right)\,\d s\right\}\,\lambda(\d z)\\
&<\infty.
\end{align*}
Since
\begin{align*}
\int_0^t\left(1\wedge \frac{z}{s^{d/\alpha}}\right)\,\d s
&=\int_0^{t\wedge z^{\alpha/d}}\,\d s
+\int_{t\wedge z^{\alpha/d}}^t \frac{z}{s^{d/\alpha}}\,\d s\\
&=
(t\wedge z^{\alpha/d})
+
\begin{dcases}
\frac{z}{d/\alpha-1}\left((t\wedge z^{\alpha/d})^{1-d/\alpha}-t^{1-d/\alpha}\right),
& d\ne \alpha,\\
z\left(\log t-\log(t\wedge z^{\alpha/d})\right),
& d=\alpha,
\end{dcases}
\end{align*}
we have
\begin{align*}
&\int_{(0.\infty)}
\left\{\int_0^t\left(1\wedge \frac{z}{s^{d/\alpha}}\right)\,\d s\right\}\,\lambda(\d z)\\
&=\int_{(0,t^{d/\alpha}]}z^{\alpha/d}\,\lambda(\d z)
+t\overline{\lambda}(t^{d/\alpha})
+\begin{dcases}
\frac{\alpha}{d-\alpha}\int_{(0,t^{d/\alpha}]}
\left(z^{\alpha/d-1}-t^{1-d/\alpha}\right)z\,\lambda(\d z),
& d\ne \alpha,\\
\int_{(0,t^{d/\alpha}]}\left(\log t-\log(z^{\alpha/d}) \right)z\,\lambda(\d z),
& d=\alpha.
\end{dcases}
\end{align*}
We thus arrive at the first assertion.

Furthermore, it follows from
\cite[p.\ 43, Theorem 2.7 (i)]{K14}
 that, for any $\theta\in \R$,
\begin{align*}
E[e^{i\theta X_A(t)}]
&=\exp\left(\int_{(0,t]\times \R^d\times (0,\infty)}
\left(\exp\left(i\theta\frac{z}{(t-s)^{d/\alpha}}
{\bf 1}_{\{y\in \overline{A}\}}\right)-1\right)
\,\d s\,\d y\,\lambda(\d z)\right)\\
&=\exp\left(|\overline{A}|\int_{(0,t]\times (0,\infty)}
\left(\exp\left(i\theta\frac{z}{s^{d/\alpha}}\right)-1\right)\,\d s\,\lambda(\d z)\right)\\
&=\exp\left(|\overline{A}|\int_{(0,\infty)}(e^{i\theta u}-1)\,\tau({\rm d}u)\right).
\end{align*}
The proof is complete.
\end{proof}

\begin{rem}\rm Similarly to
$X_A(t)$ in \eqref{eq:X_A},
Proposition \ref{thm:xa-conv}, for any Borel set $A\subset \R^d$
with $0<|\bar A|<\infty$,
we can define
\begin{equation*}
X_A^*(t):=
\sum_{i=1}^{\infty}\frac{\zeta_i}{(t-\tau_i)^{d/\alpha}}
{\bf 1}_{\{\zeta_i/(t-\tau_i)^{d/\alpha}>1, \, \eta_i\in \overline{A}, \, \tau_i\le t\}}.
\end{equation*}
Then, following the proof of Proposition  \ref{thm:xa-conv},
we see that $X_A^*(t)$ is convergent a.s.\ for any $t>0$ if and only if
\begin{equation}\label{eq:xa-conv-1a}
\int_{(0,1]}z^{\alpha/d}\,\lambda(\d z)<\infty.
\end{equation} In this case, for any $\theta\in \R$, \[
E[e^{i\theta X_A^*(t)}]
=\exp\left(|\overline{A}|\int_{(0,\infty)}(e^{i\theta u}-1)\,\tau^*({\rm d}u)\right),
\]where $\tau^*(B)=\tau(B\cap (1,\infty))$ for
$B\in {\cal B}((0,\infty))$;
that is,
\[
\tau^*(B)=(m\otimes \lambda)\left(\left\{(s,z)\in (0,t]\times (0,\infty) :
z/s^{d/\alpha}\in B\cap (1,\infty)\right\}\right),
\quad B\in {\cal B}((0,\infty)).
\]
In particular, by definition,
\[
X_A(t)=X_A^*(t)+\sum_{i=1}^{\infty}\frac{\zeta_i}{(t-\tau_i)^{d/\alpha}}
{\bf 1}_{\{\zeta_i/(t-\tau_i)^{d/\alpha}\le 1, \, \eta_i\in \overline{A}, \, \tau_i\le t\}}.
\]
Roughly speaking,
for $\alpha>d$,
since \eqref{eq:xa-conv-1a} is weaker than \eqref{eq:xa-conv-1-},
the second term in the right hand side above dominates $X_A^*(t)$.
We also mention that for $\alpha=2$ and $d=1$,
$X_A^*(t)$ is the same with $X_A(t)$ in \cite{CK22}.
\end{rem}

For a Borel
set $A\subset \R^d$,
let
\[
\overline{X}_A(t)=\sup{\left\{(t-\tau_i)^{-d/\alpha}\zeta_i
: i\ge 1,
\tau_i\le t, \, \eta_i\in \overline{A}\right\}}
\]
and
\[
T_A(r)=\left\{(s,y,z)\in(0,t]\times \overline{A}
\times (0,\infty) :
(t-s)^{-d/\alpha}z>r
\right\}.
\]
Then,  for any $r>1$,
$\overline{X}_A(t)\le r$ if and only if
$\mu(T_A(r))=0$.
Since $\nu(T_A(r))=|\overline{A}|\overline{\tau}(r)$,
we obtain for all $r>1$,
\begin{equation}\label{eq:xa-poisson}
P(\overline{X}_A(t)>r)=1-P(\overline{X}_A(t)\le r)=1-P(\mu(T_A(r))=0)
=1-e^{-|\overline{A}|\overline{\tau}(r)}.
\end{equation}
In particular, we have
\begin{prop}\label{prop:tau-tail}
Let $A\subset \R^d$ be a Borel
set
with $0<|\overline{A}|<\infty$.
\begin{enumerate}
\item[{\rm (i)}]
If \eqref{eq:xa-conv-1-} holds, then for
\[
P(X_A(t)>r)\sim P(\overline{X}_A(t)>r)\sim |\overline{A}|\overline{\tau}(r),
\quad r\rightarrow\infty.
\]
\item[{\rm (ii)}]
If \eqref{eq:xa-conv-1a} holds, then for
\[
P(X_A^*(t)>r)\sim P(\overline{X}_A(t)>r)\sim |\overline{A}|\overline{\tau}(r),
\quad r\rightarrow\infty.
\]
\end{enumerate}
\end{prop}

We omit the proof of
Proposition \ref{prop:tau-tail}
because it is similar to that of Theorem \ref{thm:eta-tail}.

\begin{rem}\rm  If \eqref{eq:xa-conv-1a} fails (i.e., if $\int_{(0,1]}z^{\alpha/d}\,\lambda(\d z)=\infty$),
then
$\overline{\tau}(r)=\infty$ for any $r>0$ by Lemma \ref{lem:tau-finite}.
Therefore,  by \eqref{eq:xa-poisson},
we have
for any Borel set $A\subset \R^d$ with
$0<|\overline{A}|<\infty$,
$$P(\overline{X_A}(t)=\infty)=1.$$
Namely, if we take $A=B(x,r)$ for $x\in \R^d$ and $r>0$,
then for any
$M>0$
large enough,
there exists a Poisson point $(\tau,\eta,\zeta)\in (0,t]\times \R^d\times (0,\infty)$
associated with $\mu$ such that
$\eta\in B(x,r)$ and
$(t-\tau)^{-d/\alpha}\zeta>
M/g(0)$.
This in particular (see \eqref{eq:x-decomposition} above
that holds for all $d\ge1$) implies that
\[
\sup_{y\in B(x,r)}X(t,y)\ge p_{t-\tau}(0)\zeta
=g(0)\frac{\zeta}{(t-\tau)^{d/\alpha}}\ge
M
\]
and thus
\[
\sup_{y\in B(x,r)}X(t,y)=\infty, \quad \text{$P$-a.s.}\]
\end{rem}

\subsection{Multiplicative noise of bounded nonlinearity}\label{subsect:nonlinear}
In this subsection, we make a comment on the validity of  Theorems \ref{thm:limsup-whole} and \ref{thm:local-growth}
to the mild solution of \eqref{e:levy1}.
Let $\sigma$ be a Lipschitz continuous function on $[0,\infty)$ such that
for some positive constants $k_1$ and $k_2$ with $k_1<k_2$,
\begin{equation}\label{eq:sigma}
k_1\le \sigma(x)\le k_2, \quad x\in [0,\infty).
\end{equation}
If $\int_{(0,\infty)}z\,\lambda(\d z)<\infty$,
then, by \cite[Th\'eor\`eme 1.2.1]{S98} and \cite[Subsection 2.2]{CK20},
there exists a unique predictable process $Y(t,x)$ such that
\[
Y(t,x)=\int_{(0,t]\times \R^d}p_{t-s}(x-y)\sigma(Y(s,y))\,\Lambda(\d s\, \d y),
\quad (t,x)\in
(0,\infty)  
\times \R^d,
\]
which is a mild solution to \eqref{e:levy1}.
Then, by \eqref{eq:sigma},
\[
k_1X(t,x) \le Y(t,x)\le k_2X(t,x), \quad (t,x)\in
(0,\infty)  
\times \R^d, \ \text{$P$-a.s.},
\]
where $X(t,x)$ is a mild solution to \eqref{eq:fractional-she1}.
Hence, Theorems \ref{thm:limsup-whole} and \ref{thm:local-growth} remain valid for $Y(t,x)$.

\bigskip

\noindent {\bf Acknowledgements.}\,\,
The research of Yuichi Shiozawa is supported by JSPS KAKENHI Grant Numbers JP22K18675, JP23H01076 and JP23K25773.
The research of Jian Wang is supported by the National Key R\&D Program of China (2022YFA1006003)
and  the National Natural Science Foundation of China (Nos.\ 12071076 and 12225104).


\begin{thebibliography}{99}
\bibitem{BC}
Q. Berger, C. Chong and H. Lacoin:
The stochastic heat equation with multiplicative L\'evy noise: existence, moments, and intermittency,
{\it Commun.\ Math.\ Phys.} \textbf{402} (2023), 2215--2299.

\bibitem{BL22}
Q.~Berger and H.~Lacoin:
The continuum directed polymer in L\'evy noise,
{\it J. \'Ec.\ polytech.\ Math.}  \textbf{9} (2022), 213--280.

\bibitem{BGT89}
N.~H.~Bingham, C.~M.~Goldie and J.~L.~Teugels:
{\it Regular Variation}, Cambridge University Press, Cambridge, 1989.

\bibitem{Bi}
T.~T.~Binh, N.~H.~Tuan and T.~B.~Ngoc:
H\"{o}lder continuity of mild solutions of space-time fractional stochastic heat equation driven by colored noise,
{\it Eur.\ Phys.\ J. Plus} \textbf{136} (2021), Paper no. 935.

\bibitem{BG60}
R.~M.~Blumenthal and R.~K.~Getoor:
Some theorems on stable processes,
{\it Trans.\ Amer.\ Math.\ Soc.} \textbf{95} (1960), 263--273.




\bibitem{BJ07}
K.~Bogdan and T.~Jakubowski:
Estimates of heat kernel of fractional Laplacian perturbed by gradient operators,
{\it Comm.\ Math.\ Phys.} \textbf{271} (2007), 179--198.

\bibitem{Chong}
C.~Chong:
Stochastic PDEs with heavy-tailed noise,
{\it Stochastic Process.\ Appl.} \textbf{127} (2017),  2262--2280.


\bibitem{Chong1}
C. Chong:
L\'evy-driven Volterra equations in space and time,
{\it J. Theoret.\ Probab.} \textbf{30} (2017), 1014--1058.



\bibitem{Chong1-D}
C. Chong, R.C. Dalang and T. Humeau:
Path properties of the solution to the stochastic heat equation with L\'evy noise,
{\it Stoch. Partial Differ. Equ. Anal. Comput.} \textbf{7} (2019), 123--168.


\bibitem{CK20}
C.~Chong and P.~Kevei:
The almost-sure asymptotic behavior of the solution to the stochastic heat equation with L\'evy noise,
{\it Ann.\ Probab.} \textbf{48}, (2020) 1466--1494.


\bibitem{CK22}
C.~Chong and P.~Kevei:
Extremes of the stochastic heat equation with additive L\'evy noise,
{\it Electron.\ J. Probab.} \textbf{27} (2022), Paper No.\ 128, 21 pp.

\bibitem{CK23}
C.~Chong and P.~Kevei:
The landscape of peaks: the intermittency islands of  the stochastic heat equation with L\'evy noise,
{\it Ann. Probab.} \textbf{51} (2023), 1449--1501.


\bibitem{DK}
R. Dalang, D. Khoshnevisan, C. Mueller, D. Nualart and Y. Xiao:
{\it A Minicourse on Stochastic Partial Differential Equations}, Lecture Notes
in Math., vol.\ {\bf 1962}, Springer-Verlag, Berlin, 2009.

\bibitem{FK}
M. Foondun and D. Khoshnevisan:
On the stochastic heat equation with spatially-colored random forcing,
{\it Trans.\ Amer.\ Math.\ Soc.} \textbf{365} (2013), 409--458.

\bibitem{FL}
M. Foondun, W. Liu and M. Omaba:
Moment bounds for a class of fractional stochastic heat equations,
{\it Ann.\ Probab.} \textbf{45} (2017), 2131--2153.

\bibitem{FN}
M. Foondun and E. Nane:
Asymptotic properties of some space-time fractional stochastic equations,
{\it Math.\ Z.} \textbf{287} (2017), 493--519.

\bibitem{K09}
D.~Khoshnevisan:
A primer on stochastic partial differential equations, in:
{\it Lecture Notes in Math.}, vol. \textbf{1962},
Springer, Berlin, 2009, pp. 1--38.

\bibitem{K0}
D.~Khoshnevisan:
{\it Analysis of Stochastic Partial Differential Equations},
CBMS Regional Conference Series in Mathematics, vol. {\bf 119}, 
the American Mathematical
Society, Providence, RI, 2014.

\bibitem{KKX17}
D.~Khoshnevisan, K.~Kim and Y.~Xiao:
Intermittency and multifractality: a case study via parabolic stochastic PDEs,
{\it Ann.\ Probab.} \textbf{45} (2017), 3697--3751.



\bibitem{Kim}
K. Kim:
On the large-scale structure of the tall peaks for stochastic heat equations with fractional Laplacian,
{\it Stochastic Process.\ Appl.} \textbf{129} (2019), 2207--2227.


\bibitem{KKW}
P. Kim, T. Kumagai and J. Wang:
Laws of the iterated logarithm for symmetric jump processes,
{\it Bernoulli} \textbf{23} (2017),  2330--2379.





\bibitem{K14}
A.~E.~Kyprianou:
{\it Fluctuations of L\'evy Processes with Applications},
Second Edition,
Springer, Heidelberg, 2014.



\bibitem{MR14}
C.~Marinelli and M.~R\"ockner:
On maximal inequalities for purely discontinuous martingales in infinite dimensions, in:
{\it Lecture Notes in Math.}, vol. \textbf{2123}, Springer, Cham, 2014, pp. 293--315.

\bibitem{M98}
C. Mueller:
The heat equation with L\'evy noise,
{\it Stochastic Process.\ Appl.} \textbf{74} (1998), 67--82.

\bibitem{P04}
A.~G.~Pakes:
Convolution equivalence and infinite divisibility,
{\it J. Appl.\ Probab.} \textbf{41} (2004), 407--424.

\bibitem{P07}
A.~G.~Pakes:
Convolution equivalence and infinite divisibility: corrections and corollaries,
{\it J. Appl.\ Probab.} \textbf{44} (2007), 295--305.


\bibitem{RR89}
B.~S.~Rajput and J.~Rosinski:
Spectral representations of infinitely divisible processes,
{\it Probab.\ Theory Related Fields} \textbf{82} (1989), 451--487.

\bibitem{RS93}
J.~Rosinski and G.~Samorodnitsky:
Distributions of subadditive functionals of sample paths of infinitely divisible processes,
{\it Ann.\ Probab.} \textbf{21} (1993), 996--1014.


\bibitem{S98}
E.~Saint Loubert Bi\'e:
\'Etude d'une EDPS conduite par un bruit poissonnien,
{\it Probab.\ Theory Related Fields} \textbf{111} (1998), 287--321.




\bibitem{Wa}
J.B. Walsh:
An introduction to stochastic partial differential equations, in: P.L. Hennequin (Ed.), \'{E}cole d'\'{E}t\'e de
Probabilit\'es de Saint Flour XIV-1984, Springer, Berlin, 1986, pp. 265--439.


\end{thebibliography}
\end{document}